\documentclass[a4paper,reqno]{amsart}
\usepackage[footnotesize, bf]{caption}
\usepackage{a4wide}
\title{}
\usepackage{setspace}
\usepackage[totalwidth=14cm,totalheight=22cm]{geometry}
\usepackage[T1]{fontenc}
\usepackage{tikz}
\usetikzlibrary{shapes,calc,cd,decorations.pathmorphing,hobby}

\usepackage{caption}
\usepackage{subfig}
\usepackage{color}
\usepackage{tikz-cd}
\usepackage{parcolumns} 
\usepackage{faktor}
\usepackage{marvosym}
\usepackage{listings}
\usepackage{booktabs}
\usepackage{adjustbox}
\usepackage{tabularx,ragged2e,booktabs}
\usepackage{hhline}
\usepackage{enumitem}
\usepackage{amsbsy}
\usepackage{amsfonts}
\usepackage{amsmath}
\usepackage{amsthm}
\usepackage{amssymb}
\usepackage{mathrsfs}
\usepackage{sidecap}
\usepackage{graphicx}
\usepackage{nicefrac}
\usepackage{wasysym}
\usepackage{relsize}
\usepackage{mathtools}
\usepackage{bbm}
\usepackage{braket}
\usepackage[english]{babel}
\usepackage{lmodern}
\usepackage[T1]{fontenc}
\usepackage{verbatim}
\usepackage{longtable}
\usepackage{microtype}
\usepackage{dsfont}
\usepackage{graphicx}
\usepackage{epstopdf}
\usepackage{fancyhdr}
\usepackage{tikz}
\usepackage{tikz-cd}
\usepackage{mathtools}
\usepackage{lipsum}
\usepackage[utf8]{inputenc}
\usepackage[T1]{fontenc}
\usepackage{lmodern} % load a font with all the characters 
\usepackage{xparse}
\usepackage{bm}
\usepackage{booktabs}
\usepackage{multirow}
\usepackage{imakeidx}

\usepackage{algorithm}
\usepackage[noend]{algpseudocode}

\usepackage{hyperref}
\hypersetup{linktocpage}

\newcommand{\numberset}{\mathbb}

\newcommand{\Z}{\numberset{Z}}
\newcommand{\R}{\numberset{R}}
\newcommand{\C}{\numberset{C}}

\newcommand{\Tau}{\text{\Capricorn}}

\newcommand{\GG}{ {\mathbb{G}}}

\newcommand{\JJ}{\mathds{J}}

\newcommand{\Hy}{\mathds{H}}

\newcommand{\eps}{\varepsilon}
\newcommand{\x}{\underline x}

\theoremstyle{definition} 
\newtheorem{theoremintro}{Theorem}

\newtheorem{Definition}{Definition}[section]
\newtheorem*{Theorem*}{Theorem}

\newtheorem{Example}{Example}

\newtheorem{Remark}[Definition]{Remark}
\newtheorem{Proposition}[Definition]{Proposition}
\newtheorem{Prop}[Definition]{Proposition}
\newtheorem{Lemma}[Definition]{Lemma}
\newtheorem{Theorem}[Definition]{Theorem}
\newtheorem{Teo}[Definition]{Theorem}
\newtheorem{Cor}[Definition]{Corollary}
\theoremstyle{plain}

\newcommand{\PSL}{PSL(2,\C)}
\newcommand{\CP}{\mathbb{C}\mathds{P}}

\newcommand{\SL}{SL(2,\C)}
\newcommand{\asl}{\mathfrak{sl}(2,\C)}

\newcommand{\NN}{\mathcal N}
\newcommand{\io}{\sigma}
\newcommand{\upsi}{\underline \Psi}

\newcommand{\inner}[1] {\langle #1 \rangle}

\newcommand{\Proj}{\mathds{P}}
\newcommand{\GGG}{{\mathbb G}}
\newcommand{\II}{\underline {\mathrm{II}}}
\newcommand{\IIs}{\mathrm{II}}

\newcommand{\CTS}{\C TS}
\newcommand{\XX}{\mathtt{X}}

\newcommand{\CTM}{\C TM}
\newcommand{\YY}{\mathtt{Y}}
\newcommand{\ZZ}{\mathtt{Z}}
\newcommand{\utheta}{\underline \theta}

\newcommand{\SOnn}{SO(n+2,\C)}
\newcommand{\inners}{\inner{\cdot,\cdot}}

\newcommand{\Lieonn}{\mathfrak{o}(n+2,\C)}
\newcommand{\Lieg}{\mathfrak{g}}
\newcommand{\PML}{{\mathds{X}_{n+1}}}
\newcommand{\ML}{{\mathds{X}_{n+1}}}
\newcommand{\XXX}{\mathds X}
\newcommand{\PXX}{\mathds {PX}}
\makeindex

\title{On immersions of surfaces into $SL(2,\C)$ and geometric consequences}
\author{FRANCESCO BONSANTE AND CHRISTIAN EL EMAM}

\begin{document}
	\maketitle
	
	\tableofcontents
	
	\section{Introduction}
	It is a well known result in Riemannian and Pseudo-Riemannian Geometry that the theory of codimension-1 (admissible) immersions of a simply connected manifold $M$ into a space form is completely described by the study of two tensors on $M$, the \emph{first fundamental form} and the \emph{shape operator} of the immersion, corresponding respectively to the pull-back metric and to the derivative of a local normal vector field.
	
	 On the one hand, for a given immersion the first fundamental form and the shape operator turn out to satisfy two equations known as Gauss and Codazzi equations. On the other hand, given a metric and a self-adjoint (1,1)-tensor on $M$ satisfying these equations, there exists one admissibile immersion into the space form with that pull-back metric and that self-adjoint tensor respectively as first fundamental form and shape operator; moreover, such immersion is unique up to post-composition with ambient isometries of the space form. This class of results are often denoted as \emph{Gauss-Codazzi Theorems}, \emph{Bonnet Theorems} or as \emph{Fundamental Theorem of hypersurfaces}. See for instance \cite{Kobayashi-Nomizu 2},  \cite{Gauss-Codazzi in space forms}.
	 
	 In this article we approach the study of immersions of smooth manifolds into holomorphic Riemannian space forms of constant curvature $-1$. A proper definition of holomorphic Riemannian metrics and sectional curvature will be given later  in this paper (see Section 2), let us just mention that holomorphic Riemannian metrics are a natural analogue of Riemannian metrics in the complex setting and a good class of examples is given by complex semisimple Lie groups equipped with the Killing form (extended globally on the group via translations).
	 
	 The study of this kind of immersions turns out to have some interesting consequences, including some remarks concerning geometric structures on surfaces. In order to give a general picture:
	 \begin{itemize}
	 	\item it provides a formalism for the study of immersions of surfaces into $\SL$ and into the space of geodesics of $\Hy^3$, that we will denote by $\GG$;
	 	\item it generalizes the classical theory of immersions into non-zero curvature space forms, leading to a model to study the transitioning of hypersurfaces among $\Hy^n$, $AdS^n$, $dS^n$ and $S^n$;
	 	\item it furnishes a tool to construct holomorphic maps from complex manifolds to the character variety of $SO(n, \C)$ (including $\PSL\cong SO(3,\C)$), providing also an alternative proof for the holomorphicity of the complex landslide map (see \cite{cyclic});
	 	\item it introduces a notion of \emph{complex metrics} which extends Riemannian metrics and which turns out to be useful to describe the geometry of couples of projective structures with the same monodromy. We also show a uniformization theorem in this setting which is in a way equivalent to the classic Bers Theorem for quasi-Fuchsian representations.
	 \end{itemize}

	 \subsection{Holomorphic Riemannian manifolds} The notion of \emph{holomorphic Riemannian metrics on complex manifolds} can be seen as a natural analogue of Riemannian metrics in the complex setting: a holomorphic Riemannian metric on $\mathbb M$ is a holomorphic never degenerate section of $Sym_\C(T^*\mathbb M \otimes T^* \mathbb M)$, namely it consists pointwise of $\C$-bilinear inner products on the holomorphic tangent spaces whose variation on the manifold is holomorphic.
	  Such structures turn out to have several interesting geometric properties and have been largely studied (e.g. see the  works by LeBrun, Dumitrescu, Zeghib, Biswas as in \cite{holomorphic riemannian 1}, \cite{holomorphic riemannian 4} \cite{holomorphic riemannian 2}, \cite{holomorphic riemannian 3}). 
	 
	 In an attempt to provide a self-contained introduction to the aspects we will deal with, Section 2 starts with some basic
	 results on holomorphic Riemannian manifolds. After a short overview in the general setting - where we recall the notions of Levi-Civita connections (with the corresponding curvature tensors) and sectional curvature in this setting - we will focus on holomorphic Riemannian space forms, namely geodesically-complete simply-connected holomorphic Riemannian manifolds with constant sectional curvature.  We prove a theorem of existence and uniqueness of the holomorphic Riemannian space form of prescribed constant curvature and dimension and then focus on the ones of curvature $-1$, that we will denote as $\mathds X_n$. 
	 The space $\mathds X_n$ can be defined as 
	 \[\mathds X_n= \{(z_1, \dots, z_{n+1}) \in \C^{n+1}\ |\ \sum_{i=1}^{n+1} z_1^2 + \dots + z_{n+1}^2 =-1 \}
	 \] with the metric inherited as a complex submanifold on $\C^{n+1}$ equipped with the standard $\C$-bilinear inner product.
	 
	 This quadric model of $\mathds X_n$ may look familiar: it is definitely analogue to some models of $\Hy^n$, $AdS^n$, $dS^n$ and $S^n$. In fact, all the pseudo-Riemannian space forms of curvature $-1$ immerge isometrically in $\mathds X_n$: as a result, $\Hy^n$, $AdS^n$ embed isometrically while $dS^n$ and $S^n$ embed anti-isometrically, i.e. $-dS^n$ and $-S^n$ (namely, $dS^n$ and $S^n$ equipped with the opposite of their standard metric) embed isometrically.
	 
	 For $n=1,2,3$, $\mathds X_n$ turns out to be familiar also in another sense:
	 \begin{itemize}
	 	\item $\mathds X_1$ is isometric to $\C^*$ equipped with the metric given by the quadratic differential $\frac{dz^2}{z^2}$;
	 	\item $\mathds X_2$ is isometric to the space $\GG$ of oriented lines of $\Hy^3$, canonically identified with $\CP^1\times \CP^1\setminus \Delta$, equipped with the only $\PSL$-invariant holomorphic Riemannian metric of curvature $-1$;
	 	\item $\mathds X_3$ is isometric up to a scale to $\SL$ equipped with the Killing form. 
	 \end{itemize}
	 
	 \subsection{Immersions of manifolds into $\mathds X_n$} In Section 3, we approach the study of immersions of smooth manifolds into $\XXX_n$. The idea is to try to emulate the usual approach as in the pseudo-Riemannian case, but things turn out to be slightly more complicated.
	 
	 Given a smooth map $\sigma\colon M \to \XXX_n$, the pull-back of the holomorphic Riemannian metric is some $\C$-valued symmetric bilinear form, so one can extend it to a $\C$-bilinear inner product on $\CTM=TM \otimes \C$; it now makes sense to require it to be non-degenerate. 
	 This is the genesis of what we define as \emph{complex (valued) metrics} on smooth manifolds, namely smoothly varying non-degenerate inner products on each $\C T_x M$, $x\in M$. We will say that an immersion $\sigma\colon M \to \XXX_n$ is \emph{admissible} if the pull-back metric is a complex valued metric. By elementary linear algebra, $\sigma$ is admissible only if $dim(M)\le n$: the real codimension is $n$ and, despite it seems high, it cannot be lower than that. It therefore makes sense to redifine the $\emph{codimension}$ of 
	 $\sigma$ as $n-dim(M)$. In this paper we focus on immersions of codimension $1$ and $0$.
	 
	 It turns out that this condition of admissibility is the correct one in order to have extrinsic geometric objects that are analogue to the ones of the pseudo-Riemannian case. Complex metrics come with some complex analogue
	 of the Levi-Civita connection, which in turn allows to define a curvature tensor and a notion of sectional curvature. In codimension 1, admissible immersions come with a notion of local normal vector field (unique up to a sign) that allows to define a second fundamental form and a shape operator. 
	 
	 Section 3 ends with Theorem \ref{Teo Gauss-Codazzi} in which we deduce some analogue of Gauss and Codazzi equations.

	 Section 4 is mainly devoted to the proof of a Gauss-Codazzi theorem for immersions into $\mathds X_n$ in codimension 1, as stated in Theorem \ref{Teoremone}.
	 
For immersions of surfaces into $\mathds X_3\cong \SL$ Theorems \ref{Teo Gauss-Codazzi} and \ref{Teoremone} can be stated in the following way.
	 
	 	\begin{theoremintro}
	 	Let $S$ be a smooth simply connected surface. Consider a complex metric $g$ on $S$, with induced Levi-Civita connection $\nabla$ and curvature $K_g$, and a $g$-self adjoint bundle homomorphism $\Psi\colon \CTS \to \CTS$.
	 	
	 	The couple $(g,\Psi)$ satisfies
	 	\begin{align}
	 	1) &d^\nabla \Psi \equiv 0;\\
	 	2)	&K_g=-1+det(\Psi).
	 	\end{align}
	 	if and only if there exists an isometric immersion $\sigma \colon (S,g) \to \SL$ whose corresponding shape operator is $\Psi$.
	 	
	 	Such immersion $\sigma$ is unique up to post-composition with elements in $Isom_0(\SL)\cong \faktor{\SL \times \SL}{\{\pm(I_2,I_2)\}}$. 
	 \end{theoremintro}
 
	 The almost uniqueness of the immersion grants that if
	 $\Gamma$ is a group acting on $S$ preserving $g$ and $\Psi$, then the immersion $\sigma$ is $(\Gamma,Isom_0(\SL) )$-equivariant, namely there exists a representation \[mon_\sigma\colon \Gamma\to Isom_0(\SL)\] such that for all $\gamma\in \Gamma$
	 \[
	 mon_\sigma(\gamma)\circ \sigma = \sigma \circ\gamma. \]
	 
	 As a result if $S$ is not simply connected, solutions of Gauss-Codazzi equations correspond to $(\pi_1(S), Isom_0(\SL) )$- equivariant immersions of its universal cover into $\SL$.
	  \vspace{5mm}

\subsection{Immersions in codimension zero and into pseudo-Riemannian space forms}	 
	 The study of immersions into $\mathds X_n$ in codimension zero leads to an interesting result. 	 
	 \begin{theoremintro}
	 	\label{codimension 0 intro}
	 	Let $M=M^n$ be a smooth manifold, $g$ a complex metric on $M$.
	 	
	 	Then $g$ has constant sectional curvature $-1$ if and only if there exists an isometric immersion $(M,g)\to \mathds X_n$.
	 	
	 	Such immersion is unique up to post-composition with elements in $Isom(\mathds X_n)$.
	 \end{theoremintro}
 
	This result can be deduced by the Gauss-Codazzi Theorem: in fact immersions in $\mathds X_n$ of codimension $0$ correspond to  codimension-$1$ totally geodesic immersions in $\mathds X_{n+1}$, namely immersions with shape operator $\Psi=0$. A full proof is in section 4.
	
	As a result, every pseudo-Riemannian space form of constant curvature $-1$ and dimension $n$ admits an essentially unique  isometric immersion into $\mathds X_n$.
	
\vspace{5mm}
	
	In fact, the last remark and the similar description of Gauss-Codazzi equations for immersions into pseudo-Riemannian space forms lead to Theorem \ref{da X_n a space forms}: as regards the case of $\SL$, the result can be stated in the following way.
	
	\begin{theoremintro}
		Let $\sigma\colon S\to \SL$ be an admissible immersion with pull-back metric $g$ and shape operator $\Psi$.
		\begin{itemize}
		\item $\sigma(S)$ is contained in the image of an isometric embedding of $\Hy^3$ if and only if $g$ is Riemannian and $\Psi$ is real.
		\item 
		$\sigma(S)$ is contained in the image of an isometric embedding of $AdS^3$ if and only if either $g$ is Riemannian and $i\Psi$ is real, or if $g$ has signature $(1,1)$ and $\Psi$ is real.
		\item $\sigma(S)$ is contained in the image of an isometric embedding of $-dS^3$ if and only if either $g$ has signature $(1,1)$ and $i\Psi$ is real, or if $g$ is negative definite and $\Psi$ is real.
		\item $\sigma(S)$ is contained in the image of an isometric embedding of $-S^3$ if and only if $g$ is negative definite and $i\Psi$ is real. 
	\end{itemize}
	\end{theoremintro}

\subsection{Holomorphic dependence of the monodromy on the immersion data}	
	Given a smooth manifold $M$ of dimension $n$, we say that $(g, \Psi)$ is a couple of \emph{immersion data} for $M$ if there exists a $\pi_1(M)$-equivariant immersion $\widetilde M\to \mathds X_{n+1}$ with pull-back metric $\widetilde g$ and shape operator $\widetilde \Psi$. As a result of the essential uniqueness of the immersion, each immersion data comes with a monodromy map $mon_\sigma\colon \pi_1(M)\to Isom_0(\mathds X_{n+1})$.
	
	In section 5, we consider families of immersion data $\{(g_\lambda, \Psi_\lambda)\}_{\lambda \in \Lambda}$ for $M$. 
	
	Let $\Lambda$ be a complex manifold. We say that the family $\{(g_\lambda, \Psi_\lambda)\}_{\lambda \in \Lambda}$ is  \emph{holomorphic} if for all $x\in M$, the functions
	\begin{align*}
	\Lambda &\to Sym^2(\C T^*_xM)\\ 
	\lambda &\mapsto g_\lambda(x)
	\end{align*}
	and 
	\begin{align*}
	\Lambda &\to End_\C(\C T_xM)\\ 
	\lambda &\mapsto \Psi_\lambda(x)
	\end{align*}
	are holomorphic. 
	
	For a fixed hyperbolic Riemannian metric $h$ on a surface $S$, an instructive example is given by the family $\{(g_z,\psi_z)\}_{z\in \C}$ defined by
	\[
	\begin{cases}
	g_z= \cosh^2(z) h; \\
	\psi_z= \tanh(z) id
	\end{cases}
	\] 
	whose monodromy is going to be the monodromy of an immersion into $\Hy^3$ for $z\in \R$ and the monodromy of an immersion into $AdS^3$ for $z\in i\R$. 
	
	The main result of section 5 is the following.
	
	\begin{theoremintro}
		\label{intro teo olomorfia}
		Let $\Lambda$ be a complex manifold and $M$ be a smooth manifold of dimension $n$. 
		
		Let $\{(g_\lambda, \Psi_\lambda)\}_{\lambda\in \Lambda}$ be a holomorphic family of immersion data for $\pi_1(M)$-equivariant immersions $\widetilde M\to \mathds X_{n+1}$. Then there exists a smooth map
		\[
		\sigma\colon \Lambda \times \widetilde M \to \mathds X_{n+1}
		\]
		such that, for all $\lambda\in \Lambda$ and $x\in M$:
		\begin{itemize}
			\item $\sigma_\lambda:= \sigma(\lambda, \cdot)\colon \widetilde M \to \mathds X_{n+1}$ is an admissible immersion with immersion data $(g_\lambda, \Psi_\lambda)$;
			\item $\sigma(\cdot,x)\colon \Lambda \to \mathds X_{n+1}$ is holomorphic.
		\end{itemize}
		Moreover, the monodromy map
		\begin{align*}
		\Lambda &\to \mathcal X (\pi_1(M), SO(n+2, \C))\\
		\lambda &\mapsto mon(\sigma_\lambda)
		\end{align*}
		is holomorphic.
	\end{theoremintro}
	
	In section 5.2 we show an alternative proof of the holomorphicity of the complex landslide using \ref{intro teo olomorfia}.

\subsection{Uniformizing complex metrics and Bers Theorem}	

In Section 6 we focus on complex metrics on surfaces. 

Even in dimension 2, complex metrics can have a rather wild behaviour. Nevertheless, we point out a neighbourhood of the Riemannian locus whose elements have some nice features: we will call these elements \emph{positive complex metrics} (Definition \ref{Def positive metrics}).

	In Theorem \ref{Gauss Bonnet} we prove that the standard Gauss-Bonnet Theorem also holds for positive complex metrics on closed surfaces. 	
		
		The most relevant result in Section 6 is a version of the Uniformization Theorem for positive complex metrics. 
		
	\begin{theoremintro}
		Let $S$ be a surface with $\chi(S)<0$. 
		
		For any positive complex metric $g$ on $S$ there exists a smooth function $f\colon S\to \C^*$ such that the positive complex metric $f\cdot g$ has constant curvature $-1$ and has quasi-Fuchsian monodromy.
	\end{theoremintro}

	The proof of this result uses Bers Simultaneous Uniformization Theorem (Theorem \ref{Bers} in this paper) and in a sense is equivalent to it.

	Indeed, by Theorem \ref{codimension 0 intro}, complex metrics on $S$ with constant curvature $-1$ correspond to equivariant isometric immersions of $\widetilde S$ into $\GG=\CP^1\times \CP^1\setminus\Delta$: hence, they can be identified with some couple of maps $\widetilde S\to \CP^1$ with the same monodromy. In this sense, Bers Theorem provides a whole group of immersions into $\CP^1\times \CP^1\setminus\Delta$: such immersions correspond to complex metrics of curvature $-1$ which prove to be positive. The proof of the uniformization consists in showing that every complex positive metric is conformal to a metric constructed with this procedure.
	
\subsection*{Acknowledgements} 
We would like to thank Jean-Marc Schlenker, Gabriele Mondello and Andrea Seppi for several interesting discussions. The second author spent two months at Institut Fourier and one month at University of Luxembourg where he took the opportunity of deepening the theme of this project. He gratefully thanks the host institutions. The authors were partially supported by Blue Sky Research project "Analytic and geometric properties of low-dimensional manifolds".

	\section{Geometry of holomorphic Riemannian space forms}
	\subsection{Holomorphic Riemannian metrics}
	
	Let $\mathbb M$ be a complex analytic manifold, with complex structure $\JJ$, let $n=dim_\C \mathbb M$ and $T\mathbb M\to \mathbb M$ be the tangent bundle. 
	
	Local coordinates $(x_1, y_1, \dots, x_n, y_n)\colon U \to \R^{2n}\equiv \C^n$, $U\subset \mathbb M$, are \emph{holomorphic} if \[\JJ\bigg(\frac{\partial}{\partial x_k}\bigg)=\frac{\partial}{\partial y_k}\qquad
	\JJ\bigg(\frac{\partial}{\partial y_k}\bigg)= -\frac{\partial}{\partial x_k}.
	\]
	A function $f\colon U \to \C^N$ is \emph{holomorphic} if, in local holomorphic coordinates, it can be seen as a holomorphic function from an open subset of $\C^n$ to $\C^N$. Equivalently, $f$ is holomorphic if and only if $df\circ \JJ= i df$.
	
	A local vector field $X$ on $\mathbb M$ is \emph{holomorphic} if $dz_k (X)$ is a holomorphic function for all $k\in\{1, \dots, n\}$, where $dz_k=dx_k + i dy_k$.

	\begin{Definition}
		
		A \emph{holomorphic Riemannian metric} $\inner{\cdot, \cdot}$ on $\mathbb M$ is a symmetric $2$-form on $TM$, i.e. a section of $Sym^2 (T^*\mathbb M)$, such that: \begin{itemize}
			\item $\langle \cdot, \cdot \rangle$ is $\C$-bilinear, i.e. for all $X,Y \in T_p \mathbb M$ we have $\langle \JJ X, Y \rangle= \langle X, \JJ Y\rangle =i \langle X, Y\rangle$;
			\item non-degenerate as a complex bilinear form at each point;
			\item for all $X_1, X_2$ local holomorphic vector fields, $\inner{X_1, X_2}$ is a holomorphic function; equivalently, for all local coordinates $(x_1, y_1, \dots, x_n, y_n)$, the functions $
			\inner{\frac{\partial}{\partial x_k}, \frac{\partial}{\partial x_h}}$ (or equivalently the functions $ \inner{\frac{\partial}{\partial x_k}, \frac{\partial}{\partial y_h}}$ and $\inner{\frac{\partial}{\partial y_k}, \frac{\partial}{\partial y_h}}$)
			are all holomorphic. 
		\end{itemize}
		
		We also denote $\| X\|^2:= \inner{X,X}$.
	\end{Definition}
	
	Observe that, for a given holomorphic Riemannian metric $\inner{\cdot, \cdot}$, both the real part $Re\inner{\cdot, \cdot}$ and the imaginary part $Im\inner{\cdot, \cdot}$ are pseudo-Riemannian metrics on $\mathbb M$ with signature $(n,n)$. 
	
Taking inspiration from basic (Pseudo-)Riemannian Geometry, one can define several constructions associated to a holomorphic Riemannian metric, such as a Levi-Civita connection - leading to notions of a curvature tensor, (complex) geodesics and completeness - and sectional curvatures. We recall the basic notion, the reader may find a more detailed analysis in \cite{holomorphic riemannian 4}.

	There exists an analogous result to the Levi-Civita Theorem.
	
	\begin{Proposition} [See \cite{holomorphic riemannian 4}]
		Given a holomorphic Riemannian metric $\inner{\cdot, \cdot}$ on $\mathbb M$, there exists a unique connection $D$ over $T\mathbb M$, that we will call \emph{Levi-Civita connection}, such that for all $X,Y \in \Gamma(T \mathbb M)$ the following conditions stand:
		\begin{align}
		\label{compatibilita con metrica}
		d \inner{X, Y} &= \inner{D X, Y}+ \inner{X, D Y} \qquad && \text{($D$ is compatible with the metric)};\\
		\label{torsion free}
		[X,Y]&= D_X Y - D_Y X \qquad &&\text{($D$ is torsion free)}.
		\end{align}
		Such connection coincides with the Levi-Civita connections of $Re\inner{\cdot, \cdot}$ and $Im\inner{\cdot, \cdot}$ and $D\JJ=0$.
	\end{Proposition}

	\begin{Remark}
		\begin{itemize}
			
			\item A direct computation shows that the Levi-Civita connection $D$ for a holomorphic Riemannian metric $\langle \cdot, \cdot \rangle$ is explicitly described, for all $X,Y, Z\in \Gamma(T\mathbb M)$, by
			\begin{equation}
			\label{Levi-Civita}
			\begin{split}
			\langle D_X Y, Z\rangle= \frac 1 2 \Big( X\inner{Y,Z} + Y\inner{Z,X} - Z\inner{X, Y} + \\
			+\inner{[X,Y],Z} -\inner{[Y,Z],X}+ \inner{[Z,X],Y} \Big).
			\end{split}
			\end{equation}
			
			\item
			Let $\mathbb M=G$ be a complex semisimple Lie group and let $\inners\colon Lie(G)\times Lie(G)\to \C$ be the Killing form. Since the Killing form is non-degenerate and $Ad$-invariant, it can be extended to a holomorphic Riemannian metric on $G$ equivalently by left or right multiplication. In this case, the induced Levi-Civita connection admits a simple description.
			
			Indeed, let $X, Y, Z$ be left-invariant vector fields for $G$ (namely, for all $g\in G$, $X(g)=(L_g)_* X(e_G)$, $Y(g)=(L_g)_* Y(e_G)$, $Z(g)=(L_g)_* Z(e_G)$).
			
			Then $\inner{X,Y},\inner{Y,Z}$ and $\inner{X,Z}$ are constant functions and $\inner{[Z,Y],X}+\inner{[Z,X],Y}=0$ since the Killing form is $ad$-invariant. In conclusion, by the explicit expression (\ref{Levi-Civita}), we get that 
			\[
			D_X Y =\frac 1 2 [X,Y]
			\]
			for all $X,Y$ left-invariant vector fields.

		\end{itemize}
	\end{Remark}

		The notion of Levi-Civita connection $D$ for the metric $\inners$ leads to the standard definition of the $(1,3)$-type and $(0,4)$-type \emph{curvature tensors}, that we will denote with $R$, defined by
		\[
	R(X,Y,Z, T):= -\inner{R(X,Y)Z,T}=- \inner{ \nabla_X \nabla_Y Z - \nabla_Y \nabla_X Z - \nabla_{[X,Y]} Z, T }
	\]
	for all $X,Y,Z, T \in \Gamma(T\mathbb M)$.

		Since $D$ is the Levi-Civita connection for $Re\inner{\cdot, \cdot}$ and for $Im\inner{\cdot, \cdot}$, it is easy to check that all of the standard  symmetries of curvature tensors for (the Levi-Civita connections of) pseudo-Riemannian metrics hold for (the Levi-Civita connections of) holomorphic Riemannian metrics, too. So, for instance, 
		\[
		R(X,Y,Z,T)= -R(X, Y, T, Z)= R(Z, T, X, Y)= -R(Z, T, Y, X).
		\]
		Since the $(0,4)$-type $R$ is obviously $\C$-linear on the last component, we conclude that it is $\C$-multilinear.

	\begin{Remark}	
		\label{curvature tensor Killing}
		Consider the Levi-Civita connection $D$ for the Killing form for a complex semisimple Lie group $G$. Let $[\cdot,\cdot]$ denote the Lie bracket on $T_e G=\Lieg$.  Then, for all $V_1, V_2, V_3\in \Lieg$ 
		\[
		R(V_1, V_2) V_3 =-\frac 1 4 [[V_1,V_2],V_3]
		\]
		See \cite{Milnor} for a proof.
	\end{Remark}
	
\begin{Definition}	A \emph{non-degenerate plane} of $T_p \mathbb M$ is a complex vector subspace $\mathcal V< T_p \mathbb M$ with $dim_\C\mathcal V=2$ and such that $\inner{\cdot, \cdot}_{|\mathcal V}$ is a non degenerate bilinear form.
	
	For holomorphic Riemannian metrics, we can define the {complex sectional curvature} of a nondegenerate complex plane $\mathcal V= Span_\C (V,W)<T_{p}M$ as
	\begin{equation}
	\label{def curvatura}
	K(Span_\C (V,W))=\frac{-\inner{R(V,W)V,W}}{\|V\|^2 \|W\|^2 - \inner{V,W}^2}\newline.
	\end{equation}
	This definition of $K(Span_\C (V,W))$ is well-posed since $R$ is $\C$-multilinear.
\end{Definition}

	\subsection{Holomorphic Riemannian space forms}
	
	We will say that a connected holomorphic Riemannian manifold $\mathbb M=(\mathbb M, \inner{\cdot, \cdot})$ is \emph{complete} if geodesic curves - with respect to the Levi-Civita connection - can be extended indefinitely, equivalently if the exponential map is defined on the whole $T\mathbb M$.
	
	We will call $\emph{holomorphic Riemannian space form}$ a complete, simply connected holomorphic Riemannian manifold with constant sectional curvature.
	
	\begin{Theorem}
		\label{Theorem space forms}
		For all $n\in \Z_+$ and $k\in \C$ there exists exactly one holomorphic Riemannian space form of dimension $n$ with constant sectional curvature $k$ up to isometry.
	\end{Theorem}
	
	We first prove uniqueness, then existence will follow from an explicit description of the space forms. 
	
	\subsubsection{Uniqueness}
	
	In order to prove uniqueness, we extend some standard results in pseudo-Riemannian geometry to the context of holomorphic Riemannian manifolds. 
	
	\begin{Lemma}
		\label{mappa differenziale determina isometria}
		Let $f,g\colon \mathbb M\to \mathbb M'$ be two isometries between holomorphic Riemannian manifolds with $\mathbb M$ connected such that, for some point $p\in \mathbb M$, $f(p)=g(p)$ and $f_{*p}=g_{*p}\colon T_p \mathbb M\to T_{f(p)} \mathbb M'$. Then $f\equiv g$.
	\end{Lemma}
	\begin{proof}
	It is sufficient to observe that $f$ and $g$ are isometries for the real part of the holomorphic Riemannian metric, then the thesis follows by standard Pseudo-Riemannian Geometry.
	\end{proof}
	
	\begin{Lemma}
		\label{lemma space form}
		If $\mathbb M$ is a manifold of constant sectional curvature $k\in \C$, then, for any $X,Y,Z\in \Gamma(T\mathbb M)$,
		\[
		R(X,Y)Z= -k (\inner{X,Z}Y - \inner{Y,Z}X).
		\]
		In particular, $R(X,Y)Z\in Span(X,Y)$.
	\end{Lemma}
	\begin{proof}
		Define the tensor $T= k (\inner{X,Z}\inner{Y,W} - \inner{Y,Z}\inner{X,W})$. We show that $R^{0,4}-T \equiv 0$. Clearly, $(R-T)(X,Y,X,Y)=0$. Via the simmetries of the curvature tensor and of $T$, for all $X,Y,Z,U, \in \Gamma(T\mathbb M)$ we have
		\begin{align*}
		&(R-T)(X,Y+Z, X,Y+Z)= 0 \Longrightarrow (R-T)(X,Y,X,Z)= 0;\\
		&(R-T)(X+U,Y, X+U, Z)=0 \Longrightarrow (R-T)(X,Y,U,Z)=(R-T)(Y,U,X,Z);\\
		&(R-T)(X,Y,U,Z)+(R-T)(U,X,Y,Z)+(R-T)(Y,U,X,Z)=0 \Longrightarrow \\
		&	\qquad \Longrightarrow 2(R-T)(X,Y,U,Z)= (R-T)(X,U,Y,Z).
		\end{align*}
		Then $4(R-T)(X,Y,U,Z)=2(R-T)(X,U,Y,Z)=(R-T)(X,Y,U,Z)$. We conclude that $R-T\equiv 0$.
	\end{proof}

	We can now prove uniqueness in Theorem \ref{Theorem space forms}.
	\begin{proof}[Proof of Theorem \ref{Theorem space forms} -Uniqueness]
		Let $({\mathbb M},\inner{\cdot, \cdot}_{\mathbb M}),(\mathbb M',\inner{\cdot,\cdot}_{\mathbb M'})$ be two holomorphic Riemannian space forms with the same dimension $n$ and constant sectional curvature $k\in \C$. Fix any $p\in \mathbb M$ and $q\in \mathbb M'$. Since all the non-degenerate complex bilinear forms on a complex vector space are isomorphic, there exists a linear isometry $L\colon (T_p \mathbb M, \inner{\cdot, \cdot}_{\mathbb M})\to (T_q \mathbb M', \inner{\cdot, \cdot}_{\mathbb M'})$. 
		
	For all $X\in T_p M$ and for all $t\in [0,1]$, the composition of $L$ with the parallel transports via the geodesics $\gamma(t)= exp^{\mathbb M} (tX)$ and $\gamma' _L(t)= exp^{\mathbb M'}( t L(X))$ induces a linear isometry
	\[
	L_\gamma\colon T_{exp^{\mathbb M}_p(X)}{\mathbb M}\to T_{exp^{\mathbb M'}_q (L(X))}\mathbb M'.\]
		The explicit description of the curvature tensor in Lemma $\ref{lemma space form}$ and the fact that $L_\gamma$ is an isometry imply that 
	\[
		L_{\gamma(1)} ^* (R^{\mathbb M'})=  R^{\mathbb M}
		\] 
	both when $R$ is meant as a $(0,4)$-tensor and as a $(1,3)$ tensor.
	
	By iteration, for any piecewise geodesic curve $\gamma\colon [0,1]\to \mathbb M$ with $\gamma(0)=p$, one as a well-defined notion of corresponding piecewise geodesic curve $\gamma_L'\colon [0,1]\to \mathbb M'$, which induces a linear isometry 
	\[ L_\gamma\colon T_{\gamma(1)}\mathbb M  \to T_{\gamma_L' (1)}\mathbb M'
	\]
such that $L_\gamma ^* (R^{\mathbb M'})=  R^{\mathbb M}$.

The classical Cartan-Ambrose-Hicks Theorem for affine connections (e.g. see \cite{Piccione}) allows to conclude that there exist a diffeomorphism 
\[f\colon \mathbb M \to \mathbb M'\]
such that $f^*(D^{\mathbb M'})= D^{\mathbb M}$, $f(p)=q$ and such that, for every piecewise geodesic curve $\gamma\colon[0,1]\to \mathbb M$ with $\gamma(0)=p$, and $f_{*, \gamma(1)}= L_\gamma$. Since any point on $\mathbb M$ can be linked to $p$ by a piecewise geodesic curve, $f$ is an isometry.
	\end{proof}

	\subsubsection{Existence - the spaces $\mathds{X}_n \cong \faktor {SO(n+1, \C)}{SO(n, \C)}$}
	
	The simplest example of complex manifold with a holomorphic Riemannian metric is $\C^n$ with the usual inner product \[\inner{\underline z,\underline w}_0= ^t\underline {z} \cdot \underline w= \sum_{i=1}^n z_i w_i.\]

	In this paper, we will focus on another important class of examples we are going to work with in the article.
	
	Consider the complex manifold
	\[\mathds{X}_n=\{\underline z\in \C^{n+1}\ |\ ^t\underline z\cdot \underline z=-1 \}.\]
	
	The restriction to $\mathds X_n$ of the metric $\inners_0$ of $\C^n$ defines a holomorphic Riemannian metric. Indeed,
	\[
	T_p \mathds{X}_n = p^\bot= \{\underline z\in \C^{n+1}\ |\ \inner{p,\underline z}_0=0\}
	\]
	and the restriction of the inner product to $p^\bot$ is non degenerate since $<p,p>_{\C^{n+1}}\ne 0$; moreover, since $\mathds X_n\subset \C^{n+1}$ is a complex submanifold, local holomorphic vector fields on $\mathds X_n$ extend to local holomorphic vector fields on $\C^{n+1}$, proving that the inherited metric is in fact holomorphic.

Since $SO(n+1,\C)$ acts transitively by isometries on $\C^{n+1}$ with its inner product, it acts transitively by isometries on $\mathds{X}_n$ as well. Moreover, for $e=(0, \dots, 0, i)\in \mathds{X}_n$, \[Stab(e)=\begin{pmatrix}
	SO(n, \C) & \underline 0\\
	^t\underline 0 & 1
	\end{pmatrix}\cong SO(n,\C);\] we conclude that $\mathds X _n$ has a structure of homogeneous space \[\mathds{X}_n\cong \faktor {SO(n+1, \C)}{SO(n, \C)}.\]
	
	\begin{Theorem}
		\label{Theorem explicit space forms}
		The $n-$dimensional space form with constant sectional curvature $k\in \C$ is:
		\begin{itemize}
			\item $\C^n$ with the usual inner product $\inners_0$ for $k=0$;
			\item $(\mathds{X}_n, -\frac {1} k\inner{\cdot, \cdot})$ for $k \in \C^*$.
		\end{itemize}	
	\end{Theorem}
	
	It is clear that $\C^n$ is the flat space form: its real and imaginary parts are indeed the pseudo-Riemannian spaces $\R^{n,n}$ which are flat pseudo-Riemannian space forms and for which the curvature tensor is constantly zero.
	
	It is also clear that, if we prove that $\mathds{X}_n=(\mathds{X}_n, \inner{\cdot, \cdot})$ is the space form for constant sectional curvature $-1$, then, for all $\alpha \in \C^*$, $(\mathds{X}_n, -\frac {1} \alpha\inner{\cdot, \cdot})$  has the same Levi-Civita connection and is the space form of constant sectional curvature $\alpha$.
	
	We give a proof of the fact that $\mathds{X}_n=(\mathds{X}_n, \inner{\cdot, \cdot})$ is the space form of constant sectional curvature $-1$ among the following remarks on the geometry of this space.
	
	\begin{Remark}		
	\label{topologia delle geodetiche}
	\begin{enumerate}

		\item 
		Define a \emph{complex orientation} for $\mathds{X}_n$ as a holomorphic non-zero $n$-form with the property of being $SO(n+1,\C)$-invariant. At least one complex orientation exists: any $\C$-multilinear $n$-form on $T_e \mathds{X}_n$ is $SO(n, \C)$-invariant (indeed, it is $SL(n, \C)$-invariant), so the action of $SO(n+1,\C)$ defines a well-posed $SO(n+1,\C)$-invariant $n$-form on $\mathds{X}_n$. 
		
		By Lemma \ref{mappa differenziale determina isometria}, we have that \[SO(n+1,\C)\cong Isom_0 (\mathds{X}_n)\cong Isom(\mathds{X}_n, \omega_0).\]
		
		Moreover, $Isom_0 (\mathds{X}_n)$ is a subgroup of index $2$ of $Isom(\mathds{X}_n)\cong O(n+1,\C)$.

		\item The Levi-Civita connection $D$ for $\mathds{X}_n$ is the tangent component of the canonical connection $d$ for $\C^{n+1}$: in other words, seeing $T \mathds X_{n}$ as a subbundle of $T\C^{n+1}_{|\mathds  X_{n} } \equiv \mathds X_n \times\C^{n+1}$,  smooth vector fields on $T \mathds X_{n}$ can be seen as smooth functions $\mathds X_n\to \C^{n+1}$, then 
		
		\[D_{X(p)} Y = (d_p(Y)\cdot X(p) )^T \] 
		
		where $  (d_p(Y)\cdot X(p) )^T $ is the tangent component of the vector $ (d_p(Y)\cdot X(p) )$ along $T_p \mathds{X}_n$.
		
		This follows by observing that $d^T$ is a linear connection for $\mathds{X}_n$ which satisfies the same properties as the Levi-Civita connection.
		
		\item The exponential map at a point $p\in \mathds{X}_n \subset \C^{n+1}$  is given by 
		\begin{equation}
		\label{descrizione geodetiche}
		\begin{split}
		\exp_p \colon T_p\mathds{X}_n &\to \mathds{X}_n\\
		v & \mapsto
		\begin{cases}
		\cosh(\sqrt{\langle v, v\rangle} ) p + \frac{\sinh \big({\sqrt{\langle v, v\rangle}}\big)}{\sqrt{\langle v, v\rangle}} v \quad &\text{if $\langle v, v \rangle\ne 0$}\\
		p + v \quad &\text{if $\langle v, v \rangle =0$.}
		\end{cases}
		\end{split}
		\end{equation}
		Indeed, if $\gamma$ is a geodesic on $\mathds{X}_n$, then, setting the conditions $D_{\dot\gamma(t)}\dot \gamma(t)= (\frac{d \dot\gamma(t)}{dt})^T=0$ and $\frac{d \inner{\dot \gamma(t),\dot \gamma (t)}}{dt}=0$, one gets the equation $\ddot \gamma(t)=  \inner{\dot \gamma(t), \dot \gamma(t)} \gamma(t)$ which brings to the expression above. 
		
		Notice that the description of the exponential map is independent of the choice of the square root.

		\item The space $\mathds{X}_n$ is diffeomorphic to $T S^n$. 
		
		Regard $T S^n$ as $\{(\underline u, \underline v)\in \R^{n+1}\times \R^{n+1}\ |\ \|\underline u\|_{\R^{n+1}}=1, \inner{\underline u, \underline v}_{\R^{n+1}} =  0 \}$. Then a diffeomorphism $\mathds{X}_n \xrightarrow{\sim} T S^n$ is given by
		\begin{equation*}
		\underline z = \underline x + i \underline y \mapsto  (\frac{ 1}{\|\underline y\|_{\R^{n+1}}}\underline y, \underline x),
		\end{equation*}
		which is well-posed since \[\inner{\underline x + i \underline y, \underline x +i\underline y}=-1 \Longleftrightarrow \begin{cases}
\|\underline y\|^2_{\R^{n+1}}= \|\underline x \|^2_{\R^{n+1}}+1>0 &\\
\inner{\underline x, \underline y}_0=0
		\end{cases}.\]
		
		In particular, $\mathds{X}_n$ is simply connected for $n \ge 2$.
		
		\item For $n\ge 2$, $\mathds{X}_n$ has constant sectional curvature $-1$. 
		
		It is clear that it has constant sectional curvature since $SO(n,\C)$ acts transitively on complex nondegenerate planes of $(T_e \mathds{X_n},\inners)$. In order to compute the value of the sectional curvature, observe that the embedding 
		\begin{align*}
		\R^{2,1} &\hookrightarrow \C^{n+1}\\
		(x_1, x_2, x_3) &\mapsto (x_1, x_2, 0, \dots, 0, ix_3)
		\end{align*}
		induces an isometric embedding 
		\[
		\Hy^2 \hookrightarrow \mathds X_n
		\]
		which is
		totally geodesic by Formula \eqref{descrizione geodetiche}.

		\item Consider the projective quotient $p_n \colon \C^{n+1}\setminus 0 \to \CP^n$.
		
		Then ${p_n}_{|\mathds{X}_n}$ is a two-sheeted covering on its image $\mathds{PX}_n$, which corresponds to the complementary in $\CP^n$ of the non-degenerate hyperquadric \[Q_n=\{z_1^2+ \dots +z_{n+1}^2=0\}.\]

		The $SO(n+1,\C)$-invariancy of the the metric on $\mathds{X}_n$ implies that the action of $SO(n+1,\C)$ on $\CP^n$ fixes $\mathds{PX}_n$ globally (hence the complementary hyperquadric) and acts by isometries on it.

		The group of isometries of  $\mathds {PX}_n$ is given by $\Proj O(n+1,\C)$ acting on the whole $\CP^n$ as a subgroup of $PGL(n+1,\C)$. Conversely, it is simple to check that $\Proj O(n+1,\C)$ coincides exactly with the subgroup of elements of $PGL(n+1,\C)$ that fix $\mathds {PX}_n$ globally (or, equivalently, that fix $Q_n$ globally).

        \item From the expression of the exponential map given in Formula \eqref{descrizione geodetiche} it is immediate to see that if $W$ is a complex subspace of $\mathbb C^{n+1}$ then $W\cap\XXX_n$ is a totally geodesic complex submanifold. Equivalently, the intersection of $\PXX_n$ with a projective subspace of $\CP^n$ is a totally geodesic complex submanifold. 
        
        It turns out that the intersection of $\PXX_n$ with a complex line $L$ concides with the exponential of a complex subspace of (complex) dimension $1$ of the tangent plane at any point of $L\cap \PXX_n$. As a result, any totally geodesic complex submanifold is in fact the intersection of $\PXX_n$ with a complex projective subspace.
        
        We remark that if $\inners|_{W}$ is not degenerate then $W\cap\XXX_n$ is isometric to $\XXX_{k}$, where k=$\dim_{\mathbb C}W-1$.
        
        \item If $L=\mathbb P(W)$ is a projective line in $\CP^n$, then $L\cap Q$ can contain either one point (if $L$ is tangent to $Q$) or two points (if the intersection is transverse). In the latter case the restriction of the product
        $\inners$ to $W$ admits two isotropic directions, and in particular is not degenerate. In the former case there is only one isotropic direction that in fact is contained in the orthogonal subspace of $W$. The product in this case is degenerate and the restriction of the metric on $L\cap Q$ is totally isotropic.

        \item We stress that there are totally geodesic submanifolds of $\XXX_n$ which are not complex submanifolds. 
        For instance, if $W$ is a real subspace such that $\inners_{|W}$ is real, then
        \eqref{descrizione geodetiche} shows that $W\cap \XXX_n$ is totally geodesic.
        In particular if $\inners_{|W}$ is not degenerate real bilinear form, then $W\cap \XXX_n$ is a pseudo-Riemannian space-form of constant sectional curvature $-1$. 
        
        \item In the above example, the condition that $\inners_{|W}$ is real is essential.
        
        For instance, let $v_1,v_2\in \C^{n+1}$ be such that $\inner{v_i, v_j}_0 =\delta_{ij}$ and define $W=Span_\R (iv_1, \sqrt i v_2)$. Then, $W\cap \XXX_n$ has real dimension $1$ and passes through $iv_1\in \mathds X_n$ where it is tangent to the vector $\sqrt i v_2$. 
        By Formula \eqref{descrizione geodetiche} it is clear that the geodesic of $\XXX_n$ by $\sqrt i v_2$  is not contained in $W$. 
        
        This example shows both that the intersection of $\XXX_n$ with a generic vector subspace of $\C^{n+1}$ need not be totally geodesic and that smooth totally geodesic submanifolds of $\XXX_n$ need not be planar.

	\end{enumerate}
	
		\end{Remark}

\subsection{$\mathds X_1$ as $\CP^1$ endowed with a holomorphic quadratic differential}

By regarding $\mathds X_1$ as $\{(z_1,z_2)\in \C^2\ |\ z_1^2 +z_2^2=-1\}=\{(i\cos(z),i\sin(z))\ |\ z\in \C\}$, one can define the map
\begin{align*}
 F_1\colon \mathds X_1 &\xrightarrow{\sim} \C^*\\
 (z_1,z_2)=(i\cos(z),i\sin(z))&\mapsto z_2-iz_1=e^{iz}
\end{align*}
which is in fact a biholomorphism.

Denote with $\inners$ also the push-forward Riemannian holomorphic metric on $\C^*$.

A straightforward calculation shows that the automorphisms of $(\C^*,\inners)$ correspond to multiplications by a constant: indeed, for all $\begin{pmatrix}
          \cos(\alpha) & - \sin(\alpha)\\
          \sin(\alpha) & \cos(\alpha)
         \end{pmatrix}\in SO(2,\C)$, $\alpha\in \C$, we have 
\[
 F_1\ ^t\bigg( \begin{pmatrix}
          \cos(\alpha) & - \sin(\alpha)\\
          \sin(\alpha) & \cos(\alpha)
         \end{pmatrix} \cdot \begin{pmatrix}
         i\cos(z)\\
         i\sin(z)
         \end{pmatrix} \bigg)= F_1( (i\cos(z+\alpha), i\sin(z+\alpha))= e^{i\alpha}e^{iz}.
\]

Using this isotropy, we can compute explicitly the holomorphic Riemannian metric. At each point $z\in \C^*$, we have 
\[\inners_z= \lambda(z) dz^2
\]
for some holomorphic function $\lambda\colon \C^*\to \C$. 
 
Invariance by constant multiplication implies $\lambda(z)=\frac {\lambda(1)} {z^2}$. 
In order to compute $\lambda (1)$, set the condition for which the vector $F_{*(0,i)} \bigg(\begin{pmatrix} 1 \\ 0 \end{pmatrix}\bigg)= -i \in T_i \C^*$ has norm $1$ to get $\lambda(1)=1$: we conclude that 
\[
 \inners= \frac {dz^2}{z^2}.
\]

In general, we can see $\mathds X_1$ as $\CP^1\setminus \{ p_1,p_2\}$ and the holomorphic Riemannian metric as some holomorphic quadratic differential on $\CP^1$ with exactly two poles of order $2$ in $p_1$ and $p_2$.

Finally, a direct computation via $F$ shows that geodesics in $(\C^*, \inners)$ are all of the form $t\mapsto \mu_1 e^{t\mu_2}$ with $\mu_1\in \C^*$, $\mu_2\in \C$: the images correspond either to a circumference with center $0$, a straight ray connecting $0$ and $\infty$ or a spiraling ray connecting $0$ and $\infty$. As a result, the geodesics for $\inners$ coincide with the geodesics for the flat structure $\frac{|dz|^2}{|z|^2}$ induced by the metric seen as a holomorphic quadratic differential. 
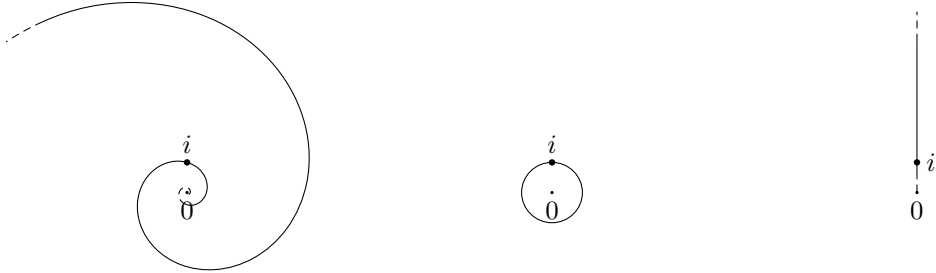
\vspace{6mm}

\begin{tikzpicture}[scale=0.4]
\draw[domain=-500:-200, hobby,densely dashed] plot ({-exp(\x/200)*sin(\x)},{exp(\x/200)*cos(\x)});
\draw[domain=-200:400, hobby] plot ({-exp(\x/200)*sin(\x)},{exp(\x/200)*cos(\x)});
\draw[domain=400:410, hobby, densely dashed] plot ({-exp(\x/200)*sin(\x)},{exp(\x/200)*cos(\x)});
\draw (12,0) ellipse (1 and 1); 
\draw (24,0.5)--(24,5);
\draw[densely dashed] (24,5)--(24,6);
\draw[densely dashed] (24,0)--(24,0.5);
\coordinate [label=above:$i$] (A) at (0,1); \fill (A) circle (3pt);
\coordinate [label=above:$i$] (B) at (12,1); \fill (B) circle (3pt);
\coordinate [label=right:$i$] (C) at (24,1); \fill (C) circle (3pt);
\coordinate [label=below:$0$] (D) at (0,0); \fill (D) circle (1.5pt);
\coordinate [label=below:$0$] (E) at (12,0); \fill (E) circle (1.5pt);
\coordinate [label=below:$0$] (F) at (24,0); \fill (F) circle (1.5pt);
\end{tikzpicture}

	\subsection{$\mathds{X}_2$ as the space of geodesics of $\Hy^3$ }	
	
	In order to get a useful description of $\mathds X _2$, we study $\mathds{P X}_2\subset \CP^2$ which corresponds to the complementary of the projective conic $\partial \mathds{P X}_2= Q:= \{z_1^2 +z_2^2 +z_3^2=0\}$. 
	
	Notice that $Q$ can be seen as the image of some adapted version of the Veronese embedding
	\begin{align*}
	v\colon \CP^1 &\to  Q\subset \CP^2\\
	(t_1 \colon t_2 ) &\mapsto (i(t_1^2 +t_2^2)\colon 2t_1 t_2\colon t_1^2-t_2^2)
	\end{align*} 

A straightforward computation shows that there exists an isomorphism 
\[\delta \colon \PSL \to SO(3,\C)\subset PGL(3,\C)\]
such that $v$ is $\delta$-equivariant with respect to the action of $\PSL$ on $\CP^1$ and of $PGL(3,\C)$ on $\CP^2$.

	We can therefore describe a holomorphic 2-sheeted covering, hence a universal covering, of $\Proj \mathds X_2$ via the map
	\begin{equation}
	\label{Mappa 2:1}
	\begin{split}
	u \colon \GG:=(\CP^1 \times \CP^1) \setminus \Delta &\to  \mathds{P X}_2 \\
	(p, q) &\mapsto \ell_{v(p)}  \cap \ell_{v(q)}
	\end{split}
	\end{equation}
	where $\ell_{v(x)}$ denotes the tangent line to the quadric $Q$ in the point $v(x)$. The fact that $u$ is a $2$-sheeted covering follows from the fact that $Q$ is a conic, hence, for every fixed external point, there are exactly two lines passing by that point and tangent to $Q$. Observe that $u$ is $\delta$-equivariant, too.
	
	A useful remark is that $u(p,q)=[\underline z]$ where $\underline z\in \C^3$ is the vector - unique up to a scalar- such that any representatives in $\C^3$ of $v(p)$ and $v(q)$ are orthogonal to $\underline z$, namely $u(p,q)=[v(p)\times v(q)]$.  
	
	By uniqueness of the universal cover, we conclude that $\XXX_{2}$ is biholomorphic to $\GG=\CP^1\times \CP^1\setminus \Delta$.
	Identifying oriented lines in $\Hy^3$ with the oriented couple of their endpoints in $\partial \Hy^3\cong \CP^1$, one can see $\GG$ as the space of maximal oriented unparametrized geodesics of $\Hy^3$.

	From now on, we see $\GG$ as $\CP^1 \times \CP^1 \setminus \Delta$ endowed with the pull-back holomorphic Riemannian metric through ($\ref{Mappa 2:1}$). 
	
	Let $(U,z)$ be an affine chart for $\CP^1$, so $(U\times U\setminus \Delta, z\times z)$ is a holomorphic chart for $\GG$, set $z\times z=:(z_1, z_2)$
		
	\begin{Prop}
	\label{metrica su G}
		The pull-back Riemannian holomorphic on $\GG$ is locally described by \[-\frac{4}{(z_1-z_2)^2} dz_1 dz_2,\] so $\GG$ endowed with this Riemannian holomorphic metric is isometric to $\mathds{X}_2$. 
		
		Moreover, $Isom_0 (\GG)\cong \PSL$ acting diagonally on $\GG$ via M\"obius maps and the immersion
		\begin{align*}
		    \Hy^2 &\to \GG\\
		    z &\mapsto (z,\overline z),
		\end{align*}
		with $\Hy^2$ in the upper half-plane model, is an equivariant isometric embedding.
	\end{Prop}	
	
	\begin{proof}
	\begin{itemize}
\item 
Isotropic complex geodesics in ${\mathbb G}$ correspond to the factor components, i.e. they are of the form $(\CP^1 \setminus \{p_2\}) \times \{p_2\}$ and $\{p_1\} \times (\CP^1 \setminus \{p_1\})$.

Indeed, the map $(\ref{Mappa 2:1})$ sends $\{p\} \times (\CP^1 \setminus \{p\})$ into the projective line $L(p)$ tangent to $Q$ at $v(p)$  by Remark $\ref{topologia delle geodetiche}.8$, we conclude that this line is totally isotropic.

\item
We fix an affine chart $(U,z)$ for $\CP^1$ and we want to write the metric on the chart $(U\times U \setminus \Delta, z\times z)$ for $\GG$. Denote $z\times z=(z_1,z_2)$.

	Since the isotropic directions for $\inners_\GG$ are the factor components, the metric at the point $(1,0)$ is of the form
	\[
	\inners_{1,0} = \lambda_0 dz_1 dz_2
	\]
	for some $\lambda_0 \in \C^*$.
	
	Since the metric is invariant by the action of $\PSL$, consider for all $(z_1, z_2)\in \C\times \C \setminus \Delta$ the morphism $\psi\colon w \mapsto  z_2w + z_1 (1-w)$ and use it to deduce that
	\[\inners_{(z_1, z_2)} = \psi^{-1}_* \inners_{0,1}=\lambda_0 \frac{1}{(z_1-z_2)^2} dz_1 dz_2.
	\]

\item	We show that for $\lambda_0=-4$ we have constant curvature $-1$. Emulating our proof of the fact that $\mathds{X}_2\subset \C^3$ has constant curvature $-1$, it is sufficient to show that, for this choice of $\lambda_0$, there exists an isometric embedding $\Hy^2\to (\GG,\inners)$ which is totally geodesic.

Consider the immersion
\begin{align*}
\sigma \colon H&\to \GG \\ 
	z &\mapsto (z,\overline z)
	\end{align*}
	where $H=\{z\in \C\ |\ Im(z)>0\}$ is the upper half plane.
	 
	For $\lambda_0=-4$, the pull-back metric on $H$ is
	 	\[
	\lambda_0 \frac 1 {- 4 (Im(z))^2} dz\cdot d\overline z=\frac 1 {(Im(z))^2} dz d\overline z,
	\]
	so $\sigma$ is an isometric immersion of the hyperbolic plane $\Hy^2$.
	
	Moreover, observe that $\sigma(H)$ is the fixed locus of the involution $(z, w)\mapsto (\overline w, \overline z)$. A direct computation shows that this involution is an isometry for the pseudo-Riemannian metric given by $Re\big( \frac{1}{(z_1-z_2)^2} dz_1 dz_2\big)$, hence it is an isomorphism for the induced Levi-Civita connection which coincides with the one induced by $\inners$: we conclude that $\sigma$ is totally geodesic.

\end{itemize}	
\end{proof}

	\subsection{$\mathds{X}_3$ as $\SL$}
	We show that $\mathds{X}_3$ is isometric (up to a scale) to the complex Lie group $\SL=\{A\in Mat(2,\C)\ |\ det(A)=1 \}$ equipped with the holomorphic Riemannian metric given by the Killing form, globally pushed forward from $I_2$ equivalently by left or right translation.
	
	Consider on $Mat(2,\C)$ the non-degenerate quadratic form given by $M\mapsto -det(M)$, which corresponds to the complex bilinear form 
	\[
	\inner{M,N}_{Mat_2}=\frac{1}{2}\bigg(tr(M\cdot N)- tr(M)\cdot tr(N) \bigg).
	\]
	In the identification $T Mat(2,\C)=Mat(2,\C)\times Mat(2,\C)$, this complex bilinear form induces a holomorphic Riemannian metric on $Mat(2,\C)$. 
	
	Observe that the action of $\SL\times \SL$ on $Mat(2,\C)$ given by $(A,B)\cdot M:= A 
	M B^{-1}$ is by isometries, because it preserves the quadratic form.

	Since all the non-degenerate complex bilinear forms on complex vector spaces of the same dimension are isomorphic, there exists a linear isomorphism $F\colon (\C^4,\inners_0) \to (Mat(2,\C),\inners_{Mat_2})$ which is also an isometry of holomorphic Riemannian manifolds: such isometry $F$ restricts to an isometry between $\mathds X_3$ and $SL(2,\C)$, where $SL(2,\C)$ is equipped with the submanifold metric.

For all $A\in \SL$,  $T_A \SL=   A\cdot T_{I_2} SL(2,\C)$ and \[T_{I_2} SL(2,\C)=\asl= \{M\in Mat(2,\C) \ |\ tr(M)=0 \}\] can be endowed with the structure of complex Lie algebra associated to the complex Lie group $\SL$.
	
	The induced metric on $\SL$ at a point $A$ is given by
	\[
	\inner{AV,AW}_A= \inner{V,W}_{I_2}= \inner{V,W}_{Mat_2}= \frac{1}{2} tr(V\cdot W),
	\]
	
	as a consequence we have $\inners_{I_2}=\frac{1}{8} Kill$ where $Kill$ is the Killing form of $\asl$, which is an $Ad$-invariant bilinear form (e.g see \cite{Killing Ad}) and $\inners_A$ is the symmetric form on $T_A \SL$ induced by pushing forward the Killing form equivalently by right or left translation by $A$.
	
	We will often focus on $\PSL=\faktor{\SL}{\{\pm I_2\}}\cong \Proj \mathds X_3$ too.

	\begin{Proposition}
		\[Isom_0 (\SL)\cong \faktor{\SL \times \SL}{\{\pm(I_2,I_2)\}}\] where the action of $\SL\times \SL$ is given by
		\begin{equation}
		\label{automorfismi SL}
		\begin{split}
		{\SL \times \SL}\times \SL \to \SL\\
		(A,B)\cdot C := A\ C\ B^{-1}
		\end{split}
		\end{equation}  
	\end{Proposition}
	\begin{proof}
	Since $Isom_0(\mathds X_3)\cong Isom_0(\C^{4}, \inners_0)\cong SO(4, \C)$, one has  $Isom_0(\SL)\cong Isom_0(Mat(2,\C))\cong SO(4,C)$. Since, for all $A, B\in \SL$, $det(A\ M\ B^{-1} )= det(M)$, the action above preserves the quadratic form, hence it is by isometries. We therefore have a homomorphism $\SL\times \SL \to Isom(\SL)$ whose kernel is $\{\pm(I_2,I_2)\}$. Finally, since \[dim_\C \bigg(\faktor{\SL\times \SL}{ \{\pm(I_2,I_2)\} }\bigg)=6= dim_\C Isom(\SL),\] we conclude that $Isom_0 (\SL)\cong \faktor{\SL\times \SL}{ \{\pm(I_2,I_2)\}}$
	\end{proof}

	\subsubsection{The immersion $\GGG \hookrightarrow \SL$}
	
	There is a geometric description of an isometric embedding of $\GGG$ into $\SL$. 
	
	Observe that one example of isomorphism between $(\C^4, \inners_0)$ and $(Mat(2,\C), \inners_{Mat_2})$ is given by
	\begin{equation}
	\label{iso C4 Mat(2,C)}
	\begin{split}
	F\colon\C^4 &\to Mat(2,\C)\\
	(z_1, z_2, z_3, z_4)&\mapsto 
	\begin{pmatrix}
	-z_1- iz_4 &  -z_2-iz_3\\
	-z_2 + iz_3 & z_1-iz_4
	\end{pmatrix}.
	\end{split}
	\end{equation}
    As we previously observed, it restricts to an isometry $F\colon \mathds X_3 \to \SL$. By regarding $\mathds X_2$ as $\{(z_1, z_2, z_3,0)\in \mathds X_3\}$, $F$ induces an isometry $F\colon \mathds X_2 \to\asl \cap \SL$. With respect to the standard action of $\SL$ on $\Hy^3$, it is simple to check that $\asl \cap \SL$ corresponds exactly to orientation-preserving isometries of order $2$, i.e. to rotations of angle $\pi$ around some axis in $\Hy^3$. We therefore have a 2-sheeted covering map 
    \[
    Axis\colon \SL\cap \asl\to \GGG_\sim =\nicefrac{(\CP^1\times \CP^1\setminus \Delta)}{(x,y)\sim(y,x)}
    \]
    that sends each matrix to the Axis of the corresponding isometry in $\Hy^3$.
    
    By uniqueness of the universal cover, $Axis$ lifts to some diffeomorphism \[\widetilde{Axis}\colon \SL\cap \asl \to \GGG.\] One can also explicitly check that the following diagram commutes:
    \[
    \begin{tikzcd}
    \mathds X_2 \arrow[r, "F"] \arrow[d, "2:1"] & \SL\cap \asl \arrow[r, "Axis"] & \GGG_\sim \\
    \PXX_2 & &\GGG \arrow[u, "2:1"] \arrow[ll, swap, "(\ref{Mappa 2:1})"]
    \end{tikzcd}
    \]
   showing that $Axis$ is a biholomorphism and that the holomorphic Riemannian metric on $\GGG$ obtained by pushing-forward via $\widetilde{Axis}$ coincides with the one we had previously defined.

\subsubsection{Orientation and cross product}

	Given a complex vector space $\mathcal V$ equipped with a non-degenerate complex bilinear form, we will say that a vector $V \in \mathcal V$ is respectively \emph{spacelike, isotropic, timelike} if $\|V\|^2$ is a \textbf{real} number and, respectively, positive, equal to zero, negative. We also say that a real vector subspace $\mathcal W \subset \mathcal V$ is \emph{spacelike, isotropic, timelike} respectively if the Killing form restricted to $\mathcal W$ is real and positive-definite, degenerate, negative-definite. 
	
	We say that two vectors $V_1, V_2 \in \mathcal V$ are \emph{orthonormal} if $\inner{V_i, V_j}= \delta_{ij}$. An \emph{orthonormal basis} is a basis of pairwise orthonormal vectors.
	
	\begin{Lemma}
		For all $V,W\in \asl$,
		\[
		\| [V, W] \|^2 =-4 \|V\|^2 \|W\|^2 + 4 \inner{V,W}^2.
		\]
		Moreover, if $V, W$ are orthonormal, then $ (V, W, \frac 1 {2i} [V,W] )$ is an orthonormal basis with respect to $\inners$.
	\end{Lemma}
	\begin{proof}
		Using Remark \ref{curvature tensor Killing}, Lemma \ref{lemma space form} and the fact that the Killing form is ad-invariant, one has
		\[
		\inner{[V,W], [V,W]}= - \inner{W, [V,[V,W]]}= -4 \inner{W, R(V,W)V}= -4 \|V\|^2 \|W\|^2 + 4 \inner{V,W}^2
		\]		
	\end{proof} 
	The previous lemma suggests the definition of a \emph{cross-product} 
	\[
	\times \colon \asl \times \asl \to \asl
	\] given by
	\[
	V \times W := \frac{1}{2i} [V,W].
	\]
	
	Since $\PSL$ acts on $T_{I_2}\SL$ via the adjoint representation and the bracket of $\asl$ is $ad$-invariant, we observe that, for all $\Phi \in Stab_0(I_2)\cong \PSL$, the image via $\Phi$ of a basis $(V,W, V\times W)$ for $T_{I_2} \SL$ is $(\Phi_*V, \Phi_*W, \Phi_*V\times \Phi_*W)$.
	
	As a result, defining a complex orientation $\omega_0$ for $\SL$ so that $\omega_0(V,W,V\times W)=1$ for some orthonormal basis $(V,W,V\times W)\subset T_{I_2}\SL$, we can conclude that, for any orthonormal frame $(V', W', X')$, $\omega_0 (V', W', V'\times W')=1$ if and only if $X' =V' \times W'$. We will call \emph{positive} such bases. 
	 $Isom_0(\SL)$ acts transitively on the set of positive orthonormal bases.

	\section{Immersed hypersurfaces in $\PML$}    
	In this section we study the geometry of smooth immersions of the form 
	\[M\to \XXX_{n+1}\]where $M$ is a smooth manifold of (real) dimension $n$ and $\XXX_{n+1}$ is the Riemannian holomorphic space form of constant sectional curvature $-1$ and complex dimension $n+1$. 
	
	As an immersion between smooth manifolds, it has very high codimension. Nevertheless, we can define a suitable class of immersions for which we can translate in this setting some aspects of the classical theory of immersions of hypersurfaces.
	In order to do it, we will introduce a new structure on manifolds that extends the notion of Riemannian metric: \emph{complex valued metrics}. 
	
	We will use $\XX, \YY, \ZZ$ to denote elements (and sections) of $TM$ and $X, Y, Z$ to denote elements (and sections) of the complexified tangent bundle $\CTM:= TM \oplus i TM= \C \otimes_\R TM$ whose elements can be seen as complex derivations of germs of complex-valued functions.

	Let $M$ be a smooth manifold of (real) dimension $m$ and $\sigma \colon M\to \mathds X_{n+1}$, with $n+1\ge m$, be a smooth immersion. Since $\mathds X_{n+1}$ is a complex manifold, the differential map $\sigma_*$ extends by $\C$-linearity to a map 
	\begin{align*}
	\sigma_* \colon \CTM &\to T\mathds X_n\\
	X= \XX + i \YY &\mapsto \sigma_*(\XX)+ \JJ\sigma_*(\YY)=: \sigma_*(X).
	\end{align*}

	Now, consider the $\C-$bilinear pull-back form $\sigma^*\inner{\cdot,\cdot}$ on $\CTM$ defined by 
	\begin{align*}
	\sigma^*\inner{\cdot, \cdot}_p\colon \C T_pM \times \C T_pM &\to \C\\
	(X, Y) &\mapsto \inner{\sigma_*X, \sigma_*Y}.
	\end{align*}
	$\sigma^*\inner{\cdot,\cdot}_p$ is $\C$-bilinear and symmetric since $\inner{\cdot, \cdot}_{\sigma(p)}$ is $\C$-bilinear and symmetric. 
	
	\begin{Definition}
		\begin{itemize}
			\item A \emph{complex (valued)} metric $g$ on $M$ is a non-degenerate smooth section of the bundle $Sym(\C T^*M\otimes \C T^*M)$, i.e. it is a smooth choice at each point $p\in M$ of a non-degenerate symmetric complex bilinear form
			\[
			g_p \colon \C T_pM\times \C T_pM \to \C.
			\]
			
			\item A smooth immersion $\sigma \colon M \to \PML$ is \emph{admissible} if $g=\sigma^*\inner{\cdot, \cdot}$ is a complex valued metric for $M$, i.e. if $\sigma^*\inner{\cdot,\cdot}_p$ is non-degenerate. 
			
			\item If $g$ is a complex metric on $M$, an immersion $\sigma\colon (M, g)\to \PML$ is \emph{isometric} if $\sigma^*\inner{\cdot, \cdot}=g$. 
		\end{itemize}
	\end{Definition}
	
	\begin{Remark}
		If $\sigma$ is an admissible immersion, then $\sigma_{*p} \colon \C T_p M\to T_{\io(p)} \PML$ is injective. Indeed, if $\sigma_{*p}(X)=0$ then clearly $\sigma^* \inner{X, \cdot }\equiv 0$, hence $X=0$.
		
		In particular, every admissible immersion is a topological immersion.
	\end{Remark}

	\subsection{Levi-Civita connection and curvature for complex metrics}    
	
	Let $M$ be a manifold of dimension $m$ and $g$ be a complex valued metric on $\CTM$. Recall that for sections $X, Y$ of $\CTM$ we have a well-posed Lie bracket $[X,Y]$ which coincides with the $\C$-bilinear extension of the usual Lie bracket for vector fields. 
	
	\begin{Definition}
		Define a \emph{connection on $\C TM$} as the $\C$-linear application
		\begin{align*}
		\nabla \colon \Gamma (\C TM) &\to \Gamma(Hom_{\C} (\C TM, \C TM))\\
		\alpha &\mapsto \nabla \alpha (\colon X \mapsto \nabla_X \alpha)
		\end{align*}
		such that, for all $f\in C^{\infty}(M, \C)$, $\nabla_X (f \alpha)= f \nabla_X \alpha + X(f) \alpha$.
	\end{Definition}
	
	In a similar way as in classical Riemannian geometry, a complex valued metric $g$ induces a canonical choice of a linear connection on $\CTM$.
	\begin{Proposition}
		For every complex valued metric $g$ on $M$, there exists a unique connection $\nabla$ on $\CTM$, that we will call \emph{Levi-Civita connection}, such that for all $X,Y \in \Gamma(\CTM)$ the following conditions stand:
		\begin{align*}
		d ( g(X, Y) )&= g(\nabla X, Y)+ g(X, \nabla Y) \qquad & & \text{($\nabla$ is compatible with the metric)};\\
		[X,Y]&= \nabla_X Y - \nabla_Y X \qquad & & \text{($\nabla$ is torsion free)}.
		\end{align*}
	\end{Proposition}
	
	Observe that if $g$ is obtained as a $\C$-bilinear extension of some (pseudo-)Riemannian metric, then the induced Levi-Civita connection is the complex extension of the Levi-Civita connection for the (pseudo-)Riemannian metric on $M$.

	We can also define the $(0,4)$-type and the $(1,3)$-type \emph{curvature tensors} for $g$ defined by 
	\[
	R(X,Y,Z, T):= -g(R(X,Y)Z,T)=- g\bigg( \nabla_X \nabla_Y Z - \nabla_Y \nabla_X Z - \nabla_{[X,Y]} Z, T \bigg)
	\]
	with $X,Y,Z, T \in \Gamma(\CTS)$.
	The curvature tensor is $\C$-multilinear and has all of the standard symmetries of the curvature tensors induced by Riemannian metrics.
	
	Finally, for avery complex plane $Span_\C (X,Y)\in \C T_p M$ such that $g_{|Span_\C (X,Y)}$ is non-degenerate, we can define the {sectional curvature} $K(X,Y):=K(Span_\C (X,Y) )$ as
	\begin{equation}
	\label{def curvatura}
	K(X,Y)=\frac{-g(R(X,Y)X,Y)}{g(X,X)g(Y,Y) -g(X,Y)^2}
	\end{equation}
	where the definition of $K(X,Y)$ is independent from the choice of the basis $\{X,Y\}$ for $Span_\C (X,Y)$.\newline
	
	It is simple to check, via the Gram-Schmidt algorithm, that in a neighbourhood of any point $p\in M$ it is possible to construct a local $g-$orthonormal frame $(X_j)_{j=1}^m$ on $M$. We show it explicitly.
	
	 Fix a orthonormal basis $(W_j (p))_{j=1}^m$ for $\C T_p M$ and locally extend, in a neighbourhood of $p$, each $W_j(p)$ to a complex vector field $W_j$. Up to shrinking the neighbourhood in order to make the definition well-posed, define by iteration the local vector fields $Y_j$ by
	\[
	Y_j:= W_j - \sum_{k=1}^{j-1} \frac{g(W_j, Y_k)}{g(Y_k, Y_k)}Y_k
	\]
	which are such that $Y_j(p)= W_j(p)$ and by construction the $Y_j$'s are pairwise orthogonal.
	Finally, up to shrinking the neighbourhood again, the vectors $X_j= \frac{Y_j}{\sqrt{g(Y_j,Y_j)}}$ (defined for any local choice of the square root) determine a local $g$-orthonormal frame around $p$. Similarly, every set of orthonormal vector fields can be extended to a orthonormal frame.\newline

	Let $(X_j)_{j=1}^m$ be a local orthonormal frame for $g$, with $X_j\in \CTM$. Let $(\theta^j)_{j=1}^m$ be the correspondent coframe, $\theta^i\in \C T^*M$, defined by $\theta^i= g(X_i, \cdot)$.
	
	We can define the \emph{Levi-Civita connection forms} $\theta^i_j$ for the frame $(X_i)_i$ by
	\[
	\nabla X_i = \sum_h \theta_i ^h \otimes X_h.
	\] 
or, equivalently, by the equations
	\[\begin{cases}
	d \theta^i =- \sum_j \theta_j ^i \wedge \theta^j\\
	\theta_j^i=-\theta_i ^j
	\end{cases}.
	\]

	\subsection{Extrinsic geometry of hypersurfaces in $\PML$}

	From now on, assume $dim (M)=n$.
	
	Let $\sigma\colon M \to \PML$ be an admissible immersion and $g=\sigma^*\inner{\cdot, \cdot}$ be the induced complex metric.
	Denote with $D$ be the Levi-Civita connection on $\PML$ and with $\nabla$ be the Levi-Civita connection for $g$. We want to adapt the usual extrinsic theory for immersed hypersurfaces to our setting.

	Define the pull-back vector bundle $\Lambda=\io^* (T\PML)\to M$, which is a complex vector bundle and is endowed on each fiber with the pull-back complex bilinear form. 
	
	If $U$ is an open subset of $M$ over which $\io$ restricts to an embedding, then $\Lambda_{|U} \cong T\PML_{| \io (U)}$.
	
	Observe that $\Lambda$ has a structure of complex vector bundle defined by $i\sigma^*(v):=\sigma^*(\mathds J v)$ for all $v\in T\PML$. Each fiber can also be equipped with the pull-back complex bilinear form that we will still denote with $\inners$.

	Since the map $\io_* \colon \C TM \to \PML$ is injective, $\C TM$ can be seen canonically as a complex sub-bundle of $\Lambda$ via the correspondence
	\begin{align*}
	\C TM &\hookrightarrow \Lambda\\
	X &\mapsto \io^* ({\io_* X}),
	\end{align*}
     Moreover, the complex bilinear form on $\Lambda$ corresponds to the one on $\C TM$ when restricted to it, since $\io$ is an isometric immersion.

	We pull back the Levi-Civita connection $D$ of $T\PML$ in order to get a $\R$-linear connection $\overline \nabla$ on $\Lambda$,
	\[
	\overline \nabla =\sigma^*D  \colon \Gamma(\Lambda) \to \Gamma \Big( Hom_{\R}(TM, \Lambda) \Big)=  \Gamma \Big( Hom_{\C}(\C TM, \Lambda) \Big).
	\] Observe that $\overline \nabla$ is completely defined by the Leibniz rule and by the condition
	\[
	\overline \nabla_\XX \ \io^*\xi := \io^*\Big( D_{ \io_* \XX}\ \xi   \Big) \qquad \forall \xi \in \Gamma(T \PML), \XX\in \Gamma(TM).
	\]
	By $\C$-linearity, we can see $\overline \nabla$ as a map 
	\[ \overline \nabla \colon \Gamma(\Lambda) \to \Gamma \Big( Hom_{\C}(\C TM, \Lambda) \Big).  \]
	by defining \[
	\overline \nabla_{\XX_1 + i \XX_2} \hat \xi := \overline \nabla_{\XX_1} \hat \xi + i \overline \nabla_{\XX_2} \hat \xi.
	\]

Via the canonical immersion of bundles $\CTM\hookrightarrow \Lambda$, it makes sense to consider the vector field $\overline \nabla _X Y$ with $X,Y\in \Gamma(\CTM)$.

Since $D\JJ=0$ on $\PML$, for all $\XX \in \Gamma(TM)$ and $Y\in \Gamma(\C TM)$ we have that
	\begin{align*}
	\overline{\nabla}_\XX (i  \sigma^*(\xi)) &= \sigma^* \Big(D_{\sigma_* \XX} { (\mathds J \xi)} \Big) =  \sigma^* \Big(D_{\io_* \XX} \JJ\ \xi \Big)=\\
	&= \sigma^* \Big(\JJ D_{\io_* \XX} \ \xi \Big)= i  \overline \nabla_\XX \sigma^*\xi.
	\end{align*}
	We can conclude that $\overline \nabla$ is $\C$-bilinear.
	
	We observed that $\C TM$ is a sub-bundle of $\Lambda$ over which the restriction of the complex bilinear form of $\Lambda$ is non-degenerate. Hence, we can consider the \emph{normal bundle} $\NN= \C TM^{\bot}$ over $M$ defined as the orthogonal complement of $\C TM$ in $\Lambda$. $\NN$ is a rank-1 complex bundle on $M$. 
	
	For all local fields $X,Y \in \Gamma( \C TM)$ we can define $\II (X,Y)$ as the component in $\NN$ of $\overline \nabla_X Y$. 
	
	\begin{Proposition}
		For all $X,Y \in \Gamma(\C TM)$, the component in $\C TM$ of $\overline \nabla_X Y$ is $\nabla_X Y$, where $\nabla$ is the Levi-Civita connection on $M$. In other words:
		\[
		\overline \nabla_X Y = \nabla_X Y + \II (X,Y).
		\]
		Moreover, $\II$ is a symmetric, $\C$-bilinear tensor.
	\end{Proposition}
	\begin{proof}
		 By $\C$-linearity of $\nabla$ and $\overline \nabla$, it is enough to prove it for $X,Y\in \Gamma(TM)$. The proof just follows the standard proof in the Riemannian case: defining $A_X Y:= \overline{\nabla}_X Y - \II(X,Y)$, one can show that $A$ is a connection on $\CTM$, that it is torsion free and compatible with the metric, hence $A=\nabla$.
	\end{proof}

	For all $p\in M$, consider on a suitable neighbourhood $U_p \subset M$ a norm-1 section $\nu$ of $\NN$: we call such $\nu$ a local \emph{normal vector field} for $\sigma$. 

A local normal vector field fixed, we can locally define the \emph{second fundamental form} of the immersion $\sigma$ as the tensor  
\[
\IIs:=\inner{\II, \nu}= \inner{\overline \nabla, \nu}.
\]

Since there are two opposite choices for the section $\nu$, $\IIs$ is defined up to a sign.

We define the \emph{shape operator} $\Psi$ associated to the immersion $\io\colon M \to \PML$ as the tensor
\begin{align*}
\Psi \in \Gamma \Big( Sym(\C T^*M \otimes_{\C} \C TM)\Big)
\end{align*}
such that, $\forall p\in M$ and $\forall \XX,\YY \in T_p M$,  $g(\Psi(\XX),\YY)= -\inner{\II(\XX,\YY), \nu} =-\IIs(\XX,\YY)$.
As $\IIs$ is defined up to a sign, $\Psi$ is defined up to a sign as well.

We will say that $\io$ is \emph{totally geodesic} if and only if $\II\equiv 0$, i.e. $\Psi \equiv 0$.

	\begin{Remark}
		Recall that the curvature tensor of $\overline\nabla$ is the pull-back of the curvature tensor $D$. That is:
     	\begin{align*}
		\overline R(\XX,\YY)Z = \io^*\Big( R^D (\io_* \XX, \io_* \YY)\io_*Z \Big)
		\end{align*}
		
	Hence, %if $\XX,\YY$ are $\R$-linearly independent in $TM$ (therefore $\C$-linearly independent in $\C TM$), then
		\begin{align*}
		g(\overline R (\XX,\YY)\YY, \XX)= \inner{R(\io_* \XX, \io_* \YY)\io_*\YY, \io_*\XX}\\
		=-(g(\XX,\XX)^2 g(\YY,\YY)^2-g(\XX, \YY)^2).
	%	= K(Span_\C (\io_*\XX, \io_*\YY) )=-1.
		\end{align*}
	\end{Remark}
	\noindent

	\begin{Proposition}
		Let $\nu$ be an normal vector field and let $\Psi$ be the corresponding shape operator. Then,
		\[
		\Psi =\overline\nabla \nu
		\]
	\end{Proposition}
	\begin{proof}
		For all $\XX \in \Gamma(TM)$ and $Y\in \Gamma(\C TM)$,
		\begin{align*}
		g(\Psi(\XX), Y)&=-\IIs(\XX,Y) =-\inner{\overline \nabla_\XX Y, \nu} =\\
		&= {\inner{ Y, \overline \nabla_\XX \nu}} - d(\inner{Y,\nu})(\XX)=\\
		&= \inner{ Y, \overline \nabla_\XX \nu}.
		\end{align*}
		Moreover, 
		\begin{equation*}
		\inner{\overline\nabla_\XX \nu, \nu}=  \frac 1 2 d(\inner{\nu, \nu})(\XX)=0.
		\end{equation*}
		Hence, $\overline \nabla_\XX \nu \in \C TM$ and $\Psi(\XX)= \overline\nabla_\XX \nu$. The proof follows by $\C$-linearity.
	\end{proof}
	
\subsection{Gauss and Codazzi equations}	
	
	Consider the exterior covariant derivative $d^\nabla$ associated to the Levi-Civita connection $\nabla$ on $\C TM$, namely
	\[
	(d^\nabla \Psi)(X,Y)= \nabla_X (\Psi (Y)) - \nabla_Y (\Psi(X)) - \Psi([X,Y])
	\]
	for all $X,Y\in \Gamma(\C TM)$.
	
	Fix a local orthonormal frame $(X_i)_{i=i}^n$ with coframe $(\theta^i)_{i=1}^n$.  We can see the shape operator $\Psi$ in coordinates:
	\[
	\Psi=: \Psi^i_j \cdot \theta^j \otimes X_i =: \Psi^i \otimes X_i,
	\]
	where $\Psi^i_j \in \C^{\infty}(U,\C)$, with $\Psi^i_j=\Psi^j_i$, and $\Psi^i\in \Omega^1(\CTM)$. Notice that
	\begin{equation}
	\label{equazione Psi_i}
	\IIs(X_i, \cdot)= -\inner{\Psi(X_i), \cdot}= - \inner{\Psi^j_i X_j, \cdot}= -\Psi^j_i \theta^j=-\Psi_j^i \theta^j=-\Psi^i
	\end{equation}

	We are ready to state the first part of an adapted version of Gauss-Codazzi Theorem.
	
	\begin{Theorem}[Gauss-Codazzi, first part]
		\label{Teo Gauss-Codazzi}
		Let $\io \colon M^n \to \PML$ be an admissible immersion and $g=\sigma^* \inner{\cdot, \cdot}$ be the induced complex metric.
		
		Let $\nabla$ be the Levi-Civita connection on $M$, $R^M$ the curvature tensor of $g$ and $\Psi$ the shape operator associated to $\sigma$. 
		
		Fix a $g$-orthonormal frame $(X_i)_i$ with corresponding coframe $(\theta^i)_i$.
		
		Then the following equations hold: \begin{align}
			\label{GC1}
			&1)d^\nabla \Psi \equiv 0 \qquad &&\text{(Codazzi equation)};\\
			\label{GC2}
			&2)    R^M(X_i, X_j, \cdot, \cdot) - \Psi^i \wedge \Psi^j= -\theta^i\wedge \theta^j \qquad &&\text{(Gauss equation)}. 
			\end{align}
	\end{Theorem}\vspace{-3ex}
	\begin{proof}
	In a neighbourhood $U$ of a point $p\in M$, fix a local normal vector field $\nu$. Let $Z\in \Gamma(\C TU)$, let $\XX(p),\YY(p) \in T_p M$ and let $\XX,\YY\in \Gamma(TU)$ be local extensions of $\XX(p)$ and $\YY(p)$ such that $[\XX, \YY]=0$ (such extensions can be constructed through a local chart). 
		
		By definition,
		\[
		\overline \nabla_ { \YY} Z= \nabla_{\YY} Z + \II (Z, \YY), 
		\]
		hence, by second derivation, 
		\begin{align*}
		\overline \nabla_\XX \overline{\nabla}_{\YY} Z&= \overline \nabla_\XX \nabla_{\YY} Z + \overline\nabla_\XX (\IIs(Z,\YY)\nu)=\\
		&= \nabla_\XX \nabla_{ \YY} Z + \IIs(\XX, \nabla_\YY Z)\nu + \overline \nabla_\XX (\IIs(Z,\widetilde \YY) \nu)=\\
		&=\nabla_\XX \nabla_{\YY} Z + \II(\XX, \nabla_{\YY} Z)+ d(\IIs(Z,{\YY}))(\XX) \nu + \IIs(Z,\YY) \overline\nabla_\XX \nu=\\
		&= \nabla_\XX \nabla_{\YY} Z +  \IIs(Z,\YY)\Psi(\XX) + \Big( \IIs(\XX, \nabla_\YY Z) + d(\IIs(Z,\YY))(\XX) \Big)\nu.
		\end{align*}
		
		We compute $\overline R (\XX,\YY)Z$.  Since $[ \XX,  \YY]=0$, we have
		\begin{align*}
		\overline R (\XX,\YY)Z =& R^M (\XX,\YY) Z - g(\Psi(\YY),Z)\Psi(\XX)+g(\Psi(\XX),Z)\Psi(\YY) + \\
		&+ \Big( \IIs(\XX, \nabla_\YY Z)- d(\IIs(Z,  \XX))(\YY)- \IIs(\YY, \nabla_\XX Z)+ d(\IIs(Z,  \YY))(\XX) \Big)\nu=\\
		=& R^M (\XX,\YY) Z - g(\Psi(\YY),Z)\Psi(\XX)+g(\Psi(\XX),Z)\Psi(\YY) +\\
		&+\Big( (\nabla_\YY \IIs)(Z,  \XX) - (\nabla_\XX \IIs)(Z,  \YY) + \IIs(\nabla_\XX  \YY - \nabla_\YY  \XX, Z) \Big) \nu=\\
		=& R^M (\XX,\YY) Z - g(\Psi(\YY),Z)\Psi(\XX)+g(\Psi(\XX),Z)\Psi(\YY) +\Big( (\nabla_\YY \IIs)(Z,  \XX) - (\nabla_\XX \IIs)(Z,  \YY)\Big) \nu.
		\end{align*}
		Recall that $\overline R (\XX,\YY)Z= \io^* \Big( R^D(\io_*\XX, \io_*\YY) \io_*Z \Big)$. 
		
		Also recall that by Lemma \ref{lemma space form}, for all $V_1,V_2,V_3\in T_{\underline z} \PML$
		\[R^D(V_1, V_2)V_3 \in Span_\C(V_1, V_2)\]
		thus $\inner{\overline R(\XX,\YY)Z , \nu}=0$.
		
		As a result, we have the equalities 
		\begin{align*}
		\text{(a)}:&\  (\nabla_\XX \IIs)(Z, \YY) - (\nabla_\YY \IIs)(Z,  \XX) =0;\\
		\text{(b)}:&\  \overline R (\XX,\YY)Z =  R^M (\XX,\YY) Z - g(\Psi(\YY),Z)\Psi(\XX)+g(\Psi(\XX),Z)\Psi(\YY).
		\end{align*}
		We deduce (\ref{GC1}) from (a) and (\ref{GC2}) from (b).

		By manipulation of (a), we get:
		\begin{align*}
				(\nabla_\YY \IIs)(Z,  \XX) - (\nabla_\XX \IIs)(Z,  \YY)
		=&-d(g(\Psi( \XX),Z) )(\YY)+ g(\nabla_\YY Z, \Psi(\XX)) + g(Z, \Psi(\nabla_\YY  \XX))+\\
		&+d(g(\Psi( \YY),Z) )(\XX) -g(\nabla_\XX Z, \Psi(\YY)) - g(Z, \Psi(\nabla_\XX  \YY) ) =\\
		=& g( Z, \nabla_\XX \Psi ( \YY) ) - g( Z, \nabla_\YY \Psi( \XX) ) =\\
=& g(Z, d^{\nabla} \Psi (\XX,\YY))
		\end{align*}
		
		Since this holds for all $Z$ in $\Gamma(\C TM)$, $d^{\nabla} \Psi (\XX,\YY)=0$ for all $\XX,\YY \in T_p M$, hence $d^{\nabla} \Psi=0$.

		In order to prove $(\ref{GC2})$, observe that by $\C$-linearity (b) is equivalent to
		\[
		b': \overline R (X,Y)Z =  R^M (X,Y) Z - g(\Psi(Y),Z)\Psi(X)+g(\Psi(X),Z)\Psi(Y)
		\]
		for all $X,Y,Z\in \Gamma(\CTM)$. 
		
		Recalling Lemma \ref{lemma space form} and equation (\ref{equazione Psi_i}), in the orthonormal frame we have
		\begin{align*}
		\overline R (X_i, X_j, X_k, X_h)=& R^M(X_i, X_j, X_k, X_h)+\\
		&+g(\Psi(X_i), X_h)g(\Psi(X_j), X_k) - g(\Psi(X_j), X_h)g(\Psi(X_i),X_k) \Longrightarrow
		\end{align*}
		\[-\big(g(X_i, X_k)g(X_j, X_h)- g(X_i, X_h)g(X_j, X_k)\big)=
		R^M(X_i, X_j, X_k, X_h) + (\Psi^i\wedge \Psi^j) (X_h, X_k) \Longrightarrow\]
		\[
		-\theta^i\wedge \theta^j= R^M(X_i, X_j, \cdot, \cdot)- \Psi^i\wedge \Psi^j.
		\]
		The proof follows.

	\end{proof}

	For $n=2$, the Gauss equation $(\ref{GC2})$ can be written in an simpler way. Fixed an orthonormal frame $\{X_1,X_2\}$ on the surface $M$, the curvature tensor $R$ is completely determined by its value $R(X_1,X_2,X_1,X_2)$ which is the curvature of $M$. Similarly, $\Psi^1\wedge \Psi^2= (\Psi_1^1 \Psi_2^2- \Psi^2_1\Psi^1_2)\theta^1\wedge \theta^2= det(\Psi)\theta^1\wedge \theta^2$. Therefore, equation $(\ref{GC2})$ is equivalent to 
	\begin{equation}
	\label{GC2superfici}
	\text{Gauss equation for surfaces in $\PSL$:}\qquad \qquad K- det(\Psi)=-1
	\end{equation}

	\section{Integration of Gauss-Codazzi equations and immersion data}
	
	The main aim of this section is to show that the converse of Theorem \ref{Teo Gauss-Codazzi} is also true for simply connected manifolds, in the way explained in the following theorem.
	
	\begin{Teo} [Gauss-Codazzi, second part]
		\label{Teoremone} 
		Let $M$ be a smooth simply connected manifold of dimension $n$. Consider a complex metric $g$ on $M$ with induced Levi-Civita connection $\nabla$, and a $g$-symmetric bundle-isomorphism 
		$\Psi\colon \CTM \to \CTM$.

		Assume $g$ and $\Psi$ satisfy
		\begin{align*}
		1) &d^\nabla \Psi \equiv 0;\\
		2)	&R(X_i, X_j, \cdot, \cdot) - \Psi^i \wedge \Psi^j= -\theta^i\wedge \theta^j\\
		&\text{for every local $g$-frame $(X_i)_{i=1}^n$ with corresponding coframe $(\theta^i)_{i=1}^n$ and with $\Psi= \Psi^i \otimes X_i$}.
		\end{align*}
		
		Then, there exists an isometric immersion $\sigma \colon (M,g) \to \ML$ with shape operator $\Psi$.
		
		Such $\sigma$ is unique up to post-composition with an element in $Isom_0(\XXX_{n+1})\cong \nobreak {SO(n+2,\C)}$. More precisely, if $\sigma'$ is another isometric immersion with the same shape operator, there exists a unique $\phi\in Isom_0(\XXX_{n+1})$ such that $\sigma'(x)= \phi \cdot \sigma(x)$ for all $x\in M$
	\end{Teo}

	We state the case of surfaces in $\SL$ as a corollary.
	\begin{Cor}
		Let $S$ be a smooth simply connected surface. Consider a complex metric $g$ on $S$, with induced Levi-Civita connection $\nabla$, and a $g$-symmetric bundle-isomorphism 
		$\Psi\colon \CTS \to \CTS$.
		
		Assume $g$ and $\Psi$ satisfy
		\begin{align}
\label{Codazzi inversa}	1) &d^\nabla \Psi \equiv 0;\\
	2)	&K=-1+det(\Psi).
	\end{align}
		Then, there exists an isometric immersion $\sigma \colon S \to \SL$ whose corresponding shape operator is $\Psi$.
		
		Moreover, such $\sigma$ is unique up to post-composition with elements in $Isom_0(\SL)= \Proj(\SL \times \SL)$. 
	\end{Cor}

	Let $G$ be a Lie group. Recall that the \emph{Maurer-Cartan form} of $G$ is the $1-$form $\omega_G\in \Omega^1(G, Lie(G))$ defined by
	\[
	(\omega_G)_g (\dot g)= (L_{g}^{-1})_* (\dot g), 
	\]	hence it is invariant by left translations. Moreover, it is completely characterized by the differential equation
	\[
	d\omega_{G} + [\omega_G, \omega_G ]=0.
	\]
	
	The proofs of both existence and uniqueness of Theorem \ref{Teoremone} are based on the following result on Lie groups.

	\begin{Lemma}
		\label{Lemma chiave unicità}
		Let $M$ be a simply connected manifold and $G$ be a Lie group.
		
		Let $\omega \in \Omega^1 (M, Lie(G))$.
		
		Then 
		\begin{equation}
		\label{differenziale MC}
	d \omega+[\omega,\omega]=0
		\end{equation}
		 if and only if there exists a smooth $\Phi\colon M \to G$ such that $\omega=\Phi^* \omega_G$, where $\omega_G$ denotes the Maurer-Cartan form of $G$.
		
		Moreover, such $\Phi$ is unique up to post-composition with some $L_g$, $g\in G$.
	\end{Lemma}

A proof of Lemma \ref{Lemma chiave unicità} follows by constructing a suitable Cartan connection - depending on $\omega$ - on the trivial principal bundle $\pi_M \colon M\times G \to M$ so that its curvature being zero is equivalent to condition \eqref{differenziale MC}. 
In Section 5 we will prove Proposition \ref{Teorema esistenza foliazione} which can be seen as a more general version of Lemma \ref{Lemma chiave unicità}.

	Lemma \ref{Lemma chiave unicità} provides a remarkable one-to-one correspondence between immersions $M\to SO(n+2,\C)$ up to translations and elements in $\Omega^1(M, \Lieonn)$ satisfying $(\ref{differenziale MC})$.

\subsection{Proof of Theorem \ref{Teoremone}}	
	
	We start the proof of Theorem \ref{Teoremone} providing a way, which will turn out to be very useful for our pourposes, to construct (local) immersions into $SO(n+2,\C)$ from immersions into $\mathds{X}_{n+1}$.

	Let $dim(M)=n$ and let $\sigma \colon (M,g) \to \PML$ be an isometric immersion with shape operator $\Psi$.

	Let $(X_i)_{i=1}^n$ be a local orthonormal frame for $\CTM$ on some open subset $U\subset M$ and $\nu$ be a normal vector field for the immersion $\sigma$. 
	Recalling that $\mathds X_{n+1}\subset \C^{n+2}$, notice that  $( \sigma_*(X_1), \dots \sigma_* (X_n), \sigma_*\nu(x), -i\sigma(x) )$ is an orthonormal basis for $\C^{n+2}$. Up to switching $\nu$ with $-\nu$, we can assume that this basis of $\C^{n+2}$ lies in the same $\SOnn$-orbit as the canonical basis $(\underline v_1^0, \dots, \underline v_{n+2}^0)$ of $\C^{n+2}$. Recall that we defined \[\underline e=\begin{pmatrix}
	0\\
	\dots\\
	0\\
	i
	\end{pmatrix}=i\underline v^0_{n+2}\in \mathds X_{n+1}.\]
	
	Given an immersion $\sigma\colon U\to \mathds X_{n+1}$, we construct the smooth map \[\Phi\colon U \to \SOnn\] defined, for all $x\in U$, by
	\begin{equation}
	\label{definizione Phi}
	\begin{split}
	\Phi(x) ( \underline v_i) &= \sigma_{*x} (X_i) \qquad \qquad i=1,\dots n\\
	\Phi(x) ( \underline v_{n+1}) &= \sigma_{*x} (\nu),\\	\Phi(x)(\underline v_{n+2})&=-i\Phi(x) (\underline e) =   -i\sigma(x) 	
	\end{split}
	\end{equation}
We denote by $\alpha, \beta, \gamma$ indices in $\{1, \dots, n+2\}$ and by $i, j, k, h$ indices in $\{1, \dots, n\}$
	
	Let $\omega_G$ be the Maurer-Cartan form for $\SOnn$ and define $\omega_\Phi:= \Phi^* \omega_G\in \Omega^1(\C TU,\Lieonn)$, where we consider the canonical identification $\Lieonn=Skew({n+2}, \C)$.

	By definition of the Maurer-Cartan form, $\omega_\Phi$ satisfies
	\[
	\omega_\Phi (x) (\XX)= (L_{\Phi(x)} ^{-1})_* \circ \Phi_{*x} (\XX)= \Phi(x)^{-1} \cdot \big(\Phi_{*x} (\XX)\big) 
	\]

	We are now able to prove the following.
	\begin{Prop} [Uniqueness in Theorem \ref{Teoremone}]
		\label{Teo unicità}
		Let  $M$ be a connected smooth manifold of dimension $n$. Let $\sigma, \sigma'\colon M \to \ML$ be two admissible immersions of a hypersurface with the same induced complex metric $g=\sigma^*\inners=(\sigma')^*\inners$ and the same shape operator $\Psi$. Then, there exists a unique $\phi\in Isom_0(\ML)$ such that $\sigma' (x)=\phi \cdot \sigma(x)$ for all $x\in M$.
	\end{Prop}
	\begin{proof}We first prove that $\phi$ is unique. Assume $\phi_1 \circ \sigma=\phi_2 \circ \sigma$ for some $\phi_1,\phi_2\in SO(n+2, \C)$, then, for any $x\in U$, $\phi_2^{-1} \circ \phi_1$ coincides with the identity on $Span_\C(\sigma(x), \sigma_*(T_x M)  )\subset \C^{n+2}$ which is a complex vector subspace of dimension $n+1$, so $\phi_1=\phi_2$.
		
		We now prove that $\phi$ exists. Observe that it is enough to show that the statement holds for any open subset of $M$ that admits a global $g$-orthonormal frame, then the thesis follows by the uniqueness of $\phi$ and by the fact that $M$ is connected. We therefore assume in the proof that $M$ admits a global ortonormal frame without loss of generality.
	
		Consider an isometric immersion $\sigma\colon M \to \ML$ with immersion data $(g,\Psi)$, fix a $g$-frame $(X_i)_i$ construct the lifting $\Phi\colon M \to Isom_0(\ML)$ as in (\ref{definizione Phi}).
		
		In order to prove the statement, it is enough to show that the form $\omega_\Phi$ depends on $\sigma$ only through $g$ and $\Psi$. Indeed, assume that $\sigma$ and $\sigma'$ are two isometric immersions with the same immersion data $(g,\Psi)$: the induced maps $\Phi, \Phi'\colon M \mapsto G= SO(n+2,\C)$ would be such that $\Phi^* \omega_G = (\Phi')^* \omega_G$, hence, by Lemma \ref{Lemma chiave unicità}, there exists $\phi\in \SOnn$ such that $\Phi'(x)=\phi\cdot \Phi(x)$ for all $x\in M$, hence $\sigma'(x)=\phi\cdot \sigma(x)$. \newline

		For $\alpha=1, \dots, n+2$, define \[\Phi_\alpha (x):= \Phi(x)\underline v^0_\alpha.\]
Notice that $ i\Phi_\alpha\in \mathds{X}_{n+1}$ for all $\alpha$ and that $\Phi_\beta(x)\in T_{i\Phi_\alpha(x)} \mathds{X}_{n+1}$ for all $\alpha\ne \beta$. In particular, for $\beta \ne n+2$,  $\Phi_\beta$ can be seen as a vector field on $\mathds X_{n+1}$ along $\sigma(U)$.

	As usual, denote with $\theta^i_j$ be the Levi-Civita connection forms for $(X_i)_{i=1}^n$, so $
\nabla X_i = \theta_i ^j \otimes X_j$. Also recall the notation $\Psi= \Psi^i \otimes X_i$.

For all $x\in U$ and $\YY\in T_x U$,

\begin{align*}
\Phi(x)\omega_\Phi (\YY) \underline v_\alpha^0 &= \Phi_{*x} (\YY) \underline v_\alpha^0 = \big(\Phi(\cdot) \underline v_\alpha^0 \big)_{*x} (\YY)= (\Phi_{\alpha})_ {*x} (\YY).
\end{align*}
In particular, \begin{align*}
\omega_\Phi(\YY) \cdot \underline v_{n+2}^0 &= (\Phi(x))^{-1}(\Phi(\cdot) \underline v^0_{n+2})_* (\YY) =-i(\Phi(x))^{-1} \sigma_*(\YY)=\\
&= \sum_{k=1}^n -i\inner{\YY, X_k} \big( (\Phi(x))^{-1}\sigma_*(X_k)\big)= \sum_{k=1}^n -i\inner{\YY, X_k} \underline v_k^0= \sum_{k=1}^n -i\theta^k(\YY) \underline v_k^0
\end{align*}

We now compute $\omega_\Phi \cdot \underline v_j^0$ with $j=1, \dots, n$. Recall that the Levi-Civita connection on $\mathds X_{n+1}\subset \C^{n+2}$ is the tangent component of the standard differentiation on $\C^{n+2}$ to deduce that
\begin{align*}
\Phi(x)\omega_\Phi (\YY) \underline v_j^0&= (\Phi_j)_{*x}(\YY)= \frac{\partial \Phi_j}{\partial \YY} = \sigma_{*x} (\overline \nabla_ \YY \Phi_j) + \inner{ ( \Phi_j)_*{ \YY}, \Phi_{n+2}} \Phi_{n+2}=\\
&= \sigma_* (\nabla_\YY X_j) + \IIs (\YY, X_j) \sigma_*(\nu) -\inner{\Phi_{j}, (\Phi_{n+2})_*{ \YY}} \Phi_{n+2}= \\
&=\sum_k \theta_j ^k (\YY) \Phi_k +\IIs (\YY, X_j) \Phi_{n+1}+ i\inner{ \Phi_j, \sigma_*(\YY)} \Phi_{n+2}.
\end{align*}

As a result, 
\begin{align*}
\omega_\Phi(\YY) \underline v_j^0&= \theta_j^k \underline v_k^0 + \IIs (\YY, X_j)\underline v_{n+1}^0 +i \inner{X_j, \YY} \underline v_{n+2}^0=\\
&= \theta_j^k \underline v_k^0 - \Psi^j(\YY)\underline v_{n+1}^0 +i \theta^j(\YY) \underline v_{n+2}^0.
\end{align*}
 
We therefore have a complete description of $\omega_\Phi \in \Gamma(\CTM, \Lieonn)= Skew(n, \Gamma(\CTM))$ in terms of $g$ (and of the induced connection $\nabla$) and $\Psi$ only:

		\[
		\omega_\Phi=
		\begin{pmatrix}
		{\scalebox{2} {$\Theta$} }& \begin{matrix}
		-\Psi^1 & -i\theta^1 \\
		\dots &\dots \\
		-\Psi^n & -i\theta^n
		\end{matrix} \\
		\begin{matrix}
		\Psi^1 & \dots & \Psi^n\\
		i\theta^1 & \dots & i\theta^n
		\end{matrix} & 
		\begin{matrix}0\text{    } & 0\text{    } \\ 
		0\text{    } & 0\text{    }
		\end{matrix}
		\end{pmatrix}  
		\]
		where $\Theta=(\theta^i_j)_{i,j}$. The proof follows.

	\end{proof}

	\begin{Cor}
		\label{Corollario Teo unicità}
		Let $\sigma, \sigma'\colon (M^n,g) \to \ML$ be two isometric immersions with the same shape operator $\Psi$. Assume $\sigma(x)=\sigma'(x)$ and $\sigma_{*x}=\sigma'_{*x}$, then $\sigma\equiv\sigma'$.
	\end{Cor}
	\begin{proof}
		By Proposition \ref{Teo unicità}, there exists $\phi \in \SOnn$ such that $\phi \circ \sigma=\sigma'$. Since there exists at most one matrix in $\SOnn$ sending $n+1$ given vectors in $\C^{n+2}$ into other $n+1$ given vectors, we have $d_{\sigma(x)} \phi=id$, hence $\phi=id$.
	\end{proof}

We are finally able to prove existence.

	\begin{proof} [Proof of Existence in Theorem \ref{Teoremone}] We prove the statement in two steps.
		
		\emph{Step 1.}     
		Assume $M=U$ is such that there exists a globally defined $g$-orthonormal frame $(X_i)_{i=1}^n$ in $\C T U$ with dual frame $(\theta^i)_i$.
		
		Our aim is to construct a suitable form $\omega\in \Omega^1(U, \Lieonn)$ satisfying \eqref{differenziale MC} in order to be able to apply Lemma \ref{Lemma chiave unicità}. 
		
		Let $\Theta=(\theta^i_j)$ be the skew-symmetric matrix of the Levi-Civita connection forms for $(X_i)_i$ and recall the notation
		$\Psi= \Psi^i \otimes X_i=\Psi^i_j \cdot \theta^j \otimes X_i.
		$.
		Also define \[\upsi:= \begin{pmatrix}
		\Psi^1 \\
		\Psi^2\\
		\dots\\
		\Psi^n
		\end{pmatrix} \quad \text{and}\quad \utheta:= \begin{pmatrix}
		\theta^1 \\
		\theta^2\\
		\dots\\
		\theta^n
		\end{pmatrix} .\]

		The proof in Proposition \ref{Teo unicità} suggests to define $\omega$ as
		\[		
		\omega=
		\begin{pmatrix}
		{\scalebox{2} {$\Theta$} }& \begin{matrix}
		-\Psi^1 & -i\theta^1 \\
		\dots &\dots \\
		-\Psi^n & -i\theta^n
		\end{matrix} \\
		\begin{matrix}
		\Psi^1 & \dots & \Psi^n\\
		i\theta^1 & \dots & i\theta^n
		\end{matrix} & 
		\begin{matrix}0\text{    } & \text{    } 0 \\ 
		0\text{    } & \text{    } 0
		\end{matrix}
		\end{pmatrix}  
		\]
		
		We want to prove that $d\omega + \omega \wedge \omega=0$. \newline
		
		By the explicit computation, we get that 
		\[
		d\omega + \omega \wedge \omega =
		\begin{pmatrix}
		d\Theta + \Theta\wedge \Theta - \upsi \wedge ^t\upsi+ \utheta \wedge ^t\utheta & -d\upsi - \Theta \wedge \upsi & -i d\utheta -i \Theta \wedge \utheta \\
	d\upsi + \Theta \wedge \upsi	 &  0 & -i^t\upsi\wedge \underline \theta \\
		i ^td\utheta+ i ^t \utheta \wedge \Theta & -i ^t\utheta \wedge \upsi & 0
		\end{pmatrix}
		\]
		We now observe the following.
		\begin{itemize}
			\item By definition of the Levi-Civita connection forms,
			\[
			d\theta^j +\theta^i_j\wedge \theta^j=0,\]
			hence $d\utheta+ \Theta\wedge \utheta=0$.
			\item Since $\Psi$ is symmetric, we have 
			\[^t\upsi\wedge \utheta= \sum_{j,k} \Psi^j_k \theta^j\wedge \theta^k=0.
			\]
			\item Expanding $d^{\nabla} \Psi$, one gets
			\begin{align*}
			g\big((d^\nabla \Psi) (X_i, X_j), X_h\big) =& \big(\nabla_{X_i} (\Psi^k(X_j) X_k), X_h\big)-\\
			&-\big(\nabla_{X_j} \Psi^k(X_i)X_k, X_h\big) - \Psi^h([X_i, X_j])=\\
			=& X_i (\Psi^h (X_j)) + \Psi^k (X_j)\theta_k^h (X_i)-\\
			&- X_j(\Psi^h(X_i)) - \Psi^k (X_i) \theta_k ^h (X_j) - \Psi^h([X_i, X_j])=\\
			=& (d\Psi^h)(X_i, X_j) +  (\Theta\wedge \upsi)^{(h)} (X_i, X_j).
			\end{align*}
			which leads to the standard formula
			\[
			\underline {d^\nabla \Psi}= d\upsi +  \Theta \wedge \upsi,
			\]
			(see also \cite{Kobayashi-Nomizu 1}). Thus, by Codazzi equation,
			\[
			d\upsi +  \Theta \wedge \upsi=0
			\]

			\item  By a straightforward computation \[ d\theta_j^i+ \sum_t \theta_k^i\wedge \theta^k_j =R(X_i, X_j, \cdot, \cdot).\]
		Now use equation (\ref{Codazzi inversa}) to get
			\[d\theta_j^i + \sum_k \theta^i_k\wedge \theta_j^k - \psi^i \wedge \psi^j+ \theta^i \wedge \theta^j=0 \]
			 for all $i, j=1,\dots, n$, i.e. \[d\Theta +\Theta \wedge \Theta-\upsi\wedge^t\upsi + \utheta\wedge^t\utheta=0.\]

		\end{itemize}
		We can finally conclude that $d\omega + \omega \wedge \omega=0$. \newline

		By Lemma \ref{Lemma chiave unicità} there exists $\Phi\colon U \to \SOnn$ such that $\omega= \omega_\Phi= \Phi^* \omega_{G}$ where $\omega_{G}$ is the Maurer-Cartan form of $\SOnn$.
		
		Define $\sigma\colon M \to \ML$ as $\sigma(x):= \Phi(x)\cdot \underline e=i\Phi(x)\underline v_{n+2}^0$.
		
		Observe that
		\begin{align*}
		\sigma_{*x}(X_j) &= i \big( \Phi(\cdot) \underline v_{n+2}^0 \big)_*(X_j)=i \Phi_{*x}(X_j)\underline v_{n+2}^0=\\
		&=i \Phi(x)\omega(X_j) \underline v^0_{(n+2)}= \theta^k(X_j) \Phi(x) \underline v_k^0= \Phi(x)\underline v_j^0.
			\end{align*}
		We conclude that $\sigma$ is an immersion at every point and that it is an isometric immersion since its differential sends an orthonormal basis into orthonormal vectors. 
			
		Observe that this construction is in fact inverse to construction $(\ref{definizione Phi})$. 	\newline
		
		\emph{Step 2.} We extend the result for simply connected manifold $M$ with immersion data $(g,\Psi)$. We will say that an open subset of $M$ is \emph{immersible} if it admits an isometric immersion into $\ML$.
		
		Fix a point $x\in M$ and an immersible open neighbourhood $U_0$ of $x$ and an isometric immersion $\sigma_0 \colon U_0 \to \ML$. In fact, one can just fix the germ of an immersion.
		
		For all $y\in M$, let $\alpha\colon [0,1] \to M$ be a simple path connecting $x$ to $y$. Consider a collection of open subsets $\{U_i\}_{i=0}^m$ of $\ML$ with the property of being a \emph{good cover} for $\alpha([0,1])$, i.e. such that: \begin{itemize}
			\item[\_] $\alpha([0,1]) \subset \bigcup_{i=0}^m U_i$;
			\item[\_] $U_i$ is immersible for all $i$;
			\item[\_] $\alpha^{-1}(U_i)$ is a connected interval;
			\item[\_] $U_{i}\cap U_j\cap \alpha([0,1]) \ne \emptyset$ iff $|i-j|\le 1$ and $U_i\cap U_{j}$ is either empty or connected for all $i,j$.
		\end{itemize}
		
		By Corollary \ref{Corollario Teo unicità}, we can construct a unique family $\{\sigma_i \}_{i=0}^m$ such that $\sigma_i \colon U_i \to \ML$ is an isometric immersion and ${\sigma_i}_{|U_i \cap U_{i-1}}= {\sigma_{i-1}}_{|U_i \cap U_{i-1}}$ for all $i=1, \dots m$.
		
		We prove that the germ of $\sigma_m$ around $q$ does not depend on the choice of the good cover for $\alpha$.
		Let $\{U'_j\}_{j=0}^m$, with $U'_0=U_0$, be another good cover for $\alpha$ with associated set of immersions $\{\sigma'_j\}_{j=0}^p$, with $\sigma'_0=\sigma_0$. For all $i\in \{0, \dots, m\}$ and $j\in \{0, \dots, p \}$, $U_i\cap U'_j\cap \alpha([0,1])$ is either empty or connected, since it is the image of a connected interval; we define $C_{i,j}$ as the connected component of $U_i \cap U'_j$ which intersects $\alpha([0,1])$.
		
		By contradiction, assume $(i_0, j_0)\in \{0, \dots, m\}\times\{0,\dots, p\}$ be such that $U_{i_0}\cap U'_{j_0}\cap \alpha(I)\ne \emptyset$ and such that $\sigma_{i_0}$ and $\sigma'_{j_0}$ do not coincide over $C_{i_0, j_0}$; also take $(i_0, j_0)$ so that this does not hold for any other couple $(i, j)$ with $i\le i_0$ and $j\le j_0$. One can see that, for such $(i_0, j_0)$, either $U_{i_0 -1}\cap U_{i_0}\cap U'_{j_0}\ne \emptyset$ or $U'_{j_0-1}\cap U_{i_0}\cap U_{j_0}\ne \emptyset$; assume the former without loss of generality. Then, by minimality $\sigma_{i_0-1}$ and $\sigma'_{j_0}$ coincide over $C_{i_0-1, j_0}$, while by construction $\sigma_{i_0-1}$ and $\sigma_{i_0}$ coincide over $U_{i_0}\cap U_{i_0-1}$: as a result $\sigma_{i_0}$ and $\sigma'_{j_0}$ coincide on $C_{i_0,j_0} \cap U_{i_0}\cap U_{i_0-1}$ which is nonempty, hence they coincide on $C_{i_0, j_0}$, which is a contradiction.
		
		Finally, we use the simply-connectedness of $M$ to prove that the germ of $\sigma_m$ over $y$ does not depend on the path $\alpha$ either. Indeed, any two paths from $x$ to $y$ are linked by some homotopy $H\colon [0,1]\times [0,1]\to M$; by compactness, it is clear that if $\{U_i\}_{i=0}^m$ is a good cover for the path $\alpha_t:=H(t,\cdot)$, then there exists $\eps_t$ such that $\{U_i\}_i$ is a good cover for $\alpha_s$ for all $|s-t|<\eps$; as a result, the function that assigns to each time $t\in [0,1]$ the germ of the corresponding $\sigma_n^t$ in $q$ constructed via $\alpha_t$ is locally constant, hence constant.
		
		We can therefore extend $\sigma_0$ to an isometric immersion $\sigma\colon M \to \ML$.
	\end{proof}
	
	\begin{Remark}
		\begin{itemize}
			\item[\_]
By Theorem \ref{Teoremone}, every pseudo-Riemannian space form of constant curvature $-1$ of dimension $n$ admits an essentially unique isometric immersion into $\mathds{X}_{n+1}$.
\item[\_]
	Let $(M, g)$ be a manifold with a complex metric $g$ and a symmetric $(1,1)$ form $\Psi$ such that $(g,\Psi)$ satisfy the Gauss-Codazzi equation. 
	
	Consider its universal cover $(\widetilde M, \widetilde g)$, over which $\pi_1(M)$ acts by isometries, and the lifting $\widetilde \Psi$ of $\Psi$, which is $\pi_1(S)$-invariant; by the previous result, there exists an isometric immersion \[\sigma\colon (\widetilde M,\widetilde g) \to \PML\] with shape operatore $\widetilde \Psi$, unique up to an ambient isometry. It is now trivial to check that $\sigma$ is $(\pi_1 (M), \SOnn)$-equivariant. Indeed, for all $\alpha \in \pi_1 (M)$, $\sigma \circ \alpha$ is a new isometric embedding, hence there exists a unique $\psi =:mon(\alpha)\in \SOnn$ such that \[\sigma \circ \alpha= mon(\alpha) \circ \sigma.\]
	
	We will call such a pair $(g, \Psi)$ \emph{immersion data} for $S$.	
		\end{itemize}
	\end{Remark}

	\subsection{Totally geodesic hypersurfaces in $\PML$}
	
	A particular case of immersions $M\to \mathds X_{n+1}$ is given by totally geodesic immersions, namely immersions with $\Psi=0$. The study of this case leads to several interesting results.
	
	\begin{Lemma}
		If $g$ is a complex metric on a smooth manifold $M$ with constant sectional curvature $k\in \C$, then, for any $X,Y,Z,W\in \Gamma(\C TM)$,
		\[
		R(X,Y,Z,W)= k (\inner{X,Z}\inner{Y,W} - \inner{Y,Z}\inner{X,W}).
		\]
		In particular, $R(X,Y)Z\in Span_{\C}(X,Y)$.
	\end{Lemma}
	\begin{proof}The proof is exactly as in Lemma \ref{lemma space form}.
	\end{proof}
	
	\begin{Theorem}
		\label{teoremone stessa dimensione}        
		Let $M$ be a smooth manifold of dimension $n$.
		
		Then, $g$ is a complex metric for $M$ with constant sectional curvature $-1$ if and only if there exists an isometric immersion
		\[(\widetilde M, \widetilde g)\to \mathds{X}_n\] which is unique up to post-composition with elements in $Isom(\mathds X _{n})$ and therefore is $(\pi_1(M), O(n+1,\C) )$-equivariant.
	\end{Theorem}
	\begin{proof}
		Let $\iota \colon \mathds \C^{n+1}\hookrightarrow \mathds \C^{n+2}$ be the immersion
		$\iota(z_1,\dots z_{n+1})=(z_1,\dots, z_{n+1}, 0)$.
		
		Assume there exists an isometric immersion $\sigma\colon (M,g)\to \mathds X_n$, then $\overline \sigma=\iota \circ \sigma\colon (M,g) \to \mathds X_{n+1}$ is another immersion and it has $\overline \sigma^* \underline v^{n+2}_0$ as a global normal vector field. The induced shape operator is therefore \[\Psi=\overline \nabla (\overline \sigma^*\underline v^{n+2}_0)=0.\] By Gauss equation, this means that, for every local orthonormal frame $(X_i)_i$, $R(X_i, X_j, X_i, X_j)=-1$, i.e. $(M,g)$ has constant sectional curvature $-1$.

		Conversely, assume $(M,g)$ has constant sectional curvature $-1$. Then, by taking $\Psi\equiv 0$, the couple $(g,\Psi)$ trivially satisfies the Codazzi equation and, by the previous lemma, we have $R(X_i, X_j,\cdot, \cdot)=-\theta^1\wedge \theta^j$, so the Gauss equation holds as well.
		By Theorem \ref{Teoremone}, there exists an isometric $\pi_1(M)$-equivariant immersion $\overline \sigma\colon (M,g)\to \mathds{X}_{n+1}$. Let $\nu=\overline \sigma^* \nu_0$ be a normal local vector field w.r.t. $\overline\sigma$, with $\nu_0(x)\in T_{\sigma(x)}\mathds X_{n+1}$.
		As $\nu_0$ is unitary it turns out that $\dot\nu_0$ is orthogonal to $\nu_0$. On the other hand, differentiating $0=\inner{\nu_0(t),\gamma(t)}$ and using that $\nu_0$ is orthogonal to $\dot\gamma$, we deduce that $\dot\nu_0$ is orthogonal to the vector $\gamma(t)$. Thus $\dot\nu_0$ is contained in $\overline\sigma_*(\mathbb CTM)$, and 
		\[
		    \dot\nu_0=\sigma_*\Psi(\dot\gamma)=0.
		\]

We conclude that $\dot \nu_0(t)\equiv0$, hence $\nu_0$ is a constant vector and $Im(\overline\sigma)\subset \nu_0^\bot$. Up to composition with elements in $\SOnn$, we can assume $\nu_0= (0, \dots, 0, 1)$, so $Im \overline \sigma \subset \mathds X_n$.
		
		Finally, an isometric immersion of $(\widetilde M, \widetilde g)$ in $\mathds{X}_n$ is unique up to composition with elements in $\SOnn$ that stabilize $\mathds X_n$, namely up to elements in $O(n,\C)$.
	\end{proof}

	An interesting case we are going to treat in Section \ref{six} is the case $n=2$. In this setting, the previous theorem can be stated in the following way.
	\begin{Proposition}
	\label{surfaces in G first part}	
		Let $(S,g)$ be a surface equipped with a complex metric, denote with $(\widetilde S, \widetilde g)$ the universal covering. Then:
		\begin{itemize}
		 \item $(S,g)$ has constant curvature $-1$ if and only if there exists an isometric immersion 
		\[
		\sigma=(f_1,f_2) \colon (\widetilde S,\widetilde g)\to {\mathds G}= (\CP^1 \times \CP^1 \setminus \Delta, -\frac{4}{(z_1 - z_2)^2}dz_1 dz_2) 
		\]
		which is $(\pi_1 (S), Isom(\GGG))$-equivariant. In particular, $g$ induces a monodromy map 
		\[
		mon_g \colon \pi_1 (S) \to \PSL \times \Z_2
		\] defined up to conjugation. Being $\sigma$ admissible, the maps $f_j\colon \widetilde S \to \CP^1$ are such that $rk( (f_j)_*) \geq 1$.
		\item
		By composing with some affine chart $(U\times U\setminus \Delta, z\times z)$ of $\GGG$, a complex metric $g$ with constant curvature $-1$ can be locally expressed as
		\[
		g= -\frac{4}{(f_1 - f_2)^2}df_1 df_2.
		\]
		\item The maps $f_1$, $f_2$ are local diffeomorphisms if and only if no real vector $\XX \in TS\setminus \{0\}$ is isotropic for $g= \sigma^*\inners= (f_1,f_2)^*\inners$.
				\end{itemize}
	\end{Proposition}
	\begin{proof}
		We only need to prove the last part of the proposition.
		There exists $\XX\in T_x S$ such that $g(\XX,\XX)=0$ if and only if $df_1(\XX)\cdot df_2(\XX)=0$ for some $v$, which holds if and only if one between $f_1$ and $f_2$ is not a local diffeomorphism.
	\end{proof}
	 \begin{Example}
	 	The hyperbolic plane $\Hy^2$ in the upper half-plane model admits the isometric immersion $z \mapsto (z,\overline z)$ as described in section 2.
	 	
	 	Consider the immersion of $S^2$ given by \begin{align*}
	 	(f_1,f_2)\colon S^2 \approx \overline \C &\to \GGG\\
	 			z &\mapsto (z,-\frac{1}{\overline z})
	 	\end{align*}
	 	which in fact embeds $S^2$ into the graph of the antipodal map. The pull-back metric is given by
	 	\[(f_1,f_2)^* \inners= - \frac{4}{(1+ |z|^2)^2} dzd\overline{z}
	 	 \]
	 	which coincides with the negative definite space form of curvature $-1$, namely $-S^2$, the sphere equipped with the opposite of the standard elliptic metric.
	 	
	 	Another example is given by the universal cover of the Anti-de Sitter plane $\widetilde {AdS^2}$ which is isometric to each of the two connected components of $\R\times \R \setminus \Delta\subset \GGG$. 
	 \end{Example}

\subsection{Connections with immersions into pseudo-Riemannian space forms}

For all $n,p\in \mathbb N$ with $n\ge 2$ and $n\ge p \ge 0$, denote with $\mathbb H^{n-p,p}$ the $n$-dimensional pseudo-Riemannian space form of signature $(n-p,p)$ and constant sectional curvature $-1$. One has:
\begin{align*}
&\mathbb H^{n,0}\cong \mathbb H^n  && \text{the hyperbolic space}\\
&\mathbb H^{n-1,1}\cong \widetilde{AdS^n} && \text{the universal Anti-de Sitter space}\\
&\mathbb H^{1, n-1}\cong -\widetilde{dS^n} && \text{the de Sitter space with opposite metric ($\widetilde {dS^n}=dS^n$ for $n\ge3$)}\\
&\mathbb H^{0, n}\cong -S^n && \text{the Riemannian sphere with opposite metric}.
\end{align*}

It is well known that $\mathbb H^{n-p,p}$ is isometric to the universal covering of 
\[
Q^{n-p,p}= \{(x_1, \dots, x_{n+1})\ |\ x_1^2 + \dots + x_{n-p}^2 - x_{n-p+1}^2 - \dots - x_{n+1}^2 =-1  \}\subset \R^{n-p,p+1}
\]
where $\R^{n-p,p+1}$ is the Minkowski space of signature $(n-p, p+1)$. 

By Theorem \ref{teoremone stessa dimensione}, one has a unique isometric immersion of $\mathbb{H}^{n-p,p}$ $\to\mathds X_{n}$ up to composition with ambient isometries. In fact, one can explicitly check that the embedding
\begin{align*}
Q^{n-p, p} &\hookrightarrow \mathds X_{n+1} \subset \C^{n+1}\\
(x_1, \dots, x_{n+1}) &\mapsto (x_1, \dots, x_{n-p}, ix_{n-p+1}, \dots, ix_{n+1})
\end{align*}
	is isometric, so its lifting to the universal covering provides an isometric immersion of $\mathbb H^{n-p,p}$ into $\mathds X_{n+1}$.
	
	There exists a general theory of immersions of hypersurfaces into $\mathbb H^{n-p,p}$, see for instance \cite{Gauss-Codazzi in space forms}. One can define an immersion $M^{n-1}\to \mathbb H^{n-p, p}$ to be \emph{admissible} if the pull-back metric is a (non-degenerate) pseudo-Riemannian metric. By composition with the isometric immersion $\iota\colon \mathbb H^{n-p,p}\to \mathds X_n$, one can see immersions into $\mathbb H^{n-p,p}$ as immersions into $\mathds X_n$ and the extrinsic geometry of the latter extends the one of the former. 
	
	Indeed:
	\begin{itemize}
\item The map $\iota_*$ induces a canonical bundle inclusion of $\sigma^* T \mathbb H^{n-p,p}$ into $\bar\sigma^*T \mathds X_{n}$: indeed, since $Span_\C (\iota_*(T_{\underline x} \mathbb H^{n-p,p}) )= T_{\iota(\underline x)} \mathds X_n$, one has
\[
\sigma^* T \mathbb H^{n-p,p} \otimes_\R \C \cong \bar \sigma^* T \mathds X_n.
\] Since $\iota$ is isometric, the pull-back Levi-Civita connections coincide on $\sigma^* T \mathbb H^{n-p,p}$.
\item If $\bar\sigma^*\inners$ has signature $(n-p-1, p)$, then one can define a local normal vector field $\nu$ as a local section of $\sigma^* T \mathbb H^{n-p,p}$ with $\inner{\nu, \nu}=1$; we denote this case as case $a)$. Conversely, if $\bar\sigma^*\inners$ has signature $(n-p, p-1)$, then any local section of $\sigma^* T \mathbb H^{n-p,p}$ orthogonal to $TM$ is timelike, so one can define a local normal vector field $\nu$ as a local section of $\sigma^* T\mathbb H^{n-p,p}$ with $\inner{\nu, \nu}=-1$; denote this case as case $b)$. 

In case $a)$, $\nu$ is also a normal vector field for $\bar \sigma$ as we defined in section 3; in case $b)$, $i\nu$ is norm-$1$ and is a normal vector field for $\bar \sigma$.
\item The exterior derivative of the local normal vector field defines a shape operator $\psi$, which coincides in case $a)$ with the shape operator $\Psi$ induced by $\bar \sigma$, while in case $b)$ one has $\Psi=i\psi$.

\item For immersions into $\mathbb H^{n-p,p}$ there exists a Bonnet theorem analogous to the one we proved for immersions into $\mathds X_n$.
\begin{Proposition}
	Let $h$ be a pseudo-Riemannian metric on $M$ with signature $(n-p-1 + \delta, p -\delta)$ with $\delta\in \{0,1\}$. Let $\psi\colon TM\to TM$ be a $h$-self adjoint $(1,1)$-form. The data $(h,\psi)$ satisfy
	\begin{align*}
	1) &d^{\nabla_h} \psi \equiv 0;\\
	2)	&R(X_i, X_j, \cdot, \cdot) - (2\delta -1)\psi^i \wedge \psi^j= -\theta^i\wedge \theta^j\\
	&\text{for every local $h$-frame $(X_i)_{i=1}^n$ with corresponding}\\ &\text{ coframe $(\theta^i)_{i=1}^n$ and with $\psi= \psi^i \otimes X_i$}.
	\end{align*}
	if and only if there exists a $\pi_1(S)$-equivariant isometric immersion $\sigma\colon (\widetilde M,\widetilde h)\to \mathbb H^{n-p, p}$ with shape operator $\psi$.
\end{Proposition}
If $\sigma\colon M\to \mathbb H^{n-p, p}$ is an admissible immersion with data $(h,\psi)$, then $\iota \circ \sigma\colon M\to \mathds X_n$ is an admissible immersion with immersion data $(g,\Psi)=(h,(i\delta + 1-\delta)\psi)$. By the other hand, the data $(g, \Psi)$ uniquely determines the immersion of $\widetilde M$ into $\mathds X_n$ up to ambient isometry. Since also $\iota$ is the unique isometric immersion of $\mathbb H^{n,p-n}$ into $\mathds X_n$ up to ambient isometry, one can conclude the following.
		\end{itemize}
	\begin{Theorem}
		\label{da X_n a space forms}
		Let $M=M^{n-1}$, $(g,\Psi)$ be immersion data for a $\pi_1(M)$-equivariant immersion of $\widetilde M$ into $\mathds X_{n}$.
		Assume that $g$ is real, namely that $g_{|TM}$ is pseudo-Riemannian, and has signature $(n-1-p,p)$, $0\le p\le n-1$. 
		
		Then, if $\Psi$ is real, i.e. if $\Psi$ restricts to a bundle homomorphism $\Psi\colon TM\to TM$, there exists an isometric $\pi_1(S)$-equivariant immersion $\sigma\colon \widetilde M\to \mathds X_n$ such that $\sigma( \widetilde M)\subset \iota(\mathbb H^{n-p,p})$.
		
				Similarly, if $i\Psi$ is real, then there exists an isometric $\pi_1(M)$-equivariant immersion $\sigma\colon \widetilde M\to \mathds X_n$ such that $\sigma( \widetilde M)\subset \iota(\mathbb H^{n-p-1,p+1})$.
	\end{Theorem}
	
	\subsection{From immersions into $\Hy^3$ to immersions into $\GG$}

	Given $\sigma\colon \widetilde S\to \Hy^3$ a
$\pi_1(S)$-equivariant immersion with immersion data $(h, \psi)$ and normal field $\nu$, one can define the immersion 
\[
\overline\sigma\colon \widetilde S\to \GG
\]
where $\overline \sigma(x)$ is the oriented maximal geodesic of $\Hy^3$ tangent to $\nu(x)$. In the identification $\GG=\CP^1 \times \CP^1\setminus \Delta$, one has $\overline \sigma= (\sigma_{+\infty}, \sigma_{-\infty})$ corresponding to the endpoints of the geodesic rays starting at $\sigma(x)$ with tangent directions respectively $\nu(x)$ and $-\nu(x)$. 
 	
The map $\overline \sigma$ is $\pi_1(S)$-equivariant with the same monodromy as $\sigma$. 

We want to prove the following formula for the pull-back metric for $\overline \sigma$.
\begin{Prop}
	\label{Prop metrica in G}
	Let $J$ be the complex structure induced by $h$ and inducing the same orientation as $\nu$. Under the notation above, the pull-back $g=\overline \sigma^* \inners$ of the metric is the $\pi_1(S)$-invariant complex bilinear form on $\CTS$ given by
	\[g = h((id + iJ\psi)\cdot, (id + iJ\psi)\cdot )= h- h(\psi\cdot, \psi\cdot)+ i h( (\psi J - J\psi)\cdot, \cdot)
	\]
	which is non degenerate (i.e. a complex metric) in $x$ if and only if $K_h (x)\ne 0$ 
\end{Prop} 	

In order to prove the proposition, we regard $\overline\sigma$ as the 
composition of the normal section $\nu:\widetilde S\to T^1\Hy^3$ with the 
natural projection $\Pi:T^1\Hy^3\to\GG$, whose fibers are the leaves of 
the geodesic flow.

First of all let us recall that the tangent space $T_{(p,v)}(T^1\Hy^3)$ 
is naturally identified with $T_p\Hy^3\oplus v^\perp$, where $v^\perp$ is 
the orthogonal space to $v$ in $T_p\Hy^3$. The identification works as 
follows: given a path $\alpha:(-\epsilon, \epsilon)\to T^1\Hy^3$, which can be written as $\alpha(t)=(p(t),v(t))$ with $v(t)$ being a unit vector field along the path $p(t)$, the identification is given by 
\[\dot\alpha(0) = (\dot p(0),\frac{Dv}{dt}(0)).\] Notice that, since $v$ is unitary, by differentiating $g_{\Hy^3}(v(t), 
v(t))=1$ one gets that the vector $\frac{Dv}{dt}(0)$ is orthogonal to 
$v$, so the correspondence is well-posed.

This identification allows us to get a simple expression for the differential of the normal section $\nu:\widetilde S\to T^1\Hy^3$. 

\begin{Lemma}
\label{Lemma 1 metrica pull back G}
    Up to the above identification, for any $x\in\widetilde S$ and $v\in T_x\widetilde S$ we have 
    \[
    (d_x\nu)(v)=(v,\psi(v))\,.
    \]
\end{Lemma}
\begin{proof}
The proof is trivial by definition of the shape operator.
\end{proof}
\begin{Lemma}
\label{Lemma 2 metrica pull back G}
  Let us fix $(p,v)\in T^1\Hy^3$ and  $w_1,w_2\in v^\perp$. Then
    \[
    \inner{d_{(p,v)}\Pi(w_1,w_2), d_{(p,v)}\Pi(w_1,w_2)}_{\GG}=g_{\Hy^3}(w_1, w_1)-g_{\Hy^3}(w_2,w_2)+i (g_{\Hy^3}(w_1, v\times w_2)-g_{\Hy^3}(w_2, v\times w_1)),
    \]
    where $\times$ is the vector product on $T\Hy^3$.  
\end{Lemma}
\begin{proof}
We consider the half-space model of $\Hy^3=\C\times\R_{+}$. Set 
$p=(0,1)$ and, in the identification $T_p\Hy^3=\C\times\R$, set $v=(1,0)$. As a result, $v^\perp= Span_\R((i,0),(0,1)) <T_p \Hy^3$, we can canonically see $\partial \Hy^3=\overline \C$ and $\Pi(p,v)$ is the geodesic with endpoints $\Pi_+(p,v)=1$ and $\Pi_-(p,v)=-1$. The tangent space $T_{(1,-1)} \GG$ can be trivially identified with $T_1 \C\times T_{-1}\C\cong \C \times \C$.

Let us consider the following $1$-parameter groups of isometries of $\Hy^3$
\[
   \begin{matrix}
   a(t)= \exp(tX), & b(t)=\exp(itX),\\
   c(t)=\exp(tY), & d(t)=\exp(itY)\,.
   \end{matrix}
\]
where $X,Y\in\asl$ are defined as  $X=\begin{pmatrix}1/2 & 0\\0&-1/2\end{pmatrix}$, and $Y=\begin{pmatrix}0 & -i/2\\i/2 & 0\end{pmatrix}$.

Notice that $a(t)$ and $c(t)$ are groups of hyperbolic transformations with axis respectively $(0,\infty)$ and $(i,-i)$. On the other hand, $b(t)$ and $d(t)$ are pure rotations around the corresponding axes. We have chosen the normalization so that the translation lengths of $a(t)$ and $c(t)$ equal $t$, and so that the rotation angle of $b(t)$ and $d(t)$ is $t$.

Observe that $p$ lies on all the axes of  $a(t), b(t), c(t), d(t)$.
It follows that $t\to a(t)\cdot v$ is a parallel vector field along the axis of $a(t)$, so  the derivative of $a(t)\cdot(p, v)$ corresponds under the natural identification to the vector $((0,1), (0,0))$. On the other hand, the variation of the endpoints of the family of geodesics $a(t)\cdot (-1,1)\in \GG$ is given by $(-1,1)\in T_{(-1,1)}\GG$. Using that the map $\Pi$ is equivariant under the action of $\PSL$, we conclude that
\[
   d_{(p,v)}\Pi((0,1),(0,0))=(-1,1)\,.
\]

In the same fashion, using $c(t)$ we deduce that
\[
   d_{(p,v)}\Pi((-i,0), (0,0))=(-i,-i)\,.
\]

On the other side, $b(t)p=p$ for all $t$, therefore one can explcitly compute that
\[b(t)\cdot v=d_pb(t)v=\cos t v+\sin t (0,1)\times v\in T_{p} \Hy^3.\] It follows that the derivative at $t=0$ of $b(t)\cdot(p,v)$ corresponds to $((0,0),(i,0))$.
We conclude as above that
\[
 d_{(p,v)}\Pi((0,0), (i,0))=(-i,i)
\]
and analogously for $d(t)$ we get
\[
d_{(p,v)}\Pi((0,0), (0,1))=(1,1).
\]
Finally, we can explcitly compute $d_{(p,v)}\Pi_{|v^\perp}\colon v^\perp \cong i\R \times \R_+ \to T_{(1,-1)}\GG\cong \C \times \C $:
\[
  d\Pi((i\alpha,\beta), (i\gamma, \delta))=((\delta-\beta)+i(\alpha-\gamma),(\delta+\beta)+i(\alpha+\gamma)).
\]
Using the description of the metric as in Proposition \ref{metrica su G}, we get that
\begin{align*}
 ||d\Pi((i\alpha,\beta), (i\gamma, \delta))||^2&=
 -4\frac{[(\delta-\beta)+i(\alpha-\gamma)][(\delta+\beta)+i(\alpha+\gamma))]}{(1-(-1))^2}=\\
&= \alpha^2+\beta^2-\gamma^2-\delta^2 - 2i(\alpha\delta-\beta\gamma)=\\
&= ||(i\alpha,\beta)||^2-||(i\gamma, \delta))|| +2i g_{\Hy^3}((i\alpha,\beta), (1,0)\times (i\gamma, \delta)))
\end{align*}
and the thesis follows.
\end{proof}

\begin{proof} [Proof of Proposition \ref{Prop metrica in G}]
The proof follows directly by Lemmas \ref{Lemma 1 metrica pull back G} and \ref{Lemma 2 metrica pull back G}.

	If $(X_1,X_2)$ is a $h$-orthonormal frame of eigenvectors for $\psi$, with corresponding eigenvalues $\lambda_1$ and $\lambda_2$ respectively, the pull-back bilinear form via $\overline \sigma$ is described by \[
	g \leftrightarrow \begin{pmatrix}1- \lambda^2_1 & i(\lambda_1 -\lambda_2)\\
	i(\lambda_1-\lambda_2) &1-\lambda^2_2
	\end{pmatrix}
	\]	
	whose determinant is $(1-\lambda_1\lambda_2)^2$: hence, by Gauss equation, $g$ is a complex metric at $x$ if and only if $K_h(x)\ne0$.
\end{proof}

	\section{On holomorphic dependence on the initial data}
	
	\subsection{The main result}
	In this section we discuss  holomorphic dependence on the immersion data for immersions into $\mathds X_{n+1}$ and their monodromy. 
	
	Given a smooth manifold $M$  of dimension $n$ and any point $x\in M$, $\C T_x M$ is a complex vector space and provides a natural complex structure to the manifolds $\C T_x ^*M$, $Sym^2 (\C T_x^*M)$ and $End(\C T_x M)=\C T_x M \otimes \C T_x^* M$. 
	
	Given a complex manifold $\Lambda$, we will say that a family of immersion data $\{(g_\lambda, \psi_\lambda)\}_{\lambda\in \Lambda}$ for $\pi_1(M)$-equivariant immersions of $\widetilde M$ into $\mathds X_{n+1}$ is \emph{holomorphic} if, for all $x\in M$, the maps
	\begin{align*}
	\Lambda &\to Sym^2 (\C T_x^*M) \\
	\lambda &\mapsto g_\lambda(x) 
	\end{align*}
	and
	\begin{align*}
	\Lambda &\to End(\C T_x M)\\
	\lambda &\mapsto \psi_\lambda (x)
	\end{align*}
	are both holomorphic. We remark that this definition does not require any holomorphic structure on $M$.

	This section is devoted to the proof of the following theorem.
	
	\begin{Theorem}
		\label{Teo dipendenza olomorfa}
		Let $\Lambda$ be a complex manifold and $M$ be a smooth manifold of dimension $n$. 
		
		Let $\{(g_\lambda, \Psi_\lambda)\}_{\lambda\in \Lambda}$ be a holomorphic family of immersion data for $\pi_1(M)$-equivariant immersions $\widetilde M\to \mathds X_{n+1}$. Then there exists a smooth map
		\[
		\sigma\colon \Lambda \times \widetilde M \to \mathds X_{n+1}
		\]
		such that, for all $\lambda\in \Lambda$ and $x\in M$:
		\begin{itemize}
			\item $\sigma_\lambda:= \sigma(\lambda, \cdot)\colon \widetilde M \to \mathds X_{n+1}$ is an admissible immersion with immersion data $(g_\lambda, \Psi_\lambda)$;
			\item $\sigma(\cdot,x)\colon \Lambda \to \mathds X_{n+1}$ is holomorphic.
		\end{itemize}
		Moreover, defined the character variety
		\[ \mathcal X(\pi_1 (M), SO(n+2, \C )= {Hom\big(\pi_1(S), \PSL\big)} // {\PSL},\]
		the monodromy map
		\begin{align*}
		\Lambda &\to \mathcal X (\pi_1(M), SO(n+2, \C) )\\
		\lambda &\mapsto mon(\sigma_\lambda)
		\end{align*}
	is holomorphic.
	\end{Theorem}
	
	\begin{Example}
		Let $h$ be a hyperbolic metric on a closed surface $S$ and let $b\colon TS\to TS$ be a $h$-self-adjoint (1,1)-form such that $d^{\nabla_h}b=0$ and $det(b)=1$; one may choose for instance $b=id$. Then, the family $\{(g_z,\psi_z)\}_{z\in \C}$ defined by
		\[
		\begin{cases}
		g_z= \cosh^2(z) \widetilde h; \\
		\psi_z= -\tanh(z) \widetilde b
		\end{cases}
		\] 
		is a holomorphic family of $\pi_1(S)$-equivariant immersion data for constant curvature immersions $\widetilde S\to \SL$, with $K_{g_z} =- \frac{1}{\cosh(z)^2}$. Observe that $z\in \R$ corresponds to an immersion data into $\Hy^3$, while $z\in i\R$ corresponds to an immersion data into $AdS^3$. 
		
		By Theorem \ref{Teo dipendenza olomorfa}, there exists a family of immersions $\sigma_z\colon \widetilde S\to \SL$ with data $(g_z, \psi_z)$ whose monodromy is a holomorphic function in $z$. 
	\end{Example}

	\subsection{An application: the complex landslide is holomorphic}
	We postpone the proof of Theorem \ref{Teo dipendenza olomorfa} to Sections 5.3 and 5.4. 
	
	In this section, we would like to provide a direct application of Theorem \ref{Teo dipendenza olomorfa}: we show an alternative proof for the fact that the holonomy of the complex landslide, defined in \cite{cyclic}, is holomorphic. 
	
	We briefly recall some basic notions on projective structures, the notions of landslide, smooth grafting and complex landslide. One can use as main references \cite{Dumas:complex projective structures} and \cite{cyclic}.
	\vspace{5mm}
	
	Let $S$ be an oriented closed surface of genus $g\ge 2$.
	
	 Denote the character variety into $\PSL$ by \[\mathcal X (S) := {Hom\big(\pi_1(S), \PSL\big)}\ //\ {\PSL}.\]
	 
	We recall that a \emph{(complex) projective structure} on $S$ is a $(\PSL, \CP^1)$-structure on $S$. A projective structure induces a complex structure and an orientation, we will stick to projective structures compatible with the orientation on $S$. As a $(\PSL, \CP^1)$-structure, a projective structure is determined by a $(\pi_1(S), \PSL)$-equivariant developing map $\widetilde S \to \CP^1$, which is unique up to post-composition with elements in $\PSL$. We therefore define
	
		\begin{equation*}
		\begin{split}
	 \widetilde{\mathcal P}(S)&:= \{\text{projective structures on $S$}\}=\\
	& =\faktor{\left \{(f, \rho)\ \bigg|\ \begin{split}
		&\rho\colon \pi_1(S)\to \PSL,\\
		 &f\colon \widetilde S \to \CP^1 \text{ $\rho$-equiv. orientation-preserving local diffeo} 
		\end{split}
		\right \}}{\PSL} \\
	 \mathcal P(S) &:=\faktor{\widetilde{\mathcal P} (S)}{\text{Diff}_0(S)}
	\end{split}
	\end{equation*}
	
	Equip the set of pairs $\{(f,\rho)\}$ with the compact-open topology, and $\widetilde {\mathcal P}(S)$ and $\mathcal P(S)$ with the quotient topology. The space $\mathcal P(S)$ is a topological manifold and one can define on it a smooth and complex structures - compatible with the topology - via the Schwarzian parametrization (e.g. see \cite{Dumas:complex projective structures}).
	
	Also define the holonomy map
	\[
	\widetilde {Hol} \colon \widetilde{\mathcal P}(S)\to \mathcal X(S),\]
	which passes to the quotient to a map
	\[Hol\colon 	{\mathcal P}(S) \to 	\mathcal{X}(S).\]
\begin{Theorem} (Hejhal \cite{Hejhal}, Hubbard \cite{Hubbard},  Earle \cite{Earle})
	\label{Hol è diffeo}
	The holonomy map $Hol$ is a local biholomorphism. 
\end{Theorem}

	\vspace{3mm}

	 Let $\Tau(S)$ be the Teichm\"uller space of $S$. 
	
	In \cite{cyclic}, the authors define two smooth functions satisfying several interesting properties: the \emph{Landslide map}
		\[
		L\colon \Tau(S)\times \Tau(S)\times \R \to \Tau(S) \times \Tau(S)
		\]
	and the \emph{Smooth Grafting map}
	\[
	SGr\colon \Tau(S) \times \Tau(S)\times \R^+ \to \mathcal P(S)
	\]
	which turn out to have several interesting geometric properties and are related respectively to the earthquake map and to the grafting map.
	
	There are several equivalent definitions of $L$ and $SGr$, we mention the ones most fitted to our pourposes.

Let $h$ be a hyperbolic metric. As a consequence of some works by Schoen \cite{Schoen} and Labourie \cite{Labourie}, it has been proved that the map
	\begin{align*}
	\left\{
	\begin{array}{l}
	b\colon TS\to TS\\
	\text{bundle isomorphism}\\
	\text{such that:}
	\end{array}\Bigg|
	 \begin{array}{l}
 \text{ $h$-self adjoint,} \\
d^{\nabla_h} b=0, \\
 det(b)=1
		\end{array} \right\}
		&\xrightarrow{\sim}  \Tau(S) \\
		b &\mapsto [h(b\cdot, b\cdot)]
	\end{align*} 
	is well-posed and bijective. We will denote the tensors in the domain as \emph{$h$-regular}.
	
\begin{itemize}
	\item Given $([h], [h(b\cdot, b\cdot)])\in \Tau(S)\times \Tau(S)$ with $b$ $h$-regular, define for all $t\in \R$ the operator $\beta_t= \cos(\frac t 2)id + \sin(\frac t 2)J b$ where $J$ is the complex structure induced by $h$. The landslide map is defined as
	\[
	L([h], [h(b\cdot, b\cdot)], t)=
	\bigg( \Big[h\Big(\beta_t\ \cdot, \beta_t \cdot \Big)\Big], \Big[h\Big(\beta_{t+\pi}\ \cdot, \beta_{t+\pi}\ \cdot \Big)\Big]  \bigg)\]
	\item Given $([h], [h(b\cdot, b\cdot)])\in \Tau(S)\times \Tau(S)$, the data $(h, b)$ satisfy Gauss-Codazzi equations for $\Hy^3$ and so does the data \[(h_s, b_s):=(\cosh^2(\frac s 2) h,  \tanh(\frac s 2) b )\] for all $s\in \R$. As a result, for all $s$ there exists a unique $\pi_1 (S)$-equivariant isometric immersion
	\[
	\sigma_s\colon (\widetilde S, \widetilde h_s) \to \Hy^3 
	\]
	with shape operator $b_s$ (here the notation is $\sigma^* 
	(D^{\Hy^3}) - \nabla^{h_s}= \IIs_{\sigma_s}\nu_s= - h_s (b_s \cdot, \cdot) \nu_s$, where $\IIs$ is the second fundamental form of $\sigma_s$ and $\nu_s$ is the normal vector field). The tensor $b_s$ having positive determinant for all $s\ne 0$, $\sigma_s$ is convex for all $s\ne 0$ and the normal is oriented towards the concave side for $s>0$. As a result, for $s>0$ the map $\sigma_s$ induces a local diffeomorphism
	\[
	\sigma_{s, +\infty}\colon \widetilde S \to \partial \Hy^3 \cong \CP^1
	\]
	defined by $\sigma_{s, +\infty} (x)$ being the endpoint of the geodesic ray starting at $\sigma(x)$ with tangent direction $\nu(x)$. The map $\sigma_{s, +\infty}$ is equivariant inducing the same monodromy into $\PSL$ as $\sigma_s$, thus it defines a projective structure on $S$. The smooth grafting is defined by
	\[
	Sgr([h], [h(b\cdot, b\cdot)], s)= [\sigma_{s, +\infty}].
	\]
\end{itemize}
In the same article, for every hyperbolic metric $h$ and $h$-regular tensor $b$, they define
\begin{align}
P_{h,b}\colon H&\to \mathcal P(S)\\
z=t+is &\mapsto Sgr\bigg( L\Big([h], [h(b\cdot, b\cdot)], -t\Big), s \bigg)
\end{align}
where $H\subset \C$ is the upper half-plane. 

\begin{Theorem}[Bonsante - Mondello - Schlenker]
	$P_{h,b}$ is holomorphic.
\end{Theorem}
\vspace{5mm}
The argument for the proof in \cite{cyclic} uses analytic methods. Here we provide an alternative proof using Theorem \ref{Teo dipendenza olomorfa} and Proposition \ref{Prop metrica in G}.

	\begin{Theorem} 
		\label{Teo alternativo landslide}
	The holonomy of the projective structure $P_{h,b}(z)$ is equal to the monodromy of the complex metric 
	\[
	g_z = h \bigg( (\cos(z)id -  \sin(z) Jb)\ \cdot,  (\cos(z)id -  \sin(z) Jb)\cdot \bigg)
	\]
	which has constant curvature $-1$.
	
	As a consequence of Theorem \ref{Teo dipendenza olomorfa} and Theorem \ref{Hol è diffeo}, $P_{h,b}$ is holomorphic.
\end{Theorem}

\begin{proof}[Proof of Theorem \ref{Teo alternativo landslide}]
Let $z=t+is$ and recall $P_{h,b}(z)= Sgr\bigg( L\Big([h], [h(b\cdot, b\cdot)], -t\Big), s \bigg)$.

Defining $\bar h_x:= h\Big(\beta_x\ \cdot, \beta_x \cdot \Big)$ for all $x\in \R$, one has
  \[L\Big([h], [h(b\cdot, b\cdot)], -t\Big)= \bigg([\bar h_{-t}], [\bar h_{-t+\pi}] \bigg). \] 
It is easy to convince oneself that  
\[\bar h_{-t+\pi} = \bar h_{-t} (\beta_{t}b \beta_{-t} \cdot, \beta_{t}b \beta_{-t} \cdot)=: \bar h_{-t}(\bar b_{-t}, \bar b_{-t}) \] 
and that the complex structure induced by $\bar h_{-t}$ is 
\[
\bar J_{-t}= \beta_{t}J\beta_{-t}.
\]
In order to compute the monodromy of $Sgr([\bar h_{-t}], [\bar h_{-t+\pi}], s)$, consider the immersion into $\Hy^3$ with immersion data $(cosh^2(s)\bar h_{-t}, \tanh(s)\bar b_{-t})$ and apply Proposition \ref{Prop metrica in G} to conclude that it has the same monodromy as the metric
\begin{align*}
	&\cosh^2(s) \bar h_{-t}\bigg( (id +i \tanh(s)\bar J_{-t} \bar b_{-t})\cdot, (id +i \tanh(s)\bar J_{-t} \bar b_{-t})\cdot  \bigg)=\\
	=&\bar h_{-t}\bigg( (\cosh(s) id +i \sinh(s) \beta_{t} J b \beta_{-t})\cdot, (\cosh(s) id +i \sinh(s) \beta_{t} J b \beta_{-t})\cdot  \bigg)= \\
	=& h \bigg( \big( (\cos(is) id - \sin(i s) J b \big)\circ\beta_{-t}\big )\cdot, \big( (\cos(is) id -i \sinh(s) J b \big)\circ\beta_{-t}\big )\cdot \bigg)= \\
	=&h \bigg( (\cos(is) id - \sin(i s) J b )(\cos(t) id - \sin(t) J b )\cdot, (\cos(is) id - \sin(i s) J b )(\cos(t) id - \sin(t) J b )\cdot \bigg)=\\
	=&h \bigg( (\cos(z) -\sin(z)Jb)\cdot, (\cos(z) -\sin(z)Jb)\cdot \bigg)=g_z
\end{align*}
where in the last passage we used - as one can check explicitly - that $JbJb=-id$.

By construction, $(g_z,0)$ is an immersion data for an immersion into $\GG$ for all $z$.
 
Clearly, for any $X, Y\in \Gamma(TS)$, \[g_z(X,Y)= \cos^2(z)h(X,X)- \sin^2(z) h(bX,bY)- \sin z \cos z ( h(X, JbY)+ h(JbX, Y ))\] is holomorphic in $z$. By Theorem \ref{Teo dipendenza olomorfa}, the monodromy of $g_z$, hence the holonomy of $P_{h,b}$, is holomorphic in $z$; by Theorem \ref{Hol è diffeo}, the projective structure $P_{h,b}(z)$ depends on $z$ holomorphicly.
\end{proof}

	\subsection{Some Lie Theory lemmas}
	
	We start the proof of Theorem \ref{Teo dipendenza olomorfa}.
	
	A great part of the proof of this result is a matter of integrating a distribution
	on a manifold. We start from a technical result which extends Lemma \ref{Lemma chiave unicità}.

	\begin{Proposition}
		\label{Teorema esistenza foliazione}
		Let $M$ and $\Lambda$ be two simply connected manifolds, $G$ a Lie group with Lie algebra $\Lieg$. 
		
		Consider a smooth family of forms $\{\omega_\lambda\}_{\lambda\in \Lambda}\subset \Omega^1(M, \Lieg)$, namely a smooth map $\Lambda\to \Omega^1 (M, \Lieg)$.
		
		The following are equivalent:
		\begin{itemize}
			\item for all $\lambda\in \Lambda$	\begin{equation}
			\label{eq MC parametrica}
			d\omega_\lambda + [\omega_\lambda, \omega_\lambda]=0;
			\end{equation}
			\item there exists a smooth map $\Phi \colon \Lambda \times M\to G$ such that, for all $\lambda\in \Lambda$, \begin{equation} 
			\label{eq pull-back MC}
			(\Phi(\lambda, \cdot))^* \omega_G= \omega_\lambda,
			\end{equation}
			where $\omega_G$ is the Maurer-Cartan form of $G$.
		\end{itemize}
		Moreover, both $\Phi$ and $\Phi'$ satisfy equation $\eqref{eq pull-back MC}$ if and only if
		\[
		\Phi' (\lambda, x)= \psi(\lambda) \cdot \Phi(\lambda,x)
		\]
		for some smooth $\psi\colon \Lambda\to G$.
		
		In other words, for every fixed $x_0\in M$, smooth $\psi_0\colon \Lambda \to G$ and for every smooth collection of $\Lieg$-valued 1-forms $\{\omega_\lambda\}_{\lambda\in \Lambda}$, there exists a unique $\Phi\colon \Lambda\times M\to G$ such that 
		\[
		\begin{cases}
		&\omega_\lambda= \Phi(\lambda, \cdot)^* \omega_G \\
		&\Phi(\lambda, x_0)=\psi_0(\lambda)
		\end{cases}
		\]
		
		For $\Lambda=\{pt\}$ one has Lemma \ref{Lemma chiave unicità}.
	\end{Proposition}

	\begin{proof}
		It is clear by the differential equation of the Maurer-Cartan form that the second statement implies the first one. We prove the opposite implication. 
		
		Endow the trivial bundles $\pi_{\Lambda\times M}\colon \Lambda\times M\times G\to \Lambda \times M$ and $\pi_M\colon  M\times G\to M$ with the natural left action of $G$ given by left translations on the last component.
		
		On the tangent bundle of $\Lambda\times M\times G$, let $D=\{D_{(\lambda,x,g)}\}_{(\lambda,x,g)}$ be the distribution
		\[
		D_{(\lambda, x,g)}=\{(\dot \lambda,\dot x, \dot g)\in T_{(\lambda, x,g)} (\Lambda \times M\times G) \ |\ \omega_\lambda(\dot x)= \omega_G(\dot g)  \}= T_{(\lambda,x,g)} \Lambda \oplus D^0_{(\lambda, x, g)} .
		\]
		where
		\[
		D^0_{(\lambda, x, g)} =\{(0,\dot x, \dot g)\in T_{(\lambda, x,g)} \Lambda \times M\times G \ |\ \omega_\lambda(\dot x)= \omega_G(\dot g)  \}.
		\]
		Both $D$ and $D^0$ are invariant under the left action of $G$ because the Maurer-Cartan form on $G$ is left invariant.
		
		\begin{itemize}
			\item [\emph{Step 1.}]
			The distribution $D^0$ is integrable. 
			
			For all $\lambda\in \Lambda$, consider on $M\times G$ the $\Lieg$-valued form
			\[
			\widetilde \omega_\lambda = \pi_G^* \omega_G - \pi_M^* \omega_\lambda \in  \Omega^1(M\times G, \Lieg).
			\]

			If $\iota_\lambda\colon M\times G\to \Lambda \times M\times G$ is the inclusion $(x,g)\mapsto (\lambda, x, g)$, then \[
			D^0= \bigcup_{\lambda\in \Lambda} (\iota_\lambda)_*\big( Ker (\widetilde\omega_\lambda) \big).
			\]
			
			For all $X, Y\in \Gamma(Ker({\widetilde\omega_\lambda}))$, one has
			\[
			d \widetilde \omega_\lambda(X, Y)= X(\widetilde \omega_\lambda (Y))- Y(\widetilde \omega_\lambda (X))- \widetilde \omega_\lambda ([X,Y])=-\widetilde \omega_\lambda ([X,Y]),
			\]therefore
			\begin{align*}
			d\widetilde \omega_\lambda (X,Y)&= d \omega_G ({\pi_G}_* X,{\pi_G}_* Y)- d\omega_\lambda ({\pi_M}_* X, {\pi_M}_* Y)=\\
			&= -[\omega_G({\pi_G}_* X),\omega_G({\pi_G}_* Y)]+[\omega_\lambda({\pi_M}_* X),\omega_\lambda({\pi_M}_* Y)]=\\
			&= -[\omega_G({\pi_G}_* X),\omega_G({\pi_G}_* Y)]+[\omega_G({\pi_G}_* X),\omega_G({\pi_G}_* Y)]=0
			\end{align*}
			so $[X,Y]\in Ker (\widetilde \omega_\lambda)$: by Frobenius Theorem, the distribution $Ker(\widetilde\omega_\lambda)$ is integrable, call  $\mathcal F_\lambda$ the integral foliation.
			As a result, $D^0$ is integrable with integral foliation
			\[
			\bigcup_{\lambda\in \Lambda} \iota_\lambda (\mathcal F_\lambda).
			\]

			\item[Step 2.] The distribution $D$ is integrable. 
			
			Recall that $D=T\Lambda \oplus D^0$ and that $D^0$ is integrable. Then, clearly, for all $X_1, Y_1\in \Gamma(T\Lambda)$ and $X_2,Y_2\in \Gamma(D^0)$, 
			\[
			[(X_1,X_2), (Y_1,Y_2)]=([X_1,Y_1], [X_2,Y_2]) \in \Gamma(T\Lambda) \oplus \Gamma(D^0)=\Gamma(D)
			\]
			and the integrability of $D$ follows by Frobenius Theorem. Denote with $\mathcal F$ the integral foliation. 
			
			\item[Step 3] Each leaf of $\mathcal F$ projects diffeomorphically onto $\Lambda\times M$.
			
			Since the distribution $D$ on $\Lambda\times M\times G$ is invariant by the left action of $G$, $G$ has a well-posed left action on the foliation, namely, given a maximal leaf $\Sigma$ passing by the point $(\lambda_0, x_0,g_0)$, $L_g(\Sigma)=g \cdot \Sigma$ is the maximal leaf passing by the point $(\lambda_0, x_0, g g_0)$.

			Since $\omega_G(\dot g)=0$ if and only if $\dot g=0$, the fibers of $\pi_{\Lambda\times M}$ are transverse to the distribution $D$, hence transverse to any maximal leaf of $\mathcal F$. 
			As a result, for each leaf $\Sigma\in \mathcal F$, the projection map $(\pi_{\Lambda\times M})_{|\Sigma}\colon \Sigma \to \Lambda \times M$ is a local diffeomorphism. 
			
			We show that it is also a proper map. Assume $\{(\lambda_n, x_n, g_n)\}_n\subset \Sigma$ is a sequence such that $(\lambda_n, x_n)\in \Lambda\times M$ converges to some $(x,\lambda)$ in $\Lambda \times M$, then, for any choice of $\bar g\in G$, the maximal leaf $\hat\Sigma$ passing by $(\lambda, x, \bar g)$ projects via a local diffeomorphism to $\Lambda\times M$: so, definetely for $n$, there exists $\hat g$ such that $(\lambda_n, x_n, \hat g g_n)\in \Sigma'$, therefore $\Sigma= \hat g^{-1} \hat \Sigma$ and the sequence $(\lambda_n, x_n, g_n)\in \Sigma$ converges in $\Sigma$ to $(\lambda, x, \hat g ^{-1} \bar g)$.
			
			Since $(\pi_{\Lambda \times M})_{|\Sigma}\colon \Sigma \to \Lambda \times M$ is a local diffeomorphism and a proper map, it is a covering map, hence a diffeomorphism because $M$ is simply connected. 
		\end{itemize}
		
		For any fixed leaf $\Sigma\in \mathcal F$, one can finally conclude that the map
		\[
		\Phi= \pi_G \circ (\pi_{\Lambda\times M})_{|\Sigma}^{-1}\colon \Lambda\times M \to G
		\]
		is smooth and such that $\Sigma=Graph(\Phi)$, therefore  \[\mathcal F=
		\{g\cdot \Sigma\}_{g\in G}=\{g\cdot Graph(\Phi) \}_{g\in G}= \{Graph(L_g \circ \Phi) \}_{g\in G}. \]
		As a result, for all $(\dot \lambda, \dot x)\in T_{(\lambda, x)}\Lambda\times M$, one has that 
		\[
		\omega_\lambda (\dot x)= \omega_G( \Phi_* (\dot \lambda,\dot x) ),
		\]
		in particular, \[\omega_\lambda (\dot x)= \omega_G (\Phi_* (0, \dot x) )= \omega_G(\Phi(\lambda, \cdot)_* (\dot x) ).\]
		\vspace{5mm}
		
		We prove uniqueness. 
		
		Assume $\Phi_1, \Phi_2\colon \Lambda\times M\to G$ both satisfy 
		\[
		(\Phi_1(\lambda, \cdot))^*\omega_G= (\Phi_2(\lambda, \cdot))^*\omega_\lambda
		\]
		for all $\lambda\in \Lambda$, i.e.
		\[
		\omega_\lambda(\dot x)= \omega_G( (\Phi_1 (\lambda, \cdot))_*\dot x)= \omega_G( (\Phi_2 (\lambda, \cdot))_*\dot x).
		\]
		Then, the graphs of $\Phi_1(\lambda, \cdot)$ and $\Phi_2 (\lambda, \cdot)$ are both integral manifolds for the distribution $D^0$. As a result, since $D^0$ is left invariant, for each $\lambda\in \Lambda$ there exists $\psi(\lambda)\in G$ such that
		\[
		\Phi_2(\lambda,x)=\psi(\lambda)\cdot \Phi_1(\lambda,x)
		\]
		for all $x\in M$. A posteriori, $\psi$ is smooth.
	\end{proof}

	With the same notations as in the previous theorem, let $\{\omega_\lambda \}_{\lambda\in \Lambda}\subset \Omega^1(M, \Lieg)$ be a smooth family of $\Lieg$-valued 1-forms, fix $x_0\in M$ and let $e\in G$ be the unity. Let $\Phi\colon \Lambda\times M\to G$ such that
	\begin{equation}
	\label{Phi integra omega base}
		\begin{cases}
	&\omega_\lambda= \Phi(\lambda, \cdot )^* \omega_G \\
	&\Phi(\lambda, x_0)=e
	\end{cases}.
	\end{equation}
	for all $\lambda$.
	
	Denote $\Phi_\lambda=\Phi(\lambda, \cdot)\colon M\to G$.
	
	The following technical Lemma will be useful to compute the derivatives of $\Phi$ with respect to the parameter $\lambda$.
	
	\begin{Lemma}
		\label {Lemma CP}
		With the above notation, assume (\ref{Phi integra omega base}) holds. Let $\lambda_0\in \Lambda$ and $\dot \lambda\in T_{\lambda_0} \Lambda$.
		
		The Cauchy problem 
		\begin{equation}
		\label{Problema di Cauchy chiave}
		\begin{cases}
		&f_*(\cdot) = Ad(\Phi_{ \lambda_0} )\circ  \partial_{\dot \lambda} \big(\omega(\cdot)\big) \colon TM\to \Lieg \\
		&f(x_0)= 0
		\end{cases}
		\end{equation}
		with unknown quantity $f \colon M \to \Lieg$ has  
		\[f(x)= \partial_ {\dot \lambda} \big(\Phi(\lambda,x) \cdot (\Phi(\lambda_0,x))^{-1} \big)\]
		as unique solution.
	\end{Lemma}
		\begin{proof}Assume $f_1$ and $f_2$ are both solutions to $(\ref{Problema di Cauchy chiave})$, then $(f_1-f_2)_*=0$, hence $f_1-f_2$ is a constant, and $f_1(x_0)=f_2(x_0)$, hence $f_1\equiv f_2$. This proves uniqueness.

		Define $ \xi:\Lambda\times M\to G $ as
		\begin{equation}
		\label{def Sigma}\xi(\lambda, x)=\xi_\lambda(x)=\Phi(\lambda, x)(\Phi(\lambda_0, x))^{-1}.
		\end{equation}
		
		In particular,  $\xi(\lambda_0,\cdot)= \xi(\cdot, x_0)\equiv e$.
		
		We need to prove that $\partial_{\dot \lambda} \xi\colon M\to \Lieg$ satisfies (\ref{Problema di Cauchy chiave}). Clearly, $\partial_{\dot \lambda} \xi (x_0)=0$.

		Differentiating
			\[
		\Phi_\lambda(x)= \xi_\lambda(x) \cdot \Phi_{\lambda_0}(x).
		\]
		 one gets
		\[
		(\Phi_\lambda)_{*x}= (L_{\xi(\lambda, x)})_{*} \circ (\Phi_{\lambda_0})_{*x} + (R_{\Phi_{\lambda_0}(x)})_*\circ (\xi_\lambda)_{*x}.
		\]
		Recalling that $\omega_G$ is left-invariant, the $\omega_\lambda$'s satisfy
		\begin{equation}
		\label{eq 5.2}
		\begin{split}
		\omega_\lambda-\omega_{\lambda_0}= \Phi_\lambda^* (\omega_G) -\omega_{\lambda_0}&=  \Phi_{\lambda_0}^*(\omega_G) + (\xi_\lambda)^* \Big( Ad(\Phi_{\lambda_0}(x) ^{-1}) (\omega_G) \Big)-\omega_{\lambda_0}=\\
		&= \Phi_{\lambda_0}^*(\omega_G) +  Ad(\Phi_{\lambda_0}(x) ^{-1})\Big( (\xi_\lambda)^*  (\omega_G) \Big)-\omega_{\lambda_0}=\\
		&=Ad(\Phi_{\lambda_0}(x) ^{-1})\Big( (\xi_\lambda)^*  (\omega_G) \Big).
		\end{split}
		\end{equation}
		We will deduce the first equation of $(\ref{Problema di Cauchy chiave})$ with $f=\partial_{\dot \lambda} \xi$ by differentiating equation $(\ref{eq 5.2})$.
		
		Extend $\xi$ to
		\begin{align*}\Xi\colon \Lambda\times M\times G&\to M\times G\\
		(\lambda, x, g)&\mapsto (x,\  g\cdot \xi(\lambda, x) ),
		\end{align*}
		so that $\Xi(\lambda, \cdot, \cdot )$ can be seen as a family of diffeomorphisms of $M\times G$ depending on $\lambda$. 
		
		Clearly $(\lambda, x, e)= (x, \xi_\lambda (x))$ and
		\[
		\partial_{\dot \lambda} \Xi\in  \Gamma (T(M\times G))
		\]
		satifies
		\[
		(\partial_{\dot \lambda} \Xi)_{(x,g)}=  (0, (L_g)_*(\partial_{\dot \lambda} \xi)_x  ), 
		\]
		hence, defining $\widetilde \omega:= \pi_G^* (\omega_G)$, 
		\[
		\widetilde \omega( \partial_{\dot \lambda} \Xi)=  \partial_{ \dot \lambda} \xi\,.
		\]
		
		Differentiating equation (\ref{eq 5.2}) and evaluating in the direction ${\dot \lambda}$ one gets
		\begin{equation}
		\label{eq 5.3}
		\partial_{\dot\lambda} \omega = Ad\Big((\Phi_{\lambda_0} (x) )^{-1}\Big) \big(j_e ^*(\mathcal L _{\partial_{\dot  \lambda} \Xi} \ \widetilde \omega) \big)
		\end{equation}
		
		where 
		$j_e\colon M\to M\times G$ is the immersion $j_e(x)=(x,e)$ and $\widetilde \omega:= \pi_G^* (\omega_G)$ is a $\Lieg$-valued 1-form on $M\times G$, $\pi_G$ being the projection $\pi_G\colon M\times G\to G$.
		
		Recalling the properties of the Lie derivative with respect to differentiation and contraction, for all $\dot x\in TM$,
		\begin{align*}
		j_e^* \big(\mathcal L _{\partial_{\dot  \lambda}\Xi} \widetilde \omega \big) (\dot x) &= \Big(\mathcal L_{\partial_{\dot  \lambda}\Xi}\ \widetilde \omega\Big) (\dot x, 0)= \\
		&=d \widetilde \omega (\partial_{\dot  \lambda}\Xi,  (\dot x, 0)) + d\big( \widetilde \omega (\partial_{\dot  \lambda}\Xi) \big)({( \dot x,0)})=\\
		&=d \omega_G (\partial_{\dot \lambda} \xi, 0)+ (d\ \partial_{\dot  \lambda} \xi ) (\dot x)= \\
			&=( \partial_{\dot  \lambda} \xi )_* (\dot x).
			\end{align*}
			From the last equation and $(\ref{eq 5.3})$, the thesis follows.
		\end{proof}
	
	\subsection{Proof of Theorem \ref{Teo dipendenza olomorfa}}

	We are now able to discuss the holomorphic dependence on the immersion data for immersions into $\mathds X_{n+1}$.

	\begin{proof}[Proof of Theorem \ref{Teo dipendenza olomorfa}]
	
	We prove that for all $(\lambda_0, x_0)\in \Lambda \times M$ there exists an open neighbourhood $U\times V$ and a $\sigma\colon U\times V\to \mathds X_{n+1}$ as in the statement of Theorem \ref{Teo dipendenza olomorfa} and that, for any other $\sigma'$ satisfying the thesis like $\sigma$, $\sigma'(\lambda, x)= \phi(\lambda)\circ \sigma(\lambda, x)$ with $\phi\colon U \mapsto Isom(\mathds X_n)$ holomorphic. Then, since  $M\times \Lambda$ is simply connected, the proof of the Theorem follows by the standard analytic continuation method as in the proof of Theorem \ref{Teoremone}.

    Choose the neighbourhood $U\times V$ so that there exists on $V$ a $g_{\lambda_0}$-orthonormal frame $(X_1,\dots X_n)$ and the Gram-Schmidt algorithm applies on $U\times V$ to this frame in order to have for all $\lambda\in U$ a $g_\lambda$-orthonormal frame $(e_{1;\lambda}, \dots, e_{n;\lambda})$ on $V$ in the form 
    \[
    e_{i;\lambda}(x)=\sum_j T_{i}^j(\lambda, x) X_j(x)
    \]
    
		where \[T=(T^j_i)_{i,j}\colon U\times V\to Mat(n, \C)
		\] is such that $T(\cdot, x)$ is holomorphic for all $x\in V$.
		Define $\{\theta_\lambda^i\}_i$ as the dual coframe. Both the $e_{i; \lambda}$'s and the $\theta^i_\lambda$'s are holomorphic in $\lambda$.
		
		We recall the steps in the construction of the immersions from the immersion data, as we showed in section 4, and show the holomorphic dependence is preserved at each step.
		
		\begin{itemize}
			\item From $(g_\lambda, \Psi_\lambda)$, locally construct the form $\omega_\lambda \in \Omega^1(V, \Lieonn)$  
			\[
			\omega_\lambda=
			\begin{pmatrix}
			{\scalebox{2} {$\Theta_\lambda$} }& \begin{matrix}
			-\Psi_\lambda^1 & -i\theta_\lambda^1 \\
			\dots &\dots \\
			-\Psi_\lambda^n & -i\theta_\lambda^n
			\end{matrix} \\
			\begin{matrix}
			\Psi_\lambda^1 & \dots & \Psi_\lambda^n\\
			i\theta_\lambda^1 & \dots & i\theta_\lambda^n
			\end{matrix} & 
			\begin{matrix}0\text{    } & 0\text{    } \\ 
			0\text{    } & 0\text{    }
			\end{matrix}
			\end{pmatrix}  
			\]
			where $(\theta^i_\lambda)_{i=1}^n$ is the dual of $(e_{i, \lambda})_{i=1}^n$, $\Theta_\lambda=(\theta^i_{j, \lambda})_{i,j=1}^n$ is the matrix of the Levi-Civita connection forms w.r.t. the $\theta_\lambda^i$'s and $\Psi_\lambda= \Psi_\lambda^i \otimes e_{i, \lambda}$.
			
			Let us show that the map $\lambda\mapsto\omega_\lambda$ is holomorphic in $\lambda$. One can easily check in coordinate charts that the $d\theta^i_\lambda$'s are holomorphic in $\lambda$; defining 
			\[
			d\theta^i=: \alpha^i_{j, k; \lambda} \theta^j \wedge \theta^k,
			\]
			the coefficients $\alpha^i_{j, k; \lambda}$'s are holomorphic in $\lambda$, hence after checking that 
			\[
			\theta^i_{k;\lambda} := -\alpha^i_{j,k; \lambda} \theta^j
			\]
			one concludes that $\Theta_\lambda$ is holomorphic in $\lambda$; finally, the $\Psi_\lambda^i$'s are holomorphic in $\lambda$ since $\Psi_\lambda= \Psi_\lambda^i\otimes e_{i;\lambda}$.
			
			\item Up to shrinking $U$ and $V$, assume they are simply connected. By Proposition \ref{Teorema esistenza foliazione}, there exists a smooth map $\Phi\colon U\times V\to \SOnn$ such that
			$\Phi(\cdot, x_0)\equiv e$ and $\Phi(\lambda,\cdot)^*\omega_G= \omega_\lambda$ where $\omega_G$ is the Maurer-Cartan form of $G$. 
			
			We prove that $\Phi$ is holomorphic w.r.t. $\lambda$. 
			
			Define $\xi$ as in Equation $(\ref{def Sigma})$. For all $\dot \lambda \in T_{\lambda_0}\Lambda$, using Lemma $\ref{Lemma CP}$ and the fact that the $\omega_\lambda$'s vary in a holomorphic way in $\lambda$, one has
			\[
			d(\partial_{i \dot \lambda} \xi) = Ad(\Phi_{\lambda_0})\circ \partial_{i\dot \lambda} {\omega} = Ad(\Phi_{\lambda_0})\circ \big(i\partial_{\dot \lambda} {\omega} \big)= i d(\partial_{\dot \lambda} \xi).
			\]
			
			As a result, both $\partial_{i \dot \lambda}\xi$ and $i\partial_{ \dot \lambda}\xi$ solve the Cauchy problem (\ref{Problema di Cauchy chiave}) for $i\dot \lambda$, hence they coincide. 
			
			By $(\ref{def Sigma})$, one has that for all $x\in V$ and $\lambda\in U$
			\begin{align*}
			(\partial_{i\dot \lambda}\Phi) (\lambda_0, x) &= (L_{\Phi(\lambda_0,x)})_*( (\partial_{i\dot \lambda}\xi)(\lambda_0,x)) = (L_{\Phi(\lambda_0,x)})_*( i(\partial_{\dot \lambda} \xi)(\lambda_0,x))  = \\
			&=i (\partial_{\dot \lambda}\Phi) (\lambda_0, x).
			\end{align*} 
			
			\item By the proof of Theorem \ref{Teoremone}, defining the map
			\[
			\sigma\colon 
			U\times V\to \mathds X_{n+1}
			\]
			by \[ \sigma (\lambda, x)= \Phi(\lambda,x) \cdot e, \]
			one has that $\sigma(\lambda, \cdot):V \to \mathds X_{n+1}$ is an immersion with immersion data $(g_\lambda, \Psi_\lambda)$.  Finally, $\sigma$ is holomorphic w.r.t. $\lambda$ since $\Phi$ is.
			\item By Theorem $\ref{Teoremone}$, if $\sigma'$ satisfies the statement of Theorem \ref{Teo dipendenza olomorfa} like $\sigma$, then $\sigma'(\lambda,x)= \phi(\lambda)\cdot \sigma(\lambda,x)$ for a unique function $\phi\colon U\to Isom_0(\mathds X_n)\cong SO(n+2,\C)$. 
			
			We prove that $\phi$ is holomorphic.
			
			Fix $x\in V$ and a basis $X_1,\dots, X_n\in T_x V$. Then, the functions $M, M'\colon U\to GL(n+2, \C)$ defined by
			\begin{equation*}
			\begin{split}
			M(\lambda)&= \begin{pmatrix} \sigma(\lambda, \cdot)_* (X_1) & \dots & \sigma(\lambda, \cdot)_* (X_n) & \nu^{\sigma(\lambda, \cdot)}(x) & \sigma(x) \end{pmatrix} \qquad \text{and}\\ M'(\lambda)&= \begin{pmatrix} \sigma'(\lambda, \cdot)_{*} (X_1) & \dots & \sigma'(\lambda, \cdot)_{*} (X_n) & \nu^{\sigma'(\lambda, \cdot)}(x) & \sigma'(x) \end{pmatrix} 
			\end{split}
			\end{equation*}
			are holomorphic with respect to $\lambda$: this is an elementary consequence of the assumptions on $\sigma$ and $\sigma'$ and of the fact that vector fields on $U$ and on $V$ commute with each other when seen as vector fields on $U\times V$. By definition of $\phi$ and by the chain rule, $\phi(\lambda)= M'(\lambda) (M(\lambda))^{-1}$, hence $\phi$ is holomorphic.

	\end{itemize}
		
		By the analytic continuation argument mentioned above, one gets the existence of a global equivariant $\sigma\colon M\times \Lambda\to \mathds X_{n+1}$ as in the statement.
		
		It is now simple to check that the monodromy of $\sigma(\lambda, \cdot)$ is holomorphic with respect to $\lambda$. Indeed,
		for all $\gamma\in \pi_1(M)$, the map 
		\[mon_\gamma \colon \Lambda\to SO(n+2,\C) \]
		is defined by 
		\[\sigma( \lambda, \gamma(x))= mon_\gamma (\lambda)\circ \sigma (\lambda,x):\]
		since both $\sigma$ and $\sigma \circ \gamma$ satisfy the conditions in the statement of the theorem for the data $(g,\Psi)$, by the previous step $mon_\gamma$ is holomorphic in $\lambda$.

	Similarly as in the previous step, the map $mon_\gamma (\lambda)$ is uniquely defined by the image of the point $\sigma(\lambda, x)$ and by the image of vectors of the form $\sigma(\lambda, \cdot)_{*x}(X_i)$ where $(X_1, \dots, X_n)$ is a basis for $T_x M$: this way one gets a description of $mon_\gamma(\lambda)\in SO(n+2,\C)$ 
	\end{proof}

	\section{A Gauss-Bonnet Theorem and a Uniformization Theorem for complex metrics}\label{six}
	
	\subsection{Positive complex metrics}
	
	In this section we suggest an intrinsic study of complex metrics on surfaces. To this aim, we will introduce a natural generalization of the concept of complex structure for surfaces and we will present a uniformization theorem in this setting.
	
	Given a surface $S$, the natural inclusion $TS\hookrightarrow \CTS$ factors to a bundle inclusion
\[
	\begin{tikzcd}
	\Proj_\R (TS) \arrow [rr, hook]\arrow[rd] & & \Proj_{\C} (\CTS) \arrow[ld]\\
	& S & 
	\end{tikzcd}.\]
	 
	Fiberwise, for every $x\in S$, $\Proj_\R (T_x S)$ is mapped homeomorphicly into a circle in $\Proj_\C(\C T_xS)$ whose complementary is the disjoint union of two open discs. The conjugation map on $\C TS$ descends to a bundle isomorphism on $\Proj_\C(\C T S)$ that fixes $\Proj_\R(TS)$ and that swaps the two discs fiberwise. 
	\begin{Definition}	
	\label{Def bicomplex structure}
		A \emph{bicomplex structure} on a surface $S$ is a tensor \textbf{J} $\in \Gamma(\C T^* S \otimes \CTS)$ such that
		\begin{itemize}
			\item $\textbf{J}^2 = \textbf{J} \circ \textbf{J}=-id_{\CTS}$; as a result, $\textbf{J}$ is diagonalizable with eigenvalues $\pm i$ and eigenspaces $V_i (\textbf{J}), V_{-i}(\textbf{J})$ with complex dimension $1$.
			\item for all $x\in S$, the eigenspaces $V_i (\textbf J_x)$ and $V_{-i}(\textbf J_x)$ of $\textbf J_x$ have trivial intersection with $T_x S$ and are such that the points $\Proj_\C(V_{i}(\textbf J_x))$ and $\Proj_\C(V_{-i}(\textbf J_x))$ lie in different connected components of $\Proj_\C(\C T_x S)\setminus \Proj_\C (T_xS)$.
		\end{itemize}
	\end{Definition}

	\begin{Remark}
		Clearly complex structures extend to bicomplex structures. Observe that a bicomplex structure $\textbf{J}$ is not a complex structure for $S$ in general, since it may not restrict to a section of $TS \otimes T^*S$. In other words, $\textbf J(\overline X)\ne \overline {\textbf J(X)}$ in general.
		
		Nevertheless, $\textbf{J}$ induces two complex structures $J_1, J_2$ defined by the conditions
		\[
		V_i (\textbf{J})= V_i (J_2) =  \overline{V_{-i} (J_2)} \qquad \text{and} \qquad V_{-i} (\textbf{J})=V_{-i}(J_1)= \overline{V_{i} (J_1)}
		\] 
		which totally characterize $J_1$ and $J_2$.
		
		Also observe that $J_1$ and $J_2$ induce the same orientation. Indeed, representing the set of complex structures on $S$ as $C(S)\sqcup C(\overline S)$, the map 
		\begin{align*}
		C(S)\sqcup C(\overline S)  &\to \Proj_\C (\C T_x S)\setminus \Proj_\R (T_xS) \\
		J_0 &\mapsto \Proj_\C (V_i(J_0)_x)
		\end{align*}
		induces a bijection between the connected components.
		
		As a result, every bicomplex structure induces an orientation. An orientation of $S$ fixed, denote with $BC(S)$ the set of orientation-consistent bicomplex strucures. With the notations above, we therefore have a bijection 
		\begin{equation}
		\label{BC(S) to C(S)xC(S)}
		\begin{split}
		BC(S) &\xrightarrow\sim C(S)\times C(S)\\
		 \mathbf {J} &\mapsto (J_1, J_2). 
		\end{split}
		\end{equation}
		Endow $BC(S)$ with the pull-back topology for which, in particular, it is connected.
	\end{Remark}

	\begin{Definition}
		\label{Def positive metrics}
		Let $g$ be a complex metric on $S$.
		\begin{itemize}
		\item	Denote the set of isotropic vectors of $g$ as
		\[
		ll(g):=\{v\in \C T S \ |\ g(v,v)=0 \}\setminus \{0_S\}
		\]
		and define $ll(g_x):=ll(g) \cap \C T_x S$. Notice that $ll(g_x)$ is the union of two complex lines with trivial intersection.
		\item
		A \emph{positive complex metric} on a surface $S$ is a complex metric $g$ such that the two points $\Proj_\C(ll(g_x))$ lie in different connected components of $\Proj_\C(\C T_x S)\setminus \Proj_\C (T_xS)$.
		\item Denote with $CM^+ (S)$ the space of positive complex metrics on $S$ with the usual $C^{\infty}$-topology for spaces of sections.
				\end{itemize}
	\end{Definition}

	\begin{Proposition}
		Let $g$ be a positive complex metric on a surface $S$.
		\begin{itemize}
			\item[1)] For all $x\in S$ there exist exactly two opposite endomorphisms $\textbf{J}_x, -\textbf{J}_x\in End(\C TS)$ such that 
		\begin{equation}
		\label{compatibility}
			\begin{cases}
		g_x(\textbf{J}_x\cdot,\textbf{J}_x \cdot)= g_x(\cdot, \cdot),\\
		\textbf{J}_x^2=-id_x
			\end{cases}.
			\end{equation}
		The two solutions $\textbf{J}_x, -\textbf{J}_x$ depend smoothly on $x$ and on $g_x$.
			\item[2)] For all $x_0\in S$, there exists a local neighbourhood $U=U(x_0)$ over which $\textbf{J}=\{\textbf{J}_x\}_{x\in U}$ and $-\textbf{J}=\{-\textbf{J}_x\}_{x\in U}$ define two bicomplex structures and
			\[	
			ll(g_x)=V_i(\textbf{J}_x)\cup V_{-i} (\textbf{J}_x).
			\] 
			We will say that $\textbf{J}$ and $-\textbf{J}$ are (locally) \emph{compatible} with $g$.
			\item [3)] If $\textbf{J}$ is compatible with $g$ and $\textbf{J}\equiv (J_1,J_2)$ in correspondence $(\ref{BC(S) to C(S)xC(S)})$, then locally 
			\[
			g= f dw_1 \cdot d\overline w_2
			\]
			where $f$ is a $\C$-valued smooth function and $w_1$, $w_2$ are local holomorphic coordinates for the complex structures $J_1$ and $J_2$ respectively. 
			\item[4)]
			If $\textbf{J}$ is compatible with $g$ on $U$, for any norm-1 local vector field $X$ on $U$, a local $g$-orthonormal frame on $U$ is given by $\{X,\textbf JX\}$. 
		\end{itemize}
	\end{Proposition}
	\begin{proof}
		\begin{itemize}
	\item[1)]  With respect to two local $g$-orthonormal vector fields $X_1$ and $X_2$, $\textbf{J}_x$ is represented by a matrix $B_x\in Mat(2,\C)$ satisfying:
		\[
		\begin{cases}
		B_x^2=-I_2\\
		^tB_x B_x= I_2
		\end{cases},
		\]
		which admits exactly two opposite constant solutions for $B_x$, namely $B_x=\pm \begin{pmatrix} 0 &1\\-1 &0 \end{pmatrix}$, so the two solutions $\textbf J_x$ and $-\textbf J_x$ are smooth.
		\item[2)] The equation comes directly from $(\ref{compatibility})$. Since $g$ is positive, $\Proj_\C(V_{i}(\textbf{J}_x))$ and $\Proj_\C(V_{-i}(\textbf{J}_x))$ lie in different connected components of $\Proj_\C(\C T_x S)\setminus \Proj_\C (T_xS)$, so $\textbf{J}$ and $-\textbf{J}$ are bicomplex structures.
		
		\item[3)] A $\C$-bilinear form on a 2-dimensional $\C-$vector space is determined up to a constant by its isotropic directions. For local holomorphic coordinates $w_1,w_2$ for $J_1$ and $J_2$, we have
		\[
		V_i (\textbf{J})= V_i (J_2)= \C \frac{\partial}{\partial w_2} \qquad V_{-i} (\textbf{J})= V_{-i} (J_1)= \C \frac{\partial}{\partial \overline w_1}, 	
		\]
		hence the locally defined complex metrics $g$ and $ dw_1 \cdot d\overline  w_2= \frac 1 2 (dw_1 \otimes d\overline  w_2+ d\overline  w_2\otimes dw_1)$ are pointwise equal up to a constant since they have the same isotropic directions. Explicitly, we have 
		\[
		g= g(\frac {\partial}{\partial w_1}, \frac {\partial}{\partial \overline w_2} ) dw_1 d\overline w_2
		\]
		\item[4)]  The thesis follows directly by the compatibility of $\textbf{J}$ with $g$.
			\end{itemize}
	\end{proof}

	\begin{Theorem}
		\label{Teorema orientazione}
		Let $g$ be a positive complex metric on a surface $S$. 
		
		The following are equivalent:
		\begin{itemize}
			\item $S$ is orientable;
			\item there exists a global positive bicomplex structure $\textbf{J}$ on $S$ compatible with $g$;
			\item there exists a complex $1$-form $dA\in \Omega^2 (\C TS, \C)$, unique up to a sign, such that $\|dA\|^2_g=1$, i.e. for all local $g-$orthonormal frame $(X_1,X_2)$ one has $\nobreak{dA(X_1,X_2)\in\{-1,1\} }$.
			
			We will call $dA$ the \emph{area form} of $g$.
		\end{itemize}
	\end{Theorem}
	\begin{proof}
		As we showed, a bicomplex structure induces two complex structures which induce an orientation on $S$. Conversely, given an orientation on $S$, locally there exists a unique bicomplex structure compatible with $g$ and inducing the same orientation, hence there exists a unique global one.
		
		If $\textbf J$ is a bicomplex structure compatible with $g$, then $dA=g(\textbf J\cdot, \cdot)$ is a  2-form of norm 1. Conversely, if $\textbf J$ is a local bicomplex structure compatible with $g$, the local $2$-forms $g(\textbf J\cdot, \cdot)$ and $g(-\textbf J\cdot, \cdot)$ are opposite, have module 1 and locally they are the only $2$-forms of $g$-module 1, since $dim_\C\bigwedge^2 \C T_xS=1$: as a result, for a given $dA$, one has $dA=g(\textbf J \cdot, \cdot)$ for a unique global $\textbf J$.

	\end{proof}
	We will say that $\{X, Y \}$ is a \emph{positive orthonormal frame} for $\CTS$ with respect to the orientation induced by $\textbf{J}$ if $\textbf J (X)=Y$. 
	
	With the language of fiber bundles, a positive complex metric on a surface $S$ provides a reduction of the frame bundle of $\CTS \to S$ to the structure group $O(2, \C)$, while the pair of a positive complex metric and an orientation on $S$ provides a reduction to $SO(2,\C)$.

	\begin{Cor}
		Given an oriented surface $S$, the map that sends a positive complex metric $g$ to the orientation-consistent bicomplex structure $\textbf{J}$ compatible with $g$ induces a homeomorphism
		\begin{equation}
		\label{conformal structures and BC}
		\begin{split}
		\faktor{CM^+(S)}{C^{\infty} (S, \C^*)}&\xrightarrow{\sim}   BC(S) \cong C(S)\times C(S)\\
	C^{\infty} (S, \C^*)\cdot g &\mapsto \textbf{J}_g
	\end{split}
	\end{equation}
		Conformal structures of Riemannian metrics correspond to elements of the diagonal of $C(S)\times C(S)$.
	\end{Cor}
	\begin{proof}
	We showed that the map is continuous. In order to construct an inverse, fix a real volume form $\alpha\in \Omega^2(TS)$, which extends by $\C$-bilinearity to an element in $\Omega^2(\CTS,\C)$, and define for any bicomplex structure $\textbf{J}$ the metric
	\[g:= \alpha(\textbf{J}\cdot, \cdot)+ \alpha(\cdot, \textbf{J} \cdot);\]
		such complex metric induces $\textbf{J}$ in turn, hence it is a positive complex metric.
		
		By choosing $\alpha$ as a real volume form for $S$, one gets an explicit correspondence between the diagonal of $C(S)\times C(S)$ and the set of conformal structures of Riemannian metrics.
	\end{proof}
	
	Finally, we remark that Proposition \ref{surfaces in G first part} can be restated in an interesting way for positive complex metrics in terms of developing maps for projective structures. 
	
	\begin{Theorem}
		\label{Totally geodesic immersions}
		Let $S$ be an oriented surface and $g$ a positive complex metric on $S$. Denote with $(\widetilde S, \widetilde g)$ the universal cover. 
		
		$(S,g)$ has constant curvature $-1$ if and only if there exists a smooth map \[(f_1, f_2)\colon \widetilde S \to \GGG=\CP^1\times \CP^1 \setminus\Delta\] such that: 
		\begin{itemize}
			
			\item 
			$\widetilde g=(f_1, f_2)^* \inners=-\frac{4}{(f_1 - f_2)^2} df_1 \cdot df_2$
			\item $f_1$ and $f_2$ are local diffeomorphisms, respectively preserving and reversing the orientation.
			\item $f_1$ and $f_2$ are $(\pi_1 (S), \PSL)$-equivariant local diffeomorphisms with the same monodromy \[mon_{f_1}=mon_{f_2}=mon_{(f_1, f_2)}\colon \pi_1 (S)\to\PSL
			\]
			
		\end{itemize}
	Moreover, if $\textbf{J}$ is the bicomplex structure on $S$ induced by $g$, then, in identification (\ref{BC(S) to C(S)xC(S)}),
	\[
	\textbf{J} \equiv (f_1 ^* (\mathbb{J}^{\CP^1}), f_2 ^* (-\mathbb{J}^{\CP^1}) )
	\] 
	where $\mathbb{J}^{\CP^1}$ is the standard complex structure on $\CP^1$. 
	\begin{proof}
		As we observed in Proposition \ref{surfaces in G first part}, the fact that the complex metric $g$ is positive implies that $f_1$ and $f_2$ are local diffeomorphisms. 
		
		Let $w_k$ be a local holomorphic chart on $x\in S$ for the pull-back complex structure $f_k^* \mathbb{J}^{\CP^1}$. Then, by Cauchy-Riemann equations,
		\[
		Ker(df_k \colon \C T_xS \to \C)= Span_\C \Big(\frac{\partial}{\partial \overline {w_k} }_{|x} \Big)
		\]
		so by the explicit description of the metric one has
		\[ll(g_x)=\bigg(
		Span_\C \Big(\frac{\partial}{\partial \overline {w_1} }_{|x} \Big)\oplus Span_\C \Big(\frac{\partial}{\partial \overline {w_2} }_{|x} \Big) 
		\bigg) \setminus \{0_x\}
		  \]
		and the two points of $\Proj_\C(ll(g_x))$ lie in two different components of $\Proj_\C (\C T_x S)\setminus \Proj_\C (T_xS)$ if and only if the two pull-back complex structures via $f_1$ and $f_2$ induce opposite orientations. 
		
		The choice of $(f_1, f_2)$ inducing the same $\widetilde g$ is unique up to post-composition with elements of $\PSL$, since the composition with the diagonal swap  $s:(x,y)\mapsto(y,x)$ on $\GGG$ gives an immersion with orientation-reversing first component.
		
		All the elements of $\pi_1(S)$ are orientation-preserving, hence $\sigma$ is $(\pi_1(S),\PSL)$-equivariant and, since $\PSL$ acts diagonally on $\GGG$, both $f_1$ and $f_2$ are equivariant with the same holonomy.
		
		From the description of $ll(g)$ above, one can conclude the last part of the theorem.
	\end{proof}
	\end{Theorem}

	\subsection{A Gauss-Bonnet Theorem for  positive complex metrics}
	
	In this subsection, our aim is to prove a generalization of Gauss-Bonnet Theorem in the setting of positive complex metrics.

	Let $g$ be a positive complex metric on an oriented surface. We defined the area form $dA$ of $g$ as the unique 2-form with constant $g$-norm $1$ and compatible with the orientation on $S$, in the sense explained in Theorem \ref{Teorema orientazione}. Locally, 
	\[
	dA= \theta^1 \wedge \theta^2
	\]
	for a local positive $g$-coframe $\{\theta^1,\theta^2\}$.

	\begin{Theorem}
		\label{Gauss Bonnet}
		Let $S$ be an oriented surface and $g$ a complex metric of constant curvature $K$.
		Then \[
		\int_S K dA = 2 \pi \chi(S)
		\]
		where $\chi(S)$ is the Euler-Poincaré characteristic of $S$.
	\end{Theorem}
	\begin{proof}
		Consider a real vector field $\XX\in \Gamma(TS)$ with finite set of zeros $\Lambda=\{x_1, \dots, x_n\}$. Since the complex metric is positive, $\| \XX\|^2\ne 0$ over $S\setminus \Lambda$. Let the index $i$ vary in $\{1, \dots, n\}$. 
		
		Consider for all $i$ two open disks $B_i$ and $\widetilde B_i$ in $S$, such that \[x_i\in B_i \subset \overline B_i \subset \widetilde B_i
		\]
		and such that $B_i \cap B_j\ne \emptyset$ iff $i=j$.
		
		Let $\textbf{J}$ be the bicomplex structure induced by $g$ and by the orientation on $S$. The $g$-frame $\bigg(\frac \XX {\|\XX\|}, \textbf{J} \big(\frac \XX {\|\XX\|}\big) \bigg)$ can be defined locally on $S\setminus \Lambda$ but is defined globally only up to a sign and the same holds for the induced dual frame $\pm(\theta^1, \theta^2)$. Nevertheless, the area form $dA=\theta^1\wedge\theta^2$ is globally well-defined on $S\setminus \Lambda$ and so is the corresponding Levi-Civita connection form $\omega$ defined by 
		\begin{align*}
		d\theta^1 =\omega \wedge \theta^2 \\
		d\theta^2= -\omega \wedge \theta^1.
		\end{align*}
		On the whole $S\setminus \Lambda$ we have $d\omega=-KdA$ as in standard Riemannian Geometry (see \cite{Kobayashi-Nomizu 1}).
		
		On each $\widetilde B_i$ we fix a norm-$1$ complex vector field $e_{i}$ and consider the coframe $(\theta^1_i, \theta^2_i)$ defined as the dual of $(e_i, \textbf{J}(e_i))$; let $\omega_i$ be the Levi-Civita connection form with respect to $(\theta^1_i, \theta^2_i)$. 
		Orient each $\partial {\widetilde B_i}$ with the orientation induced by $\widetilde B_i$. By Stokes formula,
		\begin{align*}                  
		\int_S K dA &= \int_{S\setminus \bigcup_i  B_i} K dA + \sum_{i=1}^n \int_{\overline B_i} KdA= \\
		=& -\int_{S\setminus \bigcup_i \widetilde B_i} d\omega - \sum_{i=1}^n \int_{\overline B_i} d\omega_i=\\
		=& \sum_{i=1}^n \int_{\partial B_i} \omega - \omega_i.
		\end{align*}

		We prove that $\int_S KdA$ is a conformal invariant. 
		
		For all smooth $f\colon S\to \C$, $\overline g= f g$ is a positive complex metric conformal to $g$. For some local choice of a square root of $f$ and of $g(\XX,\XX)$, one can define a local $\overline g$-orthonormal frame $(\frac \XX { \sqrt{f \|\XX,\XX\|_g^2}  }, \frac {\textbf{J}\XX} { \sqrt{f \|\XX,\XX\|_g^2}  })$ with $\overline{g}$-dual given by the coframe $(\overline \theta^1, \overline \theta^2)$, which in turn defines a Levi-Civita connection form $\overline \omega$ for $\overline g$ independent of the choice of the root. On each 
		$\widetilde B_i$, fix a square root $\sqrt f$ of 
		$f$, so a $\overline g$-orthonormal frame is given by 
		$(\frac{e_1} {\sqrt f }, \frac{\textbf{J} (e_1)} {\sqrt f })$ with 
		associated $\overline g$-coframe $(\overline \theta_i^1, \overline \theta_i ^2)$ and Levi-Civita connection form $\overline \omega_i$. By an explicit computation, one can check that on each $\widetilde B_i \setminus \{x_i\} $ 
		\[
		\overline \omega - \omega= \overline \omega_i - \omega_i = -\frac1{2f}df\circ \textbf{J}.
		\]
		As a result, 
		\[
		\int_S K_{\overline g} dA_{\overline g} = \sum_{i=1}^n \int_{\partial B_i} \overline \omega -\overline \omega_i= \sum_{i=1}^n \int_{\partial B_i} \omega -\omega_i= \int_S K_g dA_g.
		\]

		Now, on each $\widetilde B_k \setminus \{x_k\}$, the coframes $(\theta^1_i, \theta^2_i)$ and $(\theta^1, \theta^2)$ (the last one is well-defined only locally) are both $m$-orthonormal and $\textbf{J}$-oriented, so we locally have 
		\[
		\begin{pmatrix}
		\theta^1\\
		\theta^2
		\end{pmatrix}=
		\begin{pmatrix}
		\cos(\alpha_k) & \sin(\alpha_k) \\
		-\sin(\alpha_k) & \cos(\alpha_k)
		\end{pmatrix} \begin{pmatrix}
		\theta_k ^1 \\
		\theta_k^2
		\end{pmatrix}
		\]
		where $e^{i\alpha_k} \in \C^*$ is a locally defined smooth function, defined only up to a sign on $\widetilde B_k \setminus \{x_k\}$. Nevertheless, $e^{2i\alpha_k}\colon \partial B_k \to S^1$ is a well defined smooth function and $d\alpha_k$ is a well-defined one form on $\partial B_k$.
		
		A standard calculation shows that $\omega - \omega_k =d \alpha_k$, hence
		
		\begin{align*}
		\int_S KdA &= \sum_{k=1}^n \int_{\partial B_k} \omega - \omega_k= \sum_{k=1}^n \int_{\partial B_k} d \alpha_k =\\
		&= \sum_{k=1}^n \int_{\partial B_k} \frac 1 {2i} \frac{ d (e^{2i\alpha_k} )}{e^{2i \alpha_k}}= \sum_{k=1}^n \pi deg(e^{2i\alpha_k}) \in \pi \Z.
		\end{align*}

		Recalling that $BC(S)$ is connected and that the quantity $\int_S K_g dA_g$ is a conformal invariant, the map 
		\[
		\int_S K \colon BC(S) \to \pi \Z
		\]
		is well-posed and continuous, hence it is constant. Since $\int_S K dA= 2 \pi \chi(S)$ for Riemannian metrics, the thesis follows.
	\end{proof}

	\subsection{A Uniformization Theorem through Bers Theorem}
	
	The aim of this subsection is to prove a generalization of the classic Riemann's Uniformization Theorem for bicomplex structures and complex metrics.

Before stating the main theorem, we recall some notions about quasi-Fuchsian representations.
	
	A representation $\rho\colon \pi_1(S)\to \PSL$ is \emph{quasi-Fuchsian} if its limit set $\Lambda_\rho\subset \CP^1$, i.e. the set of accumulation points of any $\rho(\pi_1(S))-$orbit in $\CP^1$, is a Jordan curve in $\CP^1$. Since this condition is preserved by the action of $\PSL$, the set of quasi-Fuchsian representations defines an open subset of the character variety	$QF(S)\subset \mathcal X (S)$.
	
	\begin{Theorem} [Bers Simultaneous Uniformization Theorem, \cite{Bers}]
		\label{Bers}
		Let $S$ be an oriented surface with $\chi (S)<0$. For all $J_1, J_2$ complex structures over $S$, there exists a unique quasi-Fuchsian representation $\rho=\rho(J_1,J_2) \colon \pi_1(S) \to \PSL$ such that, defined $\Lambda_\rho\subset \CP^1$ as the limit set of $\rho$ and $\Omega_+, \Omega_-$ as the connected components of $\CP^1 \setminus \Lambda_\rho$, 
		\begin{itemize}
			\item there exists a unique diffeomorphism $f_1 \colon \widetilde S \to \Omega_+$, which is $J_1$-holomorphic and $\rho$-equivariant;
			\item there exists a unique diffeomorphism $f_2 \colon \widetilde S \to \Omega_-$, which is $J_2$-antiholomorphic and $\rho$-equivariant.
		\end{itemize} 
		Moreover, $(\rho, f_1, f_2)$ are continuous functions of $(J_1, J_2)\in BC(S)$.
		
		This correspondence determines a homeomorphism in the quotient
		\begin{equation}
		\label{Bers homeo}
		 \mathfrak B\colon\Tau(S)\times \Tau(S)\xrightarrow{\sim}QF(S)
		\end{equation}
		where $\Tau(S)$ is the Teichm\"uller space of $S$.
	\end{Theorem}
\vspace{5mm}
	
	Define 
\begin{gather*}
\widetilde{\mathcal P}_{QF}(S)= \widetilde{Hol}^{-1}(QF(S)), \\
\mathcal P_{QF} (S)= Hol^{-1} (QF(S) )
\end{gather*}
which, by continuity of the holonomy maps, are open subsets of $\widetilde{\mathcal P}(S)$ and of ${\mathcal P}(S)$. 
	
	\begin{Theorem}[W. Goldman, \cite{Goldman: projective structures}]
		\label{Teo Goldman} 
		Let $S$ be a compact oriented surface of genus $g(S)\ge 2$.
		
		An open connected component of the space $\widetilde {\mathcal P}_{QF} (S)$ is
		\[
	\widetilde {\mathcal P}_0  (S):=\{ f\colon \widetilde S \to \CP^1,  \text{$\rho$-equivariant for some quasi-Fuchsian $\rho\colon \pi_1(S)\to \CP^1$, with } f(\widetilde S)\cap \Lambda_\rho=\emptyset   \}.
		\]
		Moreover, the restriction 
		 \[
		 Hol \colon \faktor{\widetilde {\mathcal P}_0 (S)}{\text{Diff}_0 (S)}=: {\mathcal P}_0 (S) \to QF(S)
		 \]
	is a homeomorphism. 
	
	With the notations of Theorem \ref{Bers} and equation \eqref{Bers homeo}, the inverse is given by \[[\rho]\mapsto [f_1 (\rho)]=\Big[f_1 \Big(\mathfrak B^{-1} ([\rho])\Big)\Big].
\]	

\end{Theorem}

	We are now able to state a generalization of the Uniformization Theorem.
	
		For an oriented surface $S$, define 
	\[CM^+_{-1} (S):=\{\text{Positive complex metrics with constant curvature $-1$} \}\]
	endowed with the subspace topology from $CM^+ (S)$.

	\begin{Theorem} 
		\label{Uniformization 1}
		 
		Let $S$ be an oriented surface with $\chi(S)<0$. 
		\begin{itemize}
			\item[1)] 
			For all $g\in CM^+(S)$, there exists a smooth $f\colon S\to \C^*$ such that $f\cdot g$ has constant curvature $-1$ and quasi-Fuchsian holonomy.
			
			More precisely, the map
			\begin{align*}
			\mathfrak{{J}} \colon C M^+(S) &\to BC (S) \\
			g &\mapsto [g]=\textbf J_g
			\end{align*}
			admits a continuous right inverse 
			\[
			\mathfrak{U}\colon  C(S)\times C(S)\cong BC(S) \to C M^+_{-1} (S)
			\]
			such that the diagram 
				\[
			\begin{tikzcd}
			C(S)\times C(S)\cong BC(S) \arrow[r, "\mathfrak U"] \arrow[d] &CM_{-1}^+(S) \arrow[d, "mon"]  \\
			\Tau(S)\times \Tau(S) \arrow[r, "\mathfrak B"] & \mathcal X(S)
			\end{tikzcd}
			\]
			commutes, where $\mathfrak B$ is defined as in $(\ref{Bers homeo})$.

			\item[2)]	If $S$ is closed, the image of $\mathfrak U$ is the connected component of $C M^+_{-1} (S)$ containing Riemannian metrics.
		\end{itemize}
	\end{Theorem}

	\begin{proof} [Proof of Theorem $\ref{Uniformization 1}$]
		\begin{itemize}
			\item[1)]
		Consider $\textbf{J}\in BC(S)$, which corresponds to two complex structures $(J_1, J_2)\in C(S)\times C(S)$.
		
		Applying Theorem $\ref{Bers}$ to the couple $(J_2, J_1)$, there exist:
		\begin{itemize}
			\item $\rho=\rho(J_1,J_2)\colon \pi_1 (S)\to \PSL$ quasi-Fuchsian;
			\item $f_1 \colon \widetilde S \to \CP^1$ $J_1$-antiholomorphic and $\rho$-equivariant embedding;
			\item $f_2 \colon \widetilde S \to \CP^1$ $J_2$-holomorphic and $\rho$-equivariant embedding
		\end{itemize}
		with $Im(f_1)\cap Im(f_2)=\emptyset$. Hence, we have an admissible, $\rho$-equivariant embedding
		\[
		\sigma=(f_1, f_2) \colon \widetilde S \to {\GGG}= \CP^1 \times \CP^1 \setminus \Delta.
		\]
		By Theorem \ref{Totally geodesic immersions}, the pull-back complex metric on $\widetilde S$ is positive and projects to a positive complex metric $g$ on $S$. 
		
		After we show that the constructed $g$ is compatible with $\textbf{J}$, the thesis follows by defining $\mathfrak{U}(\textbf{J}):=g$. Let $w_1$ and $w_2$ be the complex coordinates on $\widetilde S$ induced by $J_1$ and $J_2$ respectively. Then,
		\begin{align*}
	g &= \frac{4}{(f_1  - f_2 )^2} f_1^* dz_1 \cdot f_2^* dz_2= \frac{4}{(f_1 - f_2 )^2}  d f_1 \cdot d f_2 = \\
		&= \frac{4}{(f_1  - f_2 )^2}  \frac{\partial f_1}{\partial \overline {w_1}} \frac{\partial f_2}{\partial w_2}  d\overline {w_1} \cdot d w_2;
		\end{align*}
		hence $V_i(J_1)=V_i(\textbf{J})$ and $V_{-i} (J_2)=V_{-i}(\textbf{J})$ are the  isotropic directions of $g$.
\item[2)] 
Since one can see
\[
Im(\mathfrak U)=\{g\ |\ (\mathfrak U \circ \mathfrak J) (g)=g \}\subset C M_{-1}^+(S),
\]
$Im(\mathfrak U)$ is a closed connected subset of $\C M_{-1}^+ (S)$. We need to prove that it is also open.

		As we stated in Theorem \ref{Totally geodesic immersions}, for any $g\in CM_{-1}^+$ there exist two developing maps for projective structures $f_1$ and $f_2$ on $S$ and $\overline S$ respectively, uniquely defined up to post-composition with elements in $\PSL$, such that $\sigma=(f_1,f_2)\colon (\widetilde S, \widetilde g)\to \GGG$ is an isometric immersion. We therefore have a map 
		\begin{align*}
		proj\colon CM_{-1} ^+ (S)&\to \mathcal {\widetilde P}(S)\times \mathcal{\widetilde  P}(\overline S)\\
		 g=(f_1,f_2)^* \inners_\GGG &\mapsto ([f_1], [f_2]).
		\end{align*}
		
		The thesis follows after we show that \[Im (\mathfrak U )= proj^{-1}  (\mathcal{\widetilde P}_0(S)\times \mathcal{\widetilde  P}_0(\overline S) )\] since $\mathcal{\widetilde P}_0(S)$ is an open subset of $\mathcal{\widetilde P}_{QF}(S)$, hence of $\mathcal{\widetilde P}(S)$.
		
		By construction of $\mathfrak U$, clearly \[Im (\mathfrak U )\subseteq proj^{-1}  (\mathcal{\widetilde P}_0(S)\times \mathcal{\widetilde  P}_0(\overline S) ).\]
		
		Conversely, assume $g=(f_1,f_2)^* \inners_\GGG \in proj^{-1}  (\mathcal{\widetilde P}_0(S)\times \mathcal{\widetilde  P}_0(\overline S) )$. Since $hol_{f_1}=hol_{f_2}\colon \pi_1(S)\to \PSL$, $f_1$ and $f_2$, by Theorem \ref{Teo Goldman}, the maps $f_1$ and $f_2$ correspond exactly to the maps one gets through Theorem \ref{Bers} from the couple of complex structures $(f_1^* (\mathbb J^{\CP^1}), f_2^*(-\mathbb J^{\CP^1})$, hence 
		\[
		(f_1,f_2)^*\inners_\GGG = \mathfrak U (f_1^* (\mathbb J^{\CP^1}), f_2^*(-\mathbb J^{\CP^1}) ). 
		\]

	\end{itemize}
	\end{proof}

\end{document}